\documentclass[12pt, a4paper,oneside]{extarticle}
\usepackage[english]{babel}
\usepackage{amsmath, amscd, amssymb, amsfonts, stmaryrd, vmargin, amsthm, soul, bm, mathrsfs}

\setpapersize[portrait]{A4}
\DeclareMathOperator{\im}{Im}
\DeclareMathOperator{\Aut}{Aut}
\DeclareMathOperator{\Ker}{Ker}
\DeclareMathOperator{\rank}{rank}
\theoremstyle{plain}        \newtheorem{thm}{Theorem}

\theoremstyle{plain}        \newtheorem{pro}[thm]{Proposition}

\theoremstyle{plain}        \newtheorem{lem}[thm]{Lemma}

\theoremstyle{plain}        \newtheorem{cor}[thm]{Corollary}

\theoremstyle{plain}        

\theoremstyle{plain}     \newtheorem{rem}{Remark}

\begin{document}
\pagestyle{empty}
\title{\begin{large}
\centerline{\bf Asymptotic total geodesy of local holomorphic curves exiting } 
\vskip 0.1cm
\centerline{\bf a bounded symmetric domain and applications to}
\vskip 0.1cm
\centerline{\bf a uniformization problem for algebraic subsets}
\end{large}}
\author{\textbf{Shan Tai Chan and Ngaiming Mok}}
\date{\text{ }}
\maketitle

\pagestyle{myheadings}
\setcounter{page}{1}

\linespread{1}
\begin{abstract}

\noindent
The current article stems from our study on the asymptotic behavior of holomorphic isometric embeddings of the Poincar\'e disk into bounded symmetric domains. As a first result we prove that any holomorphic curve exiting the boundary of a bounded symmetric domain $\Omega$ must necessarily be asymptotically totally geodesic.  Assuming otherwise we derive by the method of {\it rescaling} a hypothetical holomorphic isometric embedding of the Poincar\'e disk with $\text{\rm Aut}(\Omega')$-equivalent tangent spaces into a tube domain $\Omega' \subset \Omega$ and derive a contradiction by means of the Poincar\'e-Lelong equation. We deduce that equivariant holomorphic embeddings between bounded symmetric domains  must be totally geodesic.  Furthermore, we solve a uniformization problem on algebraic subsets $Z \subset \Omega$. More precisely, if $\check \Gamma\subset \text{\rm Aut}(\Omega)$ is a torsion-free discrete subgroup leaving $Z$ invariant such that $Z/\check \Gamma$ is compact, we prove that $Z \subset \Omega$ is totally geodesic. In particular, letting $\Gamma \subset\text{\rm Aut}(\Omega)$ be a torsion-free cocompact lattice, and $\pi: \Omega \to \Omega/\Gamma =: X_\Gamma$ be the uniformization map, a subvariety $Y \subset X_\Gamma$ must be totally geodesic whenever some (and hence any) irreducible component $Z$ of $\pi^{-1}(Y)$ is an algebraic subset of $\Omega$. For cocompact lattices this yields a characterization of totally geodesic subsets of $X_\Gamma$ by means of bi-algebraicity without recourse to the celebrated monodromy result of Andr\'e-Deligne on subvarieties of Shimura varieties, and as such our proof applies to {\it not necessarily arithmetic} cocompact lattices. In place of the monodromy result of Andr\'e-Deligne we exploit the existence theorem of Aubin and Yau on K\"ahler-Einstein metrics for projective manifolds $Y$ satisfying $c_1(Y) < 0$ and make use of Nadel's semisimplicity theorem on automorphism groups of noncompact Galois covers of such manifolds, together with the total geodesy of equivariant holomorphic isometric embeddings between bounded symmetric domains.
\end{abstract}

\linespread{1.04}
\section{Introduction}
For a bounded symmetric domain $\Omega\Subset \mathbb C^N$ in its Harish-Chandra realization, we denote by $ds_\Omega^2$ its Bergman metric.  As a first motivation for the current article, we are interested in the study of holomorphic isometries $f: (\Omega_1,\lambda ds_{\Omega_1}^2) \to (\Omega_2,ds_{\Omega_2}^2)$, $\lambda > 0$, between bounded symmetric domains.  When $\Omega_1$ is irreducible and of rank $\ge 2$, it follows from the proof of Hermitian metric rigidity that $f$ is necessarily totally geodesic (cf.\,Mok \cite{Mo89}, Clozel-Ullmo \cite{CU03}). The interest lies therefore in the cases where $\Omega_1 \cong \mathbb B^n$, $n \ge 1$, is the complex unit ball. By Mok \cite{Mo12}, it follows from the rationality of Bergman kernels of bounded symmetric domains in Harish-Chandra coordinates that any holomorphic isometry $f: (\mathbb B^n,\lambda ds_{\mathbb B^n}^2) \to (\Omega,ds_{\Omega}^2)$, $\lambda > 0$, must necessarily be a proper holomorphic isometric embedding such that ${\rm Graph}(f) \subset \mathbb B^n \times \Omega$ can be analytically continued to an affine algebraic variety $V \subset \mathbb C^n \times \mathbb C^N$. The first objective in the current article is to study the case where $\Omega_1 = \Delta := \mathbb B^1$, and we prove a more general result ascertaining that an arbitrary local holomorphic curve exiting $\Omega$ must necessarily be asymptotically totally geodesic.  More precisely, we have  

\begin{thm}\label{MainThm}
Let $\Omega\Subset \mathbb C^N$ be a
bounded symmetric domain in its Harish-Chandra realization equipped with the Bergman metric $ds_\Omega^2$.
Let $\mu:U=\mathbb B^1(b_0,\epsilon)\to \mathbb C^N$, $\epsilon>0$, be a holomorphic embedding such that $\mu(U\cap\Delta)\subset \Omega$ and $\mu(U\cap\partial\Delta)\subset \partial\Omega$, where $b_0\in \partial\Delta$.
Denote by $\sigma(z)$ the second fundamental form of $\mu(U\cap\Delta)$ in $(\Omega,ds_\Omega^2)$ at $z=\mu(w)$. Then, for a general point $b\in U\cap\partial\Delta$ we have $\lim_{w\in U\cap \Delta,\;w\to b} \lVert \sigma(\mu(w))\rVert = 0$.
\end{thm}
For the last statement we also say for short that $\mu$ is asymptotically totally geodesic at a general point $b \in \partial\Delta$.  
From Mok \cite{Mo12} we deduce readily the asymptotic total geodesy of holomorphically embedded Poincar\'e disks on $\Omega$, as follows. 

\begin{thm}\label{ThmHIAT}Let $f:(\Delta,\lambda ds_\Delta^2)\to (\Omega,ds_\Omega^2)$ be a holomorphic isometric embedding, where $\lambda$ is a positive real constant and $\Omega\Subset \mathbb C^N$ is a bounded symmetric domain in its Harish-Chandra realization.
Then, $f$ is asymptotically totally geodesic at a general point $b\in \partial\Delta$.
\end{thm}
Theorem \ref{ThmHIAT} was stated in \cite[Theorem 3.5.1]{Mo11} where it was indicated that the proof relies on the Poincar\'e-Lelong equation. Then, Mok \cite{Mo14} obtained an elementary proof of the special case of Theorem \ref{MainThm} where the local holomorphic curve exits at a smooth boundary point, i.e., at $p \in {\rm Reg}(\partial\Omega)$, and the write-up of a complete proof of Theorem \ref{ThmHIAT} was delayed in part since the second author was searching for a proof along the lines of argument of \cite{Mo14}.  In joint efforts towards that goal we soon realized that the geometry of local holomorphic curves exiting other strata of $\partial\Omega$ is much more subtle and a proof using {\it the rescaling argument\,} and the Poincar\'e-Lelong equation remains methodologically the most useful for the study of holomorphic isometries. This resulted in the write-up of the proof of Theorem \ref{MainThm} in the current article and a substantial new application of Theorem \ref{ThmHIAT} to a uniformization problem on bounded symmetric domains arising from functional transcendence theory given in Section \ref{Sec:5}. 

The rescaling argument, which was discovered by Wong \cite{Won77} and applied to characterize the complex unit ball as the strictly pseudoconvex domain with smooth boundary, unique up to biholomorphic equivalence, admitting an infinite number of automorphisms, is currently made use of in the study of uniformization problems related to the hyperbolic Ax-Lindemann conjecture.
The latter conjecture, which asserts that for a torsion-free arithmetic lattice $\Gamma \subset {\rm Aut}(\Omega)$, the Zariski closure of the image of an algebraic subset $S \subset \Omega$ under the uniformization map $\pi: \Omega \to \Omega/\Gamma =: X_\Gamma$ is necessarily totally geodesic, has been established by Klingler-Ullmo-Yafaev \cite{KUY16} (after Ullmo-Yafaev \cite{UY14} in the compact case and Pila-Tsimerman \cite{PT14} in the Siegel modular case) using methods of o-minimality in model theory in combination with methods from Hodge theory and complex differential geometry.  However, when the arithmeticity assumption on the lattice $\Gamma$ is dropped, it has so far not been possible to adapt the methods of the aforementioned articles to the problem.  In the rank-1 case, the approach of Mok \cite{Mo10} \cite{Mo18} using methods from several complex variables, algebraic geometry and complex differential geometry has yielded a resolution of the analogous conjecture in the affirmative for not necessarily arithmetic lattices, and a key point of the method is the rescaling argument applied to an irreducible component $Z$ of the preimage of the Zariski closure $\overline{\pi(S)}^{\mathscr Z\!ar}$ 
with respect to the uniformization map $\pi: \Omega \to X_\Gamma$.

Here for bounded symmetric domains $\Omega$ of arbitrary rank we give for the first time a geometric application of the rescaling argument in the proof of Theorem \ref{MainThm} for local holomorphic curves $C$  exiting $\partial\Omega$. We do this by pulling back $C$ by a divergent sequence of automorphisms of $\Omega$ to yield by taking limits of subvarieties the image of a holomorphic isometric embedding of the Poincar\'e disk.  Arguing by contradiction, in the event that Theorem \ref{MainThm} fails, by rescaling we construct holomorphically embedded Poincar\'e disks $Z$ which are closed to being homogeneous, e.g., the norm of the second fundamental form of $Z \subset \Omega$ can be made to be a nonzero constant, in order to derive a contradiction by means of the Poincar\'e-Lelong equation.  The latter equation was applied in Mok \cite{Mo02} for the characterization of totally geodesic holomorphic curves in the case of tube domains.  Exploiting the parallelism of the curvature tensor on bounded symmetric domains and estimates of the Kobayashi metric and the Kobayashi distance on bounded convex domains (cf.\,Mercer \cite{Me93}) we solve in this article a new type of integrability problem by sandwiching a tube domain between $Z$ and $\Omega$, thereby allowing us to apply the Poincar\'e-Lelong equation.  

It should be noted that, in view of the construction in Mok \cite{Mo16} of nonstandard holomorphic isometric embeddings of the complex unit ball $\mathbb B^n$ into an irreducible bounded symmetric domain of rank $\ge 2$ by means of varieties of minimal rational tangents, the analogue of Theorem \ref{MainThm} fails in general when local holomorphic curves are replaced by local complex submanifolds of dimension $\ge 2$.  But precisely Theorem \ref{MainThm} as it stands is enough to imply that any equivariant holomorphic embedding between bounded symmetric domains must be totally geodesic (cf.\,Theorem \ref{ThmEquiv}), a result which in the cases of classical domains was due to Clozel \cite{Cl07}, and we will make use of the result to give an application of Theorem \ref{MainThm} to a uniformization problem on algebraic subsets of bounded symmetric domains which is a first step towards an affirmative resolution of the Ax-Lindemann conjecture for not necessarily arithmetic lattices in the (locally reducible) higher rank case.  
 
\begin{thm}\label{ThmBialg}
Let $\Omega \Subset \mathbb C^N$ be a bounded symmetric domain in its Harish-Chandra realization, $\Gamma \subset \mbox{\rm Aut}(\Omega)$ be a not necessarily arithmetic torsion-free cocompact lattice. Write $X_\Gamma := \Omega/\Gamma$, $\pi: \Omega \to X_\Gamma$ for the uniformization map. Let $Y \subset X_\Gamma$ be an irreducible subvariety, and $Z \subset \Omega$ be an irreducible component of $\pi^{-1}(Y)$.  Suppose $Z \subset\Omega$ is an algebraic subset. Then, $Z \subset \Omega$ is a totally geodesic complex submanifold.
\end{thm}

The analogue of Theorem \ref{ThmBialg} in the case of arithmetic and not necessarily cocompact lattices $\Gamma \subset {\rm Aut}(\Omega)$ was established by Ullmo-Yafaev \cite{UY11}, and that gives the characterization of totally geodesic subsets of $X_\Gamma$ as the unique bi-algebraic subvarieties, thus yielding a reduction of the hyperbolic Ax-Lindemann conjecture.  The proof of \cite{UY11} relies heavily on the result of Andr\'e-Deligne \cite{An92} ascertaining the Zariski density of the monodromy representation of the fundamental group of an algebraic subvariety of $X_\Gamma$ unless it is contained in a proper totally geodesic algebraic subvariety, a deep result which relies on Hodge theory and which is in general not available for nonarithmetic lattices.  In its place we will deduce Theorem \ref{ThmBialg} from  Theorem \ref{ThmEquiv}, by a proof which relies in the first place on the semisimplicity theorem of Nadel \cite{Na90} on the automorphism groups of universal covering spaces of compact complex manifolds with ample canonical line bundle.  

Using the maximum principle for plurisubharmonic functions and an adaptation of \cite{Na90} we deduce that $Z \subset \Omega$ is a nonsingular complex submanifold, and the identity component $H_0$ of the subgroup of ${\rm Aut}_0(\Omega)$ which stabilizes $Z$ is a positive-dimensional real semisimple Lie group without compact factors.
Considering any $H_0$-orbit $S$ in $Z$, the proof will have followed from Theorem \ref{ThmEquiv} if we can show that $S=Z \subset \Omega$ is a Hermitian symmetric space of the semisimple and noncompact type.
We prove that this is indeed the case by means of cohomological arguments using the compactness of $Y\subset X_\Gamma$.
Using such arguments, we show on the one hand that an $H_0$-orbit $S$ in $Z$ must have real dimension equal to $\dim_{\mathbb R}(Z)$ and thus $S=Z\subset \Omega$ is complex-analytic, and on the other hand that $S\cong H_0/L$ for some maximal compact subgroup $L$ of $H_0$.
As a consequence, $S=Z\subset \Omega$ is a Hermitian symmetric space of the semisimple and noncompact type, and Theorem \ref{ThmEquiv} applies to yield Theorem \ref{ThmBialg}.
Our arguments actually yield a more general result (Theorem \ref{ThmUnif}) in which $\Gamma\subset{\rm Aut}(\Omega)$ is any torsion-free discrete subgroup and $Y\subset X_\Gamma=\Omega/\Gamma$ is a compact complex-analytic subvariety.

For the link between the study of holomorphic isometries and uniformization problems we refer the reader to the expository article Mok \cite{Mo18}.   

\noindent\textbf{Acknowledgment.}
Research work done by the second author for the current article has been supported by GRF grant 17301518 of the Hong Kong Research Grants Council.
The authors would like to thank the referee for helpful suggestions.

\section{Preliminaries}
Let $\Omega\Subset \mathbb C^N$ be an irreducible bounded symmetric domain of rank $r$.
We may identify $\Omega\cong G_0/K=:X_0$ as a Hermitian symmetric space $X_0$ of the noncompact type, where $G_0=\Aut_0(\Omega)$ and $K\subset G_0$ is the isotropy subgroup at ${\bf 0}\in \Omega$ (cf.\,\cite{Wol72}, \cite{Mo14}).
We follow some basic terminology introduced in \cite{Wol72} (cf.\,\cite{Mo89}, \cite{Mo14}).
Let $G$ be the complexification of $G_0$ and $\mathfrak g$ be the complex Lie algebra of $G$. 
Let $\mathfrak g_0\subset \mathfrak g$ be the real Lie algebra of $G_0$, which is a noncompact real form of $\mathfrak g$, and $\mathfrak k\subset \mathfrak g_0$ be the Lie algebra of $K$.
Fixing a Cartan subalgebra $\mathfrak h$ of $\mathfrak k$, the complexification $\mathfrak h^{\mathbb C}$ of $\mathfrak h$ lies in the complexification $\mathfrak k^{\mathbb C}$ of $\mathfrak k$. Then, $\mathfrak h^{\mathbb C}\subset \mathfrak g$ is also a Cartan subalgebra of $\mathfrak g$, and the set of all roots of $\mathfrak g$ lies in $\sqrt{-1}\mathfrak h^*$.
Let $\Delta_M^+$ (resp.\,$\Delta_M^-$) be the set of noncompact positive (resp.\,negative) roots as a subset of the set of all roots of $\mathfrak g$. Then, we have $\mathfrak m^+ = \bigoplus_{\varphi \in \Delta_M^+} \mathbb C e_\varphi$ and $\mathfrak g_\varphi = \mathbb C e_\varphi$ with $e_\varphi$ being of unit length with respect to the canonical K\"ahler-Einstein metric $h$.
Moreover, we have $\mathfrak m^-=\bigoplus_{\varphi \in \Delta_M^-} \mathbb C e_\varphi$ and the compact dual Hermitian symmetric space $X_c=G/P$ of $X_0$, where $P$ is the parabolic subgroup of $G$ corresponding to the parabolic subalgebra $\mathfrak p:=\mathfrak k^{\mathbb C}\varoplus\mathfrak m^-$.
We let $\Psi=\{\psi_1,\ldots,\psi_r\}$ be a maximal strongly orthogonal set of noncompact positive roots.
From the Polydisk Theorem (cf.\,\cite{Wol72}, \cite{Mo14}), there is a maximal polydisk $\Delta^r \cong \Pi\subset \Omega$ given by
$\Pi=\left(\bigoplus_{j=1}^r \mathfrak g_{\psi_j}\right)\cap \Omega$ such that $(\Pi,h|_\Pi)\subset (\Omega,h)$ is totally geodesic and $\Omega = \bigcup_{\gamma\in K} \gamma\cdot \Pi$.

For $\Lambda\subset \Psi$ we let $\mathfrak g_\Lambda=[\mathfrak l_\Lambda,\mathfrak l_\Lambda]$ be the derived algebra of $\mathfrak l_\Lambda:=\mathfrak h^{\mathbb C} + \sum_{\phi \perp \Psi\smallsetminus \Lambda}\mathfrak g_\phi$, where $\perp$ means the orthogonality with respect to the metric induced by the Killing form of $\mathfrak g$.
Then, $\mathfrak g_{\Lambda,0}:=\mathfrak g_0\cap \mathfrak g_\Lambda$ is a real form of $\mathfrak g_\Lambda$ (cf.\,Wolf \cite[p.\,287]{Wol72}).
Letting $G_{\Lambda,0}$ be the connected Lie subgroup of $G_0$ for $\mathfrak g_{\Lambda,0}$, we define the orbit $X_{\Lambda,0}:=G_{\Lambda,0}(o)\subset X_0=G_0/K$, where $o\in X_0$ is the base point identified with ${\bf 0}\in \Omega\cong X_0$.
Write $\mathfrak m_\Lambda^+:=\mathfrak m^+\cap \mathfrak g_{\Lambda}$.
Note that we have the Harish-Chandra embedding $\xi:\mathfrak m^+\to X_c=G/P$.
Denote by $\xi_\Lambda$ the restriction of $\xi$ to $\mathfrak m_\Lambda^+$.
The sets $\Omega_\Lambda:=\xi_\Lambda^{-1}(X_{\Lambda,0})\Subset \mathfrak m_\Lambda^+$ are called \emph{characteristic subdomains} of $\Omega=\xi^{-1}(X_0)$, which are also irreducible bounded symmetric domains of rank $|\Lambda|$ by \cite[pp.\,287--290]{Wol72}.
Moreover, $\Omega_\Lambda$ are also called the $|\Lambda|$-th characteristic subdomains of $\Omega$.
Actually, Wolf \cite[p.\,292]{Wol72} classified all the characteristic subdomains $\Omega_\Lambda$ of any irreducible bounded symmetric domain $\Omega$. We refer the readers to Mok-Tsai \cite{MT92} for details.

Let $X_c=G/P$ be a Hermitian symmetric space of the compact type and $h_c$ be a canonical K\"ahler metric on $X_c$, where $G=\mathrm{Aut}(X_c)$.
In Tsai \cite{Ts93}, a complex submanifold $M\subset X_c=G/P$ is said to be an \emph{invariantly geodesic submanifold} of $X_c$ if and only if $\varphi(M)$ is a totally geodesic submanifold of $(X_c,h_c)$ for any $\varphi\in G=\mathrm{Aut}(X_c)$. Let $(X_0,h)$ be the noncompact dual Hermitian symmetric space of $(X_c,h_c)$ and $X_0\subset X_c$ be the Borel embedding. We have the bounded symmetric domain $\Omega:=\xi^{-1}(X_0)$ corresponding to $X_0$ and we identify $\Omega\cong X_0\subset X_c$. A complex submanifold $\Sigma\subset \Omega$ of the bounded symmetric domain $\Omega$ is said to be an invariantly geodesic submanifold if and only if there exists an invariantly geodesic submanifold $M\subset X_c$ such that $M$ contains $\Sigma$ as an open subset (see Mok \cite[p.\,138]{Mo08}).
Then, any characteristic symmetric subdomain of $\Omega$ is an invariantly geodesic submanifold of $\Omega$ by \cite[Lemma 2.1]{Mo08}.
In addition, Tsai \cite[Proof of Proposition 4.6]{Ts93} showed that for any irreducible bounded symmetric domain $\Omega$ of rank $r\ge 2$, all invariantly geodesic submanifolds of $\Omega$ are irreducible bounded symmetric domains of rank $\le r$.
More generally, let $\Omega=\Omega_1\times\cdots \times \Omega_m$ be a reducible bounded symmetric domain with irreducible factors $\Omega_j$, $1\le j\le m$, any invariantly geodesic submanifold $\Omega'$ of $\Omega$ is of the form $\Omega'=\Omega'_1\times\cdots \times \Omega'_m$, where each $\Omega'_j \subseteq \Omega_j$ is $\Omega_j$ itself, or an invariantly geodesic submanifold of $\Omega_j$, or of dimension $0$.

\subsection{Canonical K\"ahler-Einstein metric on irreducible bounded symmetric domains}\label{Sec:2.1}
Given an irreducible bounded symmetric domain $\Omega\Subset \mathbb C^N$ in its Harish-Chandra realization, denote by $g_\Omega$ the canonical K\"ahler-Einstein metric on $\Omega$ normalized so that minimal disks of $\Omega\cong G_0/K$ are of constant Gaussian curvature $-2$.
Note that the Bergman kernel of $\Omega$ may be written as
\[ K_\Omega(z,z) = {1\over \mathrm{Vol}(\Omega)} h_\Omega(z,z)^{-(p(\Omega)+2)}, \]
where $z=(z_1,\ldots,z_N)$, $h_\Omega(z,z)$ is some polynomial in ($z_1$,\,$\ldots$,\,$z_N$, $\overline{z_1}$,$\ldots$,$\overline{z_N}$) with $h_\Omega({\bf 0},z)\equiv 1$, $\mathrm{Vol}(\Omega)$ is the Euclidean volume of $\Omega$ in $\mathbb C^N$ with respect to the standard Euclidean metric on $\mathbb C^N$ and $p(\Omega):=p(X_c)=\dim_{\mathbb C} \mathscr C_o(X_c)$ is the complex dimension of the VMRTs $\mathscr C_o(X_c)$ of $X_c\cong G_c/K$ at $o=eK$.
Here $G_c$ is a compact real form of $G=(G_0)^{\mathbb C}$ (cf.\,\cite{Mo89}).
Then, the K\"ahler form $\omega_{g_\Omega}$ with respect to $g_\Omega$ on $\Omega$ is given by
$\omega_{g_\Omega} = \sqrt{-1}\partial\overline\partial (-\log(-\rho))$,
where $\rho(z):=-h_\Omega(z,z)$.
\begin{lem}[cf.\,\cite{Mo14, Mo16}]\label{LemAsGC1}
Let $\mu: U\to \mathbb C^N$ be a holomorphic embedding such that $\mu(U\cap \Delta)\subset \Omega$ and $\mu(U\cap \partial\Delta)\subset \partial\Omega$, where $U\subset \mathbb C$ is an open neighborhood of some point $b_0\in \partial\Delta$ and $\Omega$ is an irreducible bounded symmetric domain of rank $r\ge 2$.
Assume $U\cap\partial\Delta$ is connected.
There is an integer $m\ge 1$ such that for a general point $b\in U\cap \partial\Delta$, $(U\cap \Delta, \mu^*g_\Omega|_{U\cap \Delta})$ is asymptotically of Gaussian curvature $-{2\over m}$ along $U_b\cap\partial\Delta$ for some open neighborhood $U_b$ of $b$ in $U$.
More precisely, there is an integer $m$ such that, denoting by $\kappa(w)$ the Gaussian curvature of $(U\cap \Delta, \mu^*g_\Omega|_{U\cap \Delta})$ at $w\in U\cap \Delta$, we have
\[ \kappa(w)=-{2\over m}+ O(\delta(w)^2) \]
as $w\to b$ for a general point $b\in U\cap \partial\Delta$, where $\delta(w)=1-|w|$ for $w\in \Delta$.
\end{lem}
\begin{proof}
From \cite{Mo14} and \cite{Mo16}, for a general point $b\in U\cap\partial\Delta$, the real-analytic function $-\rho(\mu(w))$ vanishes to the order $m$ on an open neighborhood of $b$ in $U\cap \partial\Delta$ for some integer $m\ge 1$. Then, we have $-\rho(\mu(w)) = (1-|w|^2)^{m}\chi(w)$ on $U_b$ for some positive smooth function $\chi$ defined on a neighborhood $U'$ of $\overline{U_b}$, where $U_b$ is an open neighborhood of $b$ in $U$ such that $U_b\Subset U$.
Then, on $U_b\cap \Delta$ we have
\[ \begin{split}
\mu^*\omega_{g_\Omega}
=& - \sqrt{-1} \partial\overline\partial \log (-\rho(\mu(w)))\\
=& -m\sqrt{-1}\partial\overline\partial\log(1-|w|^2)
-\sqrt{-1} \partial\overline\partial \log \chi(w)\\ 
=&\left({m\over (1-|w|^2)^2}+q(w)\right)\cdot \sqrt{-1}dw \wedge d\overline w
\end{split}\]
where $q(w)=-{\partial^2 \log\chi\over \partial w\partial\overline w}$ is a smooth function defined on $U'$ (cf.\,\cite{Mo14}).
From \cite{Mo14}, it suffices to show that
$q(w)\cdot (1-|w|^2)^2=O(\delta(w)^2)$ on $U_b\cap \Delta$, where $\delta(w):=1-|w|$.
It is clear that $|q(w)|^2$ is bounded on $\overline{U_b}$.
Now, on $U_b\cap \Delta$ we have
$\mu^*\omega_{g_\Omega}={u\over (1-|w|^2)^2}\cdot \sqrt{-1}dw \wedge d\overline w$,
where $u:=m+q(w)(1-|w|^2)^2$.
After shrinking $U_b$ if necessary, we may suppose that $u> 0$ on a neighborhood of $\overline{U_b}$ because $|q(w)|^2$ is locally bounded on $U'$.
Denote by $\kappa(w)$ the Gaussian curvature of $(U\cap \Delta, \mu^*g_\Omega|_{U\cap \Delta})$ at $w\in U\cap \Delta$.
For $w\in U_b\cap \Delta$, we have
\[ \begin{split}
\kappa(w) =&-{(1-|w|^2)^2\over u} {\partial^2 \over \partial w \partial \overline w} \log  {u\over (1-|w|^2)^2} \\
=&
-{1\over u}{\partial^2\log  u \over \partial w \partial \overline w}
(1-|w|^2)^2 - {2\over u}\\ 
=& -{2\over m} +\left({2q(w)\over m\cdot u} -{1\over u}{\partial^2\log  u \over \partial w \partial \overline w}\right)
(1+|w|)^2\cdot \delta(w)^2.
\end{split}\]
Note that ${2q(w)\over m\cdot u} -{1\over u}{\partial^2\log  u \over \partial w \partial \overline w}$ is smooth and real-valued on $U_b$.
Therefore, we have $\kappa(w) = -{2\over m} + O(\delta(w)^2)$ as $w\to b$ for a general point $b\in U\cap \partial\Delta$.
\end{proof}
\subsection{Convention}
Let $M$ be a smooth manifold and $E$ be a differentiable complex vector bundle over $M$. We denote by $\mathcal A(E)$ the sheaf of germs of smooth sections of $E$.  Thus, $\Gamma(M,\mathcal A(E))$ is the complex vector space of 
smooth sections of $E$ over $M$. When $M$ is a complex manifold and $E$ is a holomorphic vector bundle over $M$, $\mathcal O(E)$ denotes the sheaf of germs of holomorphic sections of $E$ over $M$, but we write for short $\Gamma(M,E) := \Gamma(M,\mathcal O(E))$. For germs of sheaves at a point $x \in M$, to emphasize the background manifold $M$ we also write $\Gamma_{{\rm loc},x}(M,\mathcal A(E)) := \mathcal A_x(E)$ and $\Gamma_{{\rm loc},x}(M,E) := \mathcal O_x(E)$.

\section{Construction of embedded Poincar\'e disks}
\label{Sec:CHIE}
\subsection{Holomorphic isometries via the rescaling argument}
\label{Sec:3.1}
Let $\Omega\Subset \mathbb C^N$ be an irreducible bounded symmetric domain of rank $r$ in its Harish-Chandra realization.
Let $\mu:U=\mathbb B^1(b_0,\epsilon)\to \mathbb C^N$, $\epsilon>0$, be a holomorphic embedding such that $\mu(U\cap\Delta)\subset \Omega$ and $\mu(U\cap\partial\Delta)\subset \partial\Omega$, where $b_0\in \partial\Delta$.
Let $\{w_k\}_{k=1}^{+\infty}$ be a sequence of points in $U\cap\Delta$ such that $w_k \to b$ as $k\to +\infty$.
Let $\varphi_k\in \Aut(\Delta)$ be the map
$\varphi_k(\zeta) = {\zeta+ w_k \over 1+\overline{w_k} \zeta}$ and $\Phi_k\in \Aut(\Omega)$ be such that $\Phi_k(\mu(w_k))={\bf 0}$, i.e., $\Phi_k(\mu(\varphi_k(0)))={\bf 0}$, for
$k=1,2,3,\ldots$.
For the sequence $\{\Phi_k\circ (\mu\circ \varphi_k)\}_{k=1}^{+\infty}$ of germs of holomorphic maps from $(\Delta;0)$ to $(\Omega;{\bf 0})$, there exists $\epsilon'>0$ such that all $\Phi_k\circ (\mu\circ \varphi_k)$ are defined on $U':=\mathbb B^1(0,\epsilon')\subset \Delta$.
\begin{lem}\label{LemSeqHE1}
Let $\{w_k\}_{k=1}^{+\infty}$ be a sequence of points in $U\cap\Delta$ converging to $b\in U\cap\partial\Delta$.
Then, shrinking $U'$ if necessary, there is a subsequence of $\{\widetilde \mu_j=\Phi_j\circ (\mu\circ \varphi_j)\}_{j=1}^{+\infty}$ converging to a holomorphic map $\widetilde \mu$ on $U'$ such that
$\widetilde \mu:(\Delta,m_0g_\Delta;0)\to (\Omega,g_\Omega;{\bf 0})$ is a germ of holomorphic isometry for some integer $m_0\ge 1$.
\end{lem} 
\begin{proof}
It is clear that the sequence $\{\widetilde\mu_j=\Phi_j\circ (\mu\circ\varphi_j)\}_{j=1}^{+\infty}$ is bounded on compact subsets of $U'$, so it contains a subsequence $\{\widetilde \mu_{j_k}\}_{k=1}^{+\infty}$ converging uniformly on compact subsets of $U'$ to a holomorphic map $\widetilde \mu$ by Montel's Theorem.
After shrinking $U'$ if necessary, we may suppose that such a sequence $\{\widetilde \mu_{j_k}\}_{k=1}^{+\infty}$ converges uniformly to $\widetilde \mu$ on $\overline{U'}$.

In the proof of Lemma \ref{LemAsGC1}, we have $\mu^*\omega_{g_\Omega} = m_0\omega_{g_\Delta} + q(w) \sqrt{-1} dw\wedge d\overline w$ on $U_b\cap\Delta$, where $U_b=\mathbb B^1(b,\epsilon_b)$ for some $\epsilon_b>0$, $m_0>0$ is an integer and $q(w)$ is smooth and bounded on $U_b$.
For $k$ sufficiently large and $w\in U'$, after shrinking $U'$ if necessary we have $\varphi_k(U')\subset U_b\cap\Delta$ by choosing a suitable sequence $\{w_k\}_{k=1}^{+\infty}$ in $U\cap\Delta$ converging to $b\in \partial\Delta$ and we have
$$
\gathered
\sqrt{-1}\partial\overline\partial\log(-\rho(\widetilde \mu_k(w)))
= \sqrt{-1}\partial\overline\partial\log(-\rho(\mu(\varphi_k(w)))) \\
= m_0 \sqrt{-1}\partial\overline\partial\log(1-|w|^2)
+ q(\varphi_k(w)) |\varphi_k'(w)|^2 \sqrt{-1}dw\wedge d\overline w
\endgathered
$$
so that
${\partial^2\over \partial w\partial\overline w}\log(-\rho(\widetilde \mu_k(w)))
= m_0 {\partial^2\over \partial w\partial\overline w} \log(1-|w|^2) + q(\varphi_k(w)) |\varphi_k'(w)|^2$.
Since $\varphi_k'$ converges uniformly on compact subsets to 0, by taking limit as $k\to +\infty$ (passing to some subsequence of $\{\widetilde \mu_k\}_{k=1}^{+\infty}$ if necessary) we have
${\partial^2\over \partial w\partial\overline w}\log(-\rho(\widetilde \mu(w))) = m_0 {\partial^2\over \partial w\partial\overline w} \log(1-|w|^2)$
so that $\widetilde \mu^* {g_\Omega} = m_0 g_\Delta$ on some open neighborhood of $0$, i.e., $\widetilde \mu : (\Delta,m_0 g_\Delta; 0) \to (\Omega,g_\Omega;{\bf 0})$ is a germ of holomorphic isometry. 
\end{proof}

\subsection{Embedded Poincar\'e disks with uniform geometric properties}
For the purpose of studying properties of certain holomorphically embedded Poincar\'e disks in bounded symmetric domains, we need the following Lemma \ref{LemQuotRA1} and Lemma \ref{LemEigen1}.
\begin{lem}\label{LemQuotRA1}
Let $\phi(\tau)={p(\tau)\over q(\tau)}$ be a quotient of real-analytic functions $p$ and $q$ on $\hat U$, where $\hat U$ is some open neighborhood of $0$ in $\mathbb C$.
Denote by $\mathcal H=\{\tau\in \mathbb C: \mathrm{Im}\tau>0\}$ the upper half-plane in $\mathbb C$.
Assume that $\phi(\tau)$ is bounded on $\hat U\cap \mathcal H$. Then, $\phi(\tau)$ extends real-analytically across a general point $b\in \hat U\cap\partial\mathcal H$.
\end{lem}
\begin{proof}
We may regard $p$ and $q$ as functions of $(x,y)$, where $\tau=x+\sqrt{-1}y \in \hat U$ for $x,y\in \mathbb R$.
We write $p(\tau)=p(x,y),q(\tau)=q(x,y)$ as real-analytic functions of $(x,y)$.
Locally around $0$, we have
$p(x,y) = \sum_{i,j=0}^{+\infty} a_{ij} x^i y^j$ and $q(x,y) = \sum_{i,j=0}^{+\infty} b_{ij} x^i y^j$
for some $a_{ij},b_{ij}\in \mathbb C$.
Then, we have the local holomorphic functions on $\mathbb C^2$ around $(0,0)\in \mathbb C^2$ given by $\hat p(\tau,\zeta) := \sum_{i,j=0}^{+\infty} a_{ij} \tau^i \zeta^j$ and $\hat q(\tau,\zeta) := \sum_{i,j=0}^{+\infty} b_{ij} \tau^i \zeta^j$
with $\mathrm{Re}\tau=x$ and $\mathrm{Re}\zeta = y$.
Let $\hat \phi(\tau,\zeta)={\hat p(\tau,\zeta)\over \hat q(\tau,\zeta)}$, which is a quotient of holomorphic functions around $(0,0)\in \mathbb C^2$.
Then, $\hat \phi$ is a meromorphic function on an open neighborhood $U$ of $(0,0)$ in $\mathbb C^2$.
The set of indeterminacy $I(\hat \phi)$ of $\hat \phi$ is of dimension at most $0$ because it is the intersection of the set $Z(\hat\phi)$ of zeros and the set $P(\hat \phi)$ of poles of $\hat\phi$.
Moreover, the restriction of $\hat\phi$ to $U':=\{(\tau,\zeta)\in U: \im \tau =0,\;\im \zeta = 0\}$ is bounded after shrinking $U$ if necessary, so $U'$ does not intersect $P(\hat \phi)\smallsetminus I(\hat \phi)$.
Note that the set of singular points of $\hat\phi$ on $U$ is $P(\hat \phi)\cup I(\hat\phi)=P(\hat \phi)$, so the above arguments show that the set of potentially bad points of $\phi$ lies inside $I(\hat \phi)\cap U'$, which is of dimension at most $0$.
Hence, for a general point $b\in \hat U\cap\partial\mathcal H$,
$\phi(\tau)$ extends real-analytically around $b$.
\end{proof}

Let $v\in T_x(\Omega)$ be a non-zero tangent vector, $x\in \Omega$. Then, under the $G_0$-action, there is a unique normal form $\eta=(\eta_1,\ldots,\eta_r) \in T_{\bf 0}(\Pi)$ of $v$ satisfying $\eta_j\in \mathbb R$ ($1\le j\le r$) and $\eta_1\ge \cdots\ge \eta_r \ge 0$, where $\Pi\cong \Delta^r$ is a maximal polydisk in $\Omega$ containing ${\bf 0}$ and $r:=\rank(\Omega)$.
We say that a non-zero vector $v\in T_x(\Omega)$ is of rank $k$ if its normal form $\eta=(\eta_1,\ldots,\eta_r)$ satisfies $\eta_1\ge\cdots\ge \eta_k >0$ and $\eta_j=0$ for $k+1\le j\le r$ whenever $k<r$. A rank-$r$ vector $v\in T_x(\Omega)$ is also said to be a generic vector.
Moreover, a zero vector in $T_x(\Omega)$ is said to be a vector of rank $0$.
For the notion of normal forms of tangent vectors in $T_x(\Omega)$, $x\in \Omega$, we refer the readers to \cite{Mo02,Mo89} for details.
\begin{lem}\label{LemEigen1}
Let $v\in T_x(\Omega)$ be a tangent vector of unit length with respect to $g_\Omega$ at $x\in \Omega$ and $\eta=\sum_{j=1}^r \eta_j e_{\psi_j}\in T_{\bf 0}(\Pi)$ be the normal form of $v$.
Then, the Hermitian bilinear form $H_\eta$ defined by $H_\eta(\alpha,\beta)=R_{\eta\overline\eta \alpha \overline \beta}(\Omega,g_\Omega)$ has real eigenvalues lying inside the closed interval $[-2,0]$ and the corresponding Hermitian matrix $\hat H_{\eta}$ of $H_\eta$ can be represented as a diagonal matrix with respect to the standard orthonormal basis $\{e_\varphi:\varphi\in \Delta_M^+\}$ of $\mathfrak m^+$.
\end{lem}
\begin{proof}
\noindent We write $R_{\alpha\overline{\alpha'}\beta\overline{\beta'}}=R_{\alpha\overline{\alpha'}\beta\overline{\beta'}}(\Omega,g_\Omega)$ for simplicity.
From the assumption, we have $\sum_{j=1}^r \eta_j^2 = 1$ and $\eta_1\ge \cdots \ge \eta_r\ge 0$ are real numbers.
Writing $\alpha=\sum_{\varphi\in \Delta_M^+} \alpha_\varphi e_\varphi$, $\beta=\sum_{\varphi\in \Delta_M^+} \beta_\varphi e_\varphi \in T_{\bf 0}(\Omega)\cong \mathfrak m^+$, we compute
\[ \begin{split}
H_{\eta}(\alpha,\beta)
&= \sum_{j=1}^r \eta_j^2 R_{e_{\psi_j}\overline{e_{\psi_j}} \alpha\overline\beta}
= \sum_{j=1}^r \sum_{\varphi \in \Delta_M^+}\eta_j^2 \alpha_\varphi \overline{\beta_\varphi}
R_{e_{\psi_j}\overline{e_{\psi_j}} e_\varphi\overline{e_\varphi}}\\
&=  -2\sum_{j=1}^r \eta_j^2 \alpha_{\psi_j} \overline{\beta_{\psi_j}}
+  \sum_{\varphi \in \Delta_M^+\smallsetminus\Psi}\left(\sum_{j=1}^r \eta_j^2 R_{e_{\psi_j}\overline{e_{\psi_j}} e_\varphi\overline{e_\varphi}} \right) \alpha_\varphi \overline{\beta_\varphi}
\end{split}\]
From \cite{Mo89}, $R_{e_{\psi_j}\overline{e_{\psi_j}} e_\varphi\overline{e_\varphi}}=0$ (resp.\,$-1$) whenever $\psi_j - \varphi$ is not a root (resp.\,$\psi_j- \varphi$ is a root).
Eigenvalues of $H_{\eta}$ are $-2\eta_j^2$, $1\le j\le r$, and those of the form $-(\eta_{i_1}^2+\ldots+\eta_{i_m}^2)$ corresponding to $e_\varphi$ for some $\varphi \in \Delta_M^+\smallsetminus \Psi$ such that $\psi_{i_j} - \varphi$ is a root for $1\le j\le m$ and $\psi_l - \varphi$ is not a root for $l\not\in \{i_j:1\le j\le m\}$.
Here we have $-2\le -2\eta_j^2 \le 0$ ($1\le j\le r$) and $0\ge -(\eta_{i_1}^2+\ldots+\eta_{i_m}^2) \ge -1$ because $\sum_{j=1}^r \eta_j^2 = 1$ and $\eta_j\ge 0$, $1\le j\le r$.
In particular, the eigenvector corresponding to the eigenvalue $-2\eta_j^2$ is precisely $e_{\psi_j}$, $1\le j\le r$.
Note that the above computations imply that the corresponding Hermitian matrix $\hat H_{\eta}$ can be represented as a diagonal matrix with diagonal
$-2\eta_1^2,\ldots,-2\eta_r^2$ and those eigenvalues $-(\eta_{i_1}^2+\ldots+\eta_{i_m}^2)$ mentioned above with respect to the standard orthonormal basis $\{e_\varphi:\varphi\in \Delta_M^+\}$ of $\mathfrak m^+$.
\end{proof}

By \cite[Proof of Proposition 1]{Mo09}, Lemma \ref{LemAsGC1} and Lemma \ref{LemQuotRA1}, we have the following result regarding the local real-analytic extension of the square of the norm of the second fundamental form around a general boundary point of the unit disk.
\begin{pro}[cf.\,$\text{Mok \cite[Proposition 1]{Mo09}}$]\label{pro:2ndFundForm_Ex_RA1}
Let $\Omega\Subset \mathbb C^N$ be a bounded symmetric domain in its Harish-Chandra realization equipped with the Bergman metric $ds_\Omega^2$. Let $\mu:U=\mathbb B^1(b_0,\epsilon)\to \mathbb C^N$, $\epsilon>0$, be a holomorphic embedding such that $\mu(U\cap\Delta)\subset \Omega$ and $\mu(U\cap\partial\Delta)\subset \partial\Omega$, where $b_0\in \partial\Delta$. Denote by $\sigma(z)$ the second fundamental form of $\mu(U\cap\Delta)$ in $(\Omega,ds_\Omega^2)$ at $z=\mu(w)$.
Then, on $U\cap \Delta$ the function $\lVert \sigma(\mu(w))\rVert^2$ is a quotient ${p(s,t)\over q(s,t)}$ of real-analytic functions $p(s,t)$ and $q(s,t)$ in $(s,t)$, where $s=\mathrm{Re}(w)$ and $t=\mathrm{Im}(w)$. Moreover, for a general point $b\in U\cap \partial \Delta$, $\lVert \sigma(\mu(w))\rVert^2$ extends to a real-analytic function $\varphi_b$ on $U_b$ for some open neighborhood $U_b$ of $b$ in $U$. In particular, $\lVert \sigma(\mu(b))\rVert^2:=\varphi_b(b)$ is defined for a general point $b\in U\cap\partial\Delta$.
\end{pro}

Reinforcing the rescaling argument as introduced in Section \ref{Sec:3.1} we are going to construct special holomorphic embeddings of the Poincar\'e disk, as follows.
\begin{pro}\label{ProConstHI1}
Let $\Omega\Subset \mathbb C^N$ be an irreducible bounded symmetric domain of rank $r$ in its Harish-Chandra realization.
Let $\mu:U=\mathbb B^1(b_0,\epsilon)\to \mathbb C^N$, $\epsilon>0$, be a holomorphic embedding such that $\mu(U\cap\partial\Delta)\subset \partial\Omega$ and $\mu(U\cap\Delta)\subset\Omega$, where $b_0\in \partial\Delta$.
Denote by $\sigma(z)$ the second fundamental form of $(\mu(U\cap\Delta),g_\Omega|_{\mu(U\cap\Delta)})$ in $(\Omega,g_\Omega)$ at $z=\mu(w)$.
Let $\{w_k\}_{k=1}^{+\infty}$ be some sequence of points in $U\cap\Delta$ converging to a general point $b\in U\cap\partial\Delta$ as $k\to+\infty$, and let $\varphi_k\in \Aut(\Delta)$ and $\Phi_k\in \Aut(\Omega)$ be such that $\varphi_k(0)=w_k$ and $\Phi_k(\mu(w_k)) = {\bf 0}$, $k=1,2,3,\ldots$.
Then, the sequence of germs of holomorphic embeddings $\{\widetilde \mu_k:=\Phi_k\circ(\mu\circ \varphi_k)\}_{k=1}^{+\infty}$ at $0\in \Delta$ into $\Omega$ $($passing to some subsequence if necessary$)$ converges to the germ of holomorphic isometry 
$\widetilde \mu:(\Delta,m_0g_\Delta;0)\to (\Omega,g_\Omega;{\bf 0})$ for some integer $m_0\ge 1$, say $\widetilde \mu$ is defined on $U'=\mathbb B^1(0,\epsilon')$ for some $\epsilon'>0$, satisfying the following properties:
\begin{enumerate}
\item $\lVert \widetilde \sigma(\widetilde\mu(w)) \rVert^2=\lVert \sigma(\mu(b))\rVert^2$ is independent of $w\in U'$, where $\widetilde \sigma(z)$ denotes the second fundamental form of $(\widetilde \mu(U'),g_\Omega|_{\widetilde \mu(U')})$ in $(\Omega,g_\Omega)$ at $z=\widetilde \mu(w)$ for $w\in U'$,
\item the normal form of ${\widetilde \mu'(w)\over \lVert \widetilde\mu'(w)\rVert_{g_\Omega}}$ is independent of $w\in U'$ and so is the rank of ${\widetilde \mu'(w)\over \lVert \widetilde\mu'(w)\rVert_{g_\Omega}}$.
\end{enumerate}
By the same procedure, this yields a holomorphic isometry from $($$\Delta$, $m_0g_\Delta$$)$ to $(\Omega,{g_\Omega})$, denoted also by $\widetilde \mu$, such that properties $1$ and $2$ hold true on $\Delta$.
\end{pro}
\begin{proof}
In Lemma \ref{LemSeqHE1} we have already constructed the germ of holomorphic isometry $\widetilde \mu$:$(\Delta,m_0g_\Delta;0)$ $\to$ $ (\Omega,{g_\Omega};{\bf 0})$. We will show that $\widetilde \mu$ satisfies properties $1$ and $2$.
We have $\widetilde \mu'(w)= \lim_{k\to +\infty} \widetilde \mu_k'(w) \neq {\bf 0}$ for $w\in U'$ because $\widetilde \mu:(\Delta,m_0g_\Delta;0)\to (\Omega,{g_\Omega};{\bf 0})$ is a germ of holomorphic isometry.
Let $\widetilde\eta_k(w)$ (resp.\,$\eta(w)$) be the normal form of ${\widetilde\mu_k'(w)\over \lVert\widetilde\mu_k'(w)\rVert_{g_\Omega}}$ (resp.\,${\mu'(w)\over \lVert \mu'(w)\rVert_{g_\Omega}}$) for $w\in U'$ (resp.\,$w\in U\cap\Delta$).
We also let $\widetilde \eta(w)$ be the normal form of ${\widetilde \mu'(w)\over \lVert \widetilde \mu'(w)\rVert_{g_\Omega}}$.

Writing a tangent vector $\upsilon=\sum_{j=1}^r \upsilon_j e_{\psi_j}$ in normal form, we let $H_{\upsilon}(\alpha,\beta):=R_{\upsilon\overline{\upsilon}\alpha\overline\beta}(\Omega,{g_\Omega})$ be the Hermitian bilinear form and $\hat H_{\upsilon}$ be the corresponding Hermitian matrix. Denote by $P_{\upsilon}(\lambda):=\det(\lambda I_N - \hat H_{\upsilon})$ the characteristic polynomial of $\hat H_{\upsilon}$.
We have shown that all eigenvalues of $H_{\eta(w)}$ (resp.\,$H_{\widetilde\eta_k(w)}$, resp.\,$H_{\widetilde\eta(w)}$) belong to $[-2,0]\subset\mathbb R$ by Lemma \ref{LemEigen1}.
For simplicity, we may suppose that $\varphi_k(U')\subset U\cap\Delta$ for any $k\ge 1$.
Fix an arbitrary point $w\in U'$.
From the construction ${\widetilde \mu_k'(w)\over \lVert \widetilde \mu_k'(w) \rVert_{g_\Omega}}$ is equivalent to ${\varphi_k'(w)\over |\varphi_k'(w)|}{\mu'(\varphi_k(w))\over \lVert \mu'(\varphi_k(w))\rVert_{g_\Omega}}$ under the $G_0$-action so that the normal form $\widetilde \eta_k(w)$ is equivalent to $\eta(\varphi_k(w))$ under the $K$-action for $k\ge 1$, where $G_0:=\mathrm{Aut}_0(\Omega)$ and $K\subset G_0$ is the isotropy subgroup at ${\bf 0}$.
From the uniqueness of the normal form (cf.\,Mok \cite{Mo02}) we have $\widetilde \eta_k(w)=\eta(\varphi_k(w))$ and thus $H_{\widetilde \eta_k(w)}=H_{\eta(\varphi_k(w))}$ for any integer $k\ge 1$.
Since the eigenvalues of $H_{\eta(\zeta)}$ belong to $[-2,0]\subset\mathbb R$, the coefficients of $P_{\eta(\zeta)}(\lambda)$ are bounded functions of $\zeta$ on $U\cap\Delta$ and may be written as quotients of real-analytic functions of $\zeta$ on $U_b=\mathbb B^1(b,\epsilon_b)\subset U$ from the construction for some $\epsilon_b>0$.
It follows from Lemma \ref{LemQuotRA1} that for a general point $b\in U\cap\partial\Delta$ all coefficients of $P_{\eta(\zeta)}(\lambda)$ can be extended as real-analytic functions of $\zeta$ on $U_{b}$. 
By shrinking $U'$ if necessary, we may suppose that $\varphi_k(U')$ lies inside $U_b\cap\Delta$ for $k$ sufficiently large.
Since $\varphi_k(w)\to b$ as $k\to+\infty$ for any $w\in U'$, $\{P_{\eta(\varphi_k(w))}(\lambda)\}_{k=1}^{+\infty}$ converges to some polynomial $P_\infty(\lambda)$ of $\lambda$ which is independent of $w\in U'$.
In particular, the roots of $P_\infty(\lambda)$ are independent of $w\in U'$.
Since $P_{\widetilde \eta_k(w)}(\lambda)=P_{\eta(\varphi_k(w))}(\lambda)$ and some subsequence of $\{P_{\widetilde \eta_k(w)}(\lambda)\}_{k=1}^{+\infty}$ converges to $P_{\widetilde \eta(w)}(\lambda)$, we have $P_{\widetilde\eta(w)}(\lambda)=P_\infty(\lambda)$ so that the roots of $P_{\widetilde\eta(w)}(\lambda)$, equivalently the eigenvalues of $H_{\widetilde \eta(w)}$, are independent of $w\in U'$.

Write $\widetilde \eta(w)=\sum_{j=1}^r a_j(w)e_{\psi_j}$, where $a_1(w)\ge\cdots \ge a_r(w)\ge 0$ are real.
Then, $-2a_1(w)^2$, $\ldots$, $-2a_r(w)^2$ are some eigenvalues of $H_{\widetilde\eta(w)}$ by the proof of Lemma \ref{LemEigen1}. Since for each $j$, $1 < j < r$, $a_j(w)$ varies continuously in $w$ and there are only finitely many nonnegative real numbers $\alpha$ such that each $-2\alpha^2$ is among the $N$ eigenvalues of $H_{\widetilde 
\eta(w)}$ (which are independent of $w$), we conclude that
the normal form $\widetilde \eta(w)=\sum_{j=1}^r a_j(w)e_{\psi_j}$ is independent of $w\in U'$ and so is the rank of $\widetilde\eta(w)$, i.e., $\widetilde \mu$ satisfies property $2$.

Since $\widetilde\mu$ is a germ of holomorphic isometry from $(\Delta,m_0g_\Delta)$ to $(\Omega,g_\Omega)$, from the Gauss equation we have
$$\lVert \widetilde\sigma(\widetilde \mu(w)) \rVert^2
=R_{\widetilde \eta(w)\overline{\widetilde \eta(w)}\widetilde \eta(w)\overline{\widetilde \eta(w)}}(\Omega,g_\Omega)-\left(-{2\over m_0}\right)
=-2\sum_{j=1}^r a_j(w)^4+{2\over m_0},$$
which is independent of $w\in U'$ because $a_j(w)$, $1\le j\le r$, are independent of $w\in U'$ from the last paragraph.
Actually, denoting by $\kappa(\zeta)$ the Gauss curvature of $(U\cap\Delta,\mu^*g_\Omega|_{U\cap\Delta})$ at $\zeta\in U\cap\Delta$ we have
\[ R_{\widetilde \eta_k(w)\overline{\widetilde \eta_k(w)}\widetilde \eta_k(w)\overline{\widetilde \eta_k(w)}}(\Omega,{g_\Omega})
-\kappa(\varphi_k(w)) = \lVert \sigma(\mu(\varphi_k(w)))\rVert^2\]
for $w\in U'$.
Since the right-hand side of the above equality converges to $\lVert \widetilde\sigma(\widetilde \mu(w)) \rVert^2$ as $k\to +\infty$ (by passing to some subsequence if necessary) and $\lVert \sigma(\mu(\zeta))\rVert^2$ extends as a real-analytic function around a general point $b$ of $U\cap\partial\Delta$ (cf.\,Mok \cite{Mo09}), we have $\lVert \widetilde\sigma(\widetilde \mu(w)) \rVert^2
=\lVert \sigma(\mu(b)) \rVert^2$ for $w\in U'$.

Note that $\widetilde \mu$ extends to a holomorphic isometry $\widetilde \mu:(\Delta,m_0g_\Delta)\to (\Omega,g_\Omega)$ by \cite{Mo12}. By choosing a good boundary point $b\in \partial\Delta$ and by the same procedure, we can construct a (global) holomorphic isometry from $(\Delta,m_0g_\Delta)$ to $(\Omega,g_\Omega)$, denoted also by $\widetilde\mu$, such that $\widetilde\mu$ satisfies properties $1$ and $2$ on $\Delta$, as desired. 
\end{proof}
\begin{rem}\label{RemarkCHIE1}
\rm{\begin{enumerate}
\item[\rm{(a)}]
The positive integer $m_0$ is actually the vanishing order of $\rho(\mu(w))$ as $w\to b$ and we have $-\rho(\mu(w))=(1-|w|^2)^{m_0}\chi(w)$ on $U_b=\mathbb B^1(b,\epsilon_b)$, $\epsilon_b>0$, for some positive smooth function $\chi$ on $U_b$.
\item[\rm{(b)}] We equip a bounded symmetric domain $\Omega$ with the Bergman metric in the statement of Theorem \ref{MainThm} since we need to apply the extension theorem of Mok \cite{Mo12} which was only proven for Bergman metrics.
\end{enumerate}}
\end{rem}

\section{Proof of Theorem \ref{MainThm}}
\label{Sec:Proof}
We first prove a special case of Theorem \ref{MainThm} as follows where the bounded symmetric domain $\Omega$ is irreducible and of tube type, and the argument will be generalized to the case where $\Omega$ is reducible and of tube type.
We will also show that the general case of Theorem \ref{MainThm} where $\Omega$ is an arbitrary bounded symmetric domain is reducible to the case where $\Omega$ is of tube type.

\begin{thm}\label{ThmTubeDomain1}
Let $\Omega\Subset \mathbb C^N$ be an irreducible bounded symmetric domain of rank $r\ge 2$ in its Harish-Chandra realization.
Suppose $\Omega$ is of tube type.
Let $\mu:U=\mathbb B^1(b_0,\epsilon)\to \mathbb C^N$, $\epsilon>0$, be a holomorphic embedding such that $\mu(U\cap\Delta)\subset \Omega$ and $\mu(U\cap\partial\Delta)\subset \partial\Omega$, where $b_0\in \partial\Delta$.
Denote by $\sigma(z)$ the second fundamental form of $\mu(U\cap\Delta)$ in $(\Omega,g_\Omega)$ at $z=\mu(w)$. Then, for a general point $b\in U\cap\partial\Delta$ we have $\lim_{w\in U\cap \Delta,\;w\to b} \lVert \sigma(\mu(w))\rVert = 0$.
\end{thm}
\subsection{Geometry on embedded Poincar\'e disks}
\label{Proof:Sec1}
\subsubsection{Geometry on embedded Poincar\'e disks in tube domains}
In this section, we suppose that $\Omega$ is an irreducible bounded symmetric domain of tube type and of rank $\ge 2$.
Recall that we have constructed a holomorphic isometry $\widetilde\mu:(\Delta,m_0g_\Delta)\to (\Omega,{g_\Omega})$ from $\mu$ such that $\widetilde \mu(0)={\bf 0}$, $\lVert \widetilde \sigma(\widetilde \mu(w))\rVert^2\equiv \lVert \sigma(\mu(b))\rVert^2$ for any $w\in \Delta$ and $\widetilde\mu'(w)=d\widetilde\mu\left({\partial\over \partial w}\right)(w)$ is of constant rank $k$ on $\Delta$ for some $k$, $1\le k\le r=\rank(\Omega)$.
We write $Z=\widetilde\mu(\Delta)$ and $\eta(w)=\sum_{j=1}^k \eta_j(w) e_{\psi_j}$ as the normal form of ${\widetilde \mu'(w)\over \lVert \widetilde \mu'(w) \rVert_{g_\Omega}}$ with $\eta_1(w)\ge \cdots \ge \eta_k(w)>0$, where $\Psi=\{\psi_1,\ldots,\psi_r\}$ is a maximal strongly orthogonal set of noncompact positive roots \cite{Wol72}.
Let $\mathcal N_\eta$ be the null space of the Hermitian bilinear form $H_\eta(\alpha,\beta)=R_{\eta\overline\eta\alpha\overline\beta}(\Omega,{g_\Omega})$, which is of complex dimension $n_k(\Omega)$.
Here $n_k(\Omega)$ is the $k$-th null dimension of the irreducible bounded symmetric domain $\Omega$ (cf.\,\cite{Mo89}), noting that $n_r(\Omega) = 0$.
We also write $n_0(\Omega) := \dim_{\mathbb C}(\Omega)$.
For $x\in \Omega$, let $Q_x$ be the Hermitian bilinear form on $T_x(\Omega)\varotimes\overline{T_x(\Omega)}$ given by
$Q_x(\alpha\varotimes\overline\beta,\alpha'\varotimes\overline{\beta'}) = R_{\alpha\overline{\alpha'}\beta'\overline\beta}(\Omega,g_\Omega)$.
For $w\in \Delta$, we define
\[ W_{\widetilde\mu(w)}
= \left\{ v\in T_{\widetilde\mu(w)}(\Omega) : Q_{\widetilde\mu(w)}(v\varotimes \overline \zeta,\cdot )\equiv 0\;\;\forall\;\zeta\in \mathcal N_{\widetilde\mu'(w)}\right\}, \]
where $\mathcal N_{\widetilde\mu'(w)}
=\mathcal N_{\eta(w)}
=\left\{ v\in T_{\widetilde\mu(w)}(\Omega): R_{\eta(w)\overline{\eta(w)}v\overline v}(\Omega,g_\Omega)=0\right\}=\{ v \in T_{\widetilde\mu(w)}(\Omega): \eta(w)\varotimes \overline{v} \in \Ker(Q_{\widetilde\mu(w)})\}$.
Then, we have $T_{\widetilde\mu(w)}(Z)\subset W_{\widetilde\mu(w)}\subset T_{\widetilde\mu(w)}(\Omega)$.
Letting
$\mathcal N_k = \bigcap_{j=1}^k \{ \varphi\in \Delta_M^+: \varphi\neq \psi_j, \;\varphi -\psi_j \text{ is not a root}\}$, we have $\mathcal N_\eta = \bigoplus_{\varphi\in \mathcal N_k} \mathfrak g_\varphi$.
Define $\widetilde{\mathcal N}:=\bigcap_{\varphi \in \mathcal N_k} \{ \psi \in \Delta_M^+; \psi\neq \varphi,\; \psi - \varphi \text{ is not a root}\}$. Then, the normal form of $W_{\widetilde\mu(w)}$ is given by
$\bigcap_{\zeta\in \mathcal N_\eta} \mathcal N_\zeta
= \bigoplus_{\psi\in \widetilde{\mathcal N}} \mathfrak g_\psi$.
\begin{lem}\label{LemHVSB_CSD1}
In the above construction, if $\Omega$ is of tube type, then for any $x\in Z$, $W_x = T_x(\Omega'_x)$ for some characteristic subdomain $\Omega'_x\subseteq \Omega$ of rank $k$ passing through $x$ and $\Omega'_x$ is of tube type.
\end{lem}
\begin{proof}
Fix $x\in Z$.
We first consider the case where $\Omega=D^\mathrm{VI}$.
If $k=3=\mathrm{rank}(\Omega)$, then $W_x=T_x(\Omega)$ so that the result follows directly and $\Omega'_x=\Omega$.
If $k=1$, then $W_x=T_x(Z)=T_x(\Delta_\eta)$ with $\Delta_\eta\subset \Omega$ being the minimal disk passing through $x=\widetilde \mu(w)$ because $\bigcap_{\zeta\in \mathcal N_{\eta(w)}} \mathcal N_\zeta = \mathbb C \eta(w)$ (cf.\,\cite[p.\,98]{MT92}).
Suppose $k=2$.
Note that the automorphism group of the exceptional domain $D^{\mathrm{VI}}$ corresponds to the Lie group $E_7$.
From \cite{Zh84} and \cite[p.\,868]{Si81}, we put $\Psi=\{\psi_1,\psi_2,\psi_3\}$ with $\psi_1=x_1-x_2$, $\psi_2=x_1+x_2+x_3$ and $\psi_3=\sum_{j=1}^7 x_j - x_3$, where $x_j$, $1\le j\le 7$, is the standard basis of $\mathbb R^7$.
Write $\eta(w) = \eta_1(w) e_{x_1-x_2}
+\eta_2(w) e_{x_1+x_2+x_3}$.
Then, we have
\[\begin{split}
 \mathcal N_2
=&\bigcap_{j=1}^2\{ \varphi\in \Delta_M^+: \varphi\neq \psi_j, \;\varphi -\psi_j \text{ is not a root}\}\\
=&\left\{\sum_{j=1}^7 x_j - x_3\right\}
= \{\psi_3\}. \end{split}\]
Actually, if $\eta(w)=\eta_1(w) e_{\psi_{j_1}}
+\eta_2(w) e_{\psi_{j_2}}$ for some distinct $j_1,j_2\in \{1,2,3\}$, then
$\mathcal N_{\eta(w)} = \mathbb C e_{\psi_{j_3}}$
with $j_3\in \{1,2,3\}\smallsetminus \{j_1,j_2\}$.
Since $e_{\psi_{j_3}}$ is a characteristic vector, the normal form of $W_{\widetilde \mu(w)}$ is
$\mathcal N_{e_{\psi_{j_3}}}=T_{\bf 0} (\Omega')$ for some characteristic subdomain $\Omega'\subset\Omega=D^{\mathrm{VI}}$ of rank $2$ by \cite[Proposition 1.8]{MT92} and we have $\Omega'\cong D^{\mathrm{IV}}_{10}$ by \cite{Wol72}.
When $\Omega$ is of type-$\mathrm{IV}$ and $k=1$ (resp.\,$k=2$), we have $W_x=T_x(Z)=T_x(\Delta_\eta)$ for a unique minimal disk $\Delta_\eta\subset \Omega$ passing through $x\in Z$ satisfying $T_x(\Delta_\eta)=\mathbb C\eta$ (resp.\,$W_x=T_x(\Omega)$). Note that these arguments not only work for $D^{\mathrm{IV}}_N$, but also for any irreducible bounded symmetric domain of rank $2$, including $D^{\mathrm{V}}$.

When $\Omega$ is of type $\mathrm{I},\mathrm{II}$ or $\mathrm{III}$, the result follows from the use of normal form $\eta$ and the computations in \cite{Mo89}.
If $k=r$, then we have $W_x=T_x(\Omega)$.
For each $x\in Z$, we see that the normal form of $W_{x}$ is the holomorphic tangent space to some characteristic symmetric subdomain $\Omega'\subset \Omega$ of rank $k$ at ${\bf 0}$, as follows.
\begin{enumerate}
\item When $\Omega=D^{\mathrm{I}}_{p,p}$, $2\le p=r$, and $1\le k\le p$, the normal form $\eta$ $=$ $\mathrm{diag}_{p,p}$ $(\eta_1,\ldots,\eta_k,0,\ldots,0)$ is a $p$-by-$p$ diagonal matrix and it is clear that
\[ \bigcap_{\zeta \in \mathcal N_\eta} \mathcal N_\zeta
=  \left\{\begin{bmatrix}
Z' & \\
 & {\bf 0}
\end{bmatrix} \in M(p,p;\mathbb C): Z'\in M(k,k;\mathbb C)\right\} = T_{\bf 0} (D^{\mathrm{I}}_{k,k}) \]
by \cite{Mo89},
where we identify $\Omega'=D^{\mathrm{I}}_{k,k}$ with its image via the standard embedding $D^{\mathrm{I}}_{k,k} \hookrightarrow D^{\mathrm{I}}_{p,p}$, $Z' \mapsto \begin{bmatrix}
Z' & \\
 & {\bf 0}
\end{bmatrix}$.
\item When $\Omega=D^{\mathrm{III}}_r$, the normal form $\eta=\mathrm{diag}_{r,r}(\eta_1,\ldots,\eta_k,0,\ldots,0)$ is a $r$-by-$r$ diagonal matrix, and it is clear that
\[ \bigcap_{\zeta \in \mathcal N_\eta} \mathcal N_\zeta = \left\{\begin{bmatrix}
Z' & \\
 & {\bf 0}
\end{bmatrix} \in M_s(r;\mathbb C): Z'\in M_s(k;\mathbb C)\right\} = T_{\bf 0} (D^{\mathrm{III}}_{k}) \]
by \cite{Mo89},
where we identify $\Omega'=D^{\mathrm{III}}_{k}$ with its image via the standard embedding $D^{\mathrm{III}}_{k} \hookrightarrow D^{\mathrm{III}}_{r}$, $Z' \mapsto \begin{bmatrix}
Z' & \\
 & {\bf 0}
\end{bmatrix}$.
\item When $\Omega=D^{\mathrm{II}}_{2r}$, the normal form
$\eta = \mathrm{diag}_{2r,2r}(\eta_1 J_1,\ldots,\eta_k J_1,{\bf 0}
)$ is a $2r$-by-$2r$ block diagonal matrix,
where $J_1 := \begin{pmatrix}
0 & 1\\
-1 & 0
\end{pmatrix}$.
Then, it is clear that
\[ \bigcap_{\zeta \in \mathcal N_\eta} \mathcal N_\zeta
=  \left\{\begin{bmatrix}
Z' & \\
 & {\bf 0}
\end{bmatrix} \in M_a(2r;\mathbb C): Z'\in M_a(2k;\mathbb C)\right\} = T_{\bf 0}(D^{\mathrm{II}}_{2k}) \]
by \cite{Mo89},
where $\Omega'=D^{\mathrm{II}}_{2k}$ is identified with its image via the standard embedding $D^{\mathrm{II}}_{2k}\hookrightarrow D^{\mathrm{II}}_{2r}$, $Z'\mapsto \begin{bmatrix}
Z'&\\
& {\bf 0}
\end{bmatrix}$.
\end{enumerate}
From the classification of boundary components of any irreducible bounded symmetric domain and the notion of the characteristic subdomains in \cite{Wol72} and \cite{MT92}, we see that $\Omega'\subset \Omega$ is a characteristic subdomain of rank $k$.
Then, by using the $G_0$-action and the fact that $\Omega'$ is an invariantly geodesic submanifold of $\Omega$, we see that $W_x = T_x(\Omega'_x)$ for some characteristic subdomain $\Omega'_x\subseteq \Omega$ of rank $k$.
Since $\Omega$ is of tube type, all its characteristic subdomains are of tube type (cf.\,\cite{Wol72}).
\end{proof}
\begin{rem}\label{RemHVSB_CSD1}\rm{
Let $\Omega$ be an irreducible bounded symmetric domain of rank $r\ge 2$ which is not necessarily of tube type. Assume that $T_x(Z)$ is spanned by a rank-$k$ vector $\eta_x\in T_x(\Omega)$ for each $x\in Z$ with $k<r$. Then, for any $x\in Z$ we have $W_x = T_x(\Omega'_x)$ for some invariantly geodesic submanifold $\Omega'_x\subseteq \Omega$ passing through $x$ such that $\Omega'_x$ is an irreducible bounded symmetric domain of rank $k$ and of tube type.
More precisely, when $\Omega$ is of non-tube type and $\eta$ is of rank $k<r$, we have
\begin{enumerate}
\item[(a)] If $\Omega\cong D^{\mathrm{I}}_{p,q}$, $q> p=r\ge 2$, then $W_x\cong T_{\bf 0}(D^{\mathrm{I}}_{k,k})=M(k,k;\mathbb C)$ and $\Omega'_x\cong D^{\mathrm{I}}_{k,k}$.
\item[(b)] If $\Omega\cong D^{\mathrm{II}}_{2r+1}$, $r\ge 2$, then
$W_x\cong T_{\bf 0}(D^{\mathrm{II}}_{2k})=M_a(2k;\mathbb C)$ and $\Omega'_x\cong D^{\mathrm{II}}_{2k}$.
\item[(c)] When $\Omega\cong D^{\mathrm{V}}$ (which corresponds to the Lie group $E_6$), the result already follows as we have mentioned in the proof of Lemma \ref{LemHVSB_CSD1}.
\end{enumerate}
}
\end{rem}

From now on the holomorphic vector bundle $W$ is taken to be defined for $\Omega$ irreducible and for $Z\subset \Omega$ a holomorphically embedded Poincar\'e disk with $\mathrm{Aut}(\Omega)$-equivalent holomorphic tangent spaces $T_x(Z)=\mathbb C \eta_x$ of rank $k<r=:\mathrm{rank}(\Omega)$, i.e., $\eta_y\in T_y(\Omega)$ is a rank-$k$ tangent vector for any $y\in Z$.

\begin{lem}\label{LemHVSB_W1}
In the above construction, $W:=\bigcup_{x\in Z}W_x\subset T_{\Omega}|_Z$ is a holomorphic vector subbundle.
\end{lem}
\begin{proof}
On the holomorphic curve $Z$ we write $\mathcal N:= \bigcup_{w\in \Delta} \mathcal N_{\eta(w)}$.  For $x\in Z=\widetilde \mu(\Delta)$, we have 
$\mathcal N_x = \{ \zeta \in T_x(\Omega): Q(\xi\varotimes \overline\zeta,\cdot ) \equiv 0,\;\forall\;\xi\in T_x(Z) \}$ and 
$W_x = \{ \gamma\in T_x(\Omega): Q(\gamma\varotimes \overline\zeta,\cdot ) \equiv 0,\;\forall\;\zeta\in \mathcal N_x \}$.
We claim first of all that the vector subbundle $\mathcal N \subset T_\Omega|_Z$ is $\nabla$-invariant. 
We consider arbitrary $\xi \in T_x(Z)$, $\zeta \in \mathcal N_x$ and $\alpha, \beta \in T_x(\Omega)$ and by abuse of notation use the same symbols to denote extensions of these vectors at the point $x$ to smooth local sections at $x$ sometimes subject to additional conditions, and the same convention will be adopted throughout the rest of the section. Since $T_Z \subset T_\Omega|_Z$ is a holomorphic line subbundle, any $\xi \in T_x(Z)$ can be extended to $\xi \in \Gamma_{\mathrm{loc},x}(Z,T_Z)$.  Since $\nabla R \equiv 0$, for any $(1,0)$-tangent vector $v$ of $Z$ at $x$, for the extensions $\zeta \in \Gamma_{\mathrm{loc},x}(Z,\mathcal A(\mathcal N))$ and $\alpha, \beta \in \Gamma_{\mathrm{loc},x}(Z,\mathcal A(T_\Omega|_Z))$, we have
\[ 0=\nabla_{\overline v} (Q(\xi\varotimes \overline\zeta,\alpha\varotimes \overline\beta))
= Q(\xi\varotimes \overline{\nabla_v\zeta},\alpha\varotimes \overline\beta). \]
It follows that $\nabla_v\zeta \in \mathcal N_x$, hence $\mathcal N$ is $\nabla$-invariant, as claimed.
If we identify $\overline{T_\Omega|_Z}$ with $T^*_\Omega|_Z$ by means of the lifting operator defined by the K\"ahler metric $g_\Omega$, $\overline{\mathcal N}$ can be identified with a holomorphic vector subbundle of $T^*_\Omega|_Z$. Through this identification, any $\zeta \in \mathcal N_x$ can be extended to $\zeta \in \Gamma_{\mathrm{loc},x}(Z,\mathcal A(\mathcal N))$ such that $\nabla_v\zeta \equiv 0$.  

We fix $x\in Z$. Let $\gamma \in \Gamma_{\mathrm{loc},x}(Z,\mathcal A(W))$.
Then, for any $(1,0)$-tangent vector $v$ of $Z$ at $x$ and any $\zeta \in \Gamma_{\mathrm{loc},x}(Z,\mathcal A(\mathcal N))$ so that $\nabla_v\zeta \equiv 0$, we have  
\[ 0=\nabla_{\overline v} (Q(\gamma\varotimes \overline\zeta,\alpha\varotimes \overline\beta))
= Q(\nabla_{\overline v} \gamma \varotimes \overline\zeta,\alpha\varotimes \overline\beta) \]
for $\alpha, \beta \in \Gamma_{\mathrm{loc},x}(Z,\mathcal A(T_\Omega|_Z))$.
Therefore, we have $(\nabla_{\overline v} \gamma)(x)\in W_x$. As a consequence, $W \subset T_\Omega|_Z$ is $\overline{\nabla}$-invariant, i.e., $\overline{\partial}$-invariant, hence $W \subset T_\Omega|_Z$ is a holomorphic vector subbundle, as desired.
\end{proof}
\begin{lem}\label{TauTensorHolo1}
Define the $(1,0)$-part of the second fundamental form $\tau:T_Z \varotimes W \to T_\Omega|_Z/W$ of the holomorphic vector subbundle $(W,g_\Omega|_W)\subset (T_\Omega|_Z,g_\Omega)$ by $\tau_x(\eta\varotimes \gamma) = (\nabla_\eta \gamma)(x) \mod W_x$ for each $x\in Z$, $\eta\in T_x(Z)$ and $\gamma\in W_x$.
Then, $\tau$ is holomorphic.
\end{lem}
\begin{proof}
We fix $x\in Z$. It suffices to show that for any $(1,0)$-tangent vectors $\beta$ and $\eta$ of $Z$ at $x$, and any $\gamma\in \Gamma_{\mathrm{loc},x}(Z,W)$, $\nabla_{\overline \beta}(\nabla_\eta \gamma)(x) \in W_x$. This would imply that $\nabla_{\overline\beta}(\tau(\eta\varotimes \gamma))= 0$, i.e., $\tau$ is holomorphic, by projecting $\nabla_{\overline \beta}(\nabla_\eta \gamma)(x)$ to the quotient bundle $T_\Omega|_Z/W$.
Note that $R(\eta,\overline\beta)\gamma=-\nabla_{\overline \beta}(\nabla_\eta \gamma)$, so it suffices to show that $R_{\eta\overline \beta \gamma \overline \xi}(\Omega,g_\Omega)=0$ for any $\xi$ orthogonal to $W$, equivalently $R(\eta,\overline\beta)\gamma$ takes values in $W$.
For each $x\in Z$, $W_x=T_x(\Omega'_x)$ for some characteristic subdomain $\Omega'_x\subset \Omega$ of rank $k$ containing $x$.
Note that $\Omega'_x \subset \Omega$ is an invariantly geodesic submanifold, we can regard $x$ as a base point of $\Omega$ and thus
$[[\mathfrak m^-,W_x],W_x] \subset W_x$
by \cite[Lemma 4.3]{Ts93}.
This shows that $(R(\eta,\overline\beta)\gamma)(x) = [[\overline{\beta(x)},\eta(x)],\gamma(x)] \in W_x$ because $\eta(x)\in T_x(Z) \subset W_x$ and $\gamma(x)\in W_x$.
Then, $-\nabla_{\overline\beta}(\nabla_\eta \gamma) = R(\eta,\overline\beta)\gamma$ takes value in $W$ so that $\tau$ is holomorphic.
Moreover, we can regard $\tau \in \Gamma(Z,  T_Z^*\varotimes W^* \varotimes (T_\Omega|_Z/W))$ as a holomorphic section.
\end{proof}

Our next goal is to show that $\tau$ vanishes identically.
The first step is to obtain the asymptotic vanishing of $\lVert \tau|_{T_Z\varotimes T_Z}(\zeta)\rVert^2$ as $\zeta$ approaches a general point $\hat b\in \partial\Delta$.
For this purpose, we will need the local holomorphic extension of the second fundamental form $\tau$ around a general point $b'\in \partial\Delta$.
Therefore, we will extend the definition of $W_{\widetilde\mu(\zeta)}$, $\zeta\in \Delta$, to some open neighborhood of $b'$ by making use of the local holomorphic extension $f$ of $\widetilde\mu$ to $U_{b'}:=\mathbb B^1(b',\epsilon)$ for some $\epsilon>0$.
Actually, we will define a complex vector space $V_x$ for any $x\in Z=\widetilde\mu(\Delta)$ (resp.\,$x\in f(U_{b'})$).
Then, we show that the vector bundles $V=\bigcup_{x\in Z}V_x$ and $W$ are identical (when $\Omega$ is of tube type) and that $V':= \bigcup_{x\in S}V_x$ is also a holomorphic vector bundle over $S:=f(U_{b'})$.

Before defining the vector bundles $V$ and $V'$, we need the following basic setting.
Identify $\Omega\cong G_0/K$, where $G_0=\mathrm{Aut}_0(\Omega)$ and $K\subset G_0$ is the isotropy subgroup at ${\bf 0}$.
Note that we have the Harish-Chandra decomposition
$\mathfrak g=\mathfrak m^-\varoplus \mathfrak k^{\mathbb C}\varoplus \mathfrak m^+$ of the Lie algebra $\mathfrak g$ of the Lie group $G$, where $G$ is the complexification of $G_0$.
Moreover, we have $\Omega \Subset \mathfrak m^+\cong T_{\bf 0}(\Omega)=\mathbb C^N$ and we can identify $T_x(\Omega)\cong \mathfrak m^+$ for any $x\in \Omega$.
Recall that $\mathfrak p =\mathfrak m^-\varoplus \mathfrak k^{\mathbb C}$ is the parabolic subalgebra of $\mathfrak g$ and $P\subset G$ is the parabolic subgroup with Lie algebra $\mathfrak p$.
We identify $\Omega\Subset \mathfrak m^+\cong\mathbb C^N \subset X_c=G/P$ as an open subset by the Harish-Chandra and Borel embeddings, where $X_c=G/P$ is the compact dual of $\Omega$.
Note that we may regard $\mathfrak g$ as the Lie algebra of holomorphic vector fields on $X_c$, and $\mathfrak m^-$ as the vector space of holomorphic vector fields vanishing to the order $\ge 2$ at $x$ for $x\in X_c$.
We write $\mathfrak p_x=\mathfrak m^-\varoplus\mathfrak k_x^{\mathbb C}$, where $\mathfrak p_x$ is the Lie algebra of the parabolic isotropy subgroup $P_x \subset G$ at $x\in X_c$ and $\mathfrak k_x^{\mathbb C}$ is the Lie algebra of a Levi factor $K_x^{\mathbb C}\subset P_x$ of $P_x$.

In Lemma \ref{LemHVSB_W1} we have proven that $W \subset T_\Omega|_Z$ is a holomorphic vector subbundle.
From \cite{Mo12} there is a local holomorphic extension $f$ of $\widetilde\mu$ to $U_{b'}=\mathbb B^1(b',\epsilon)$, $\epsilon>0$, for a general point $b'\in \partial\Delta$, namely, $f:U_{b'}\to \mathbb C^N$ is a holomorphic embedding such that $f|_{U_{b'}\cap \Delta}=\widetilde\mu|_{U_{b'}\cap \Delta}$.
For any $x\in S:=f(U_{b'}) \subset \mathfrak m^+$, we can also identify $T_x(S)\subset \mathfrak m^+$ as an affine linear subspace.
Now, for any $x\in Z=\widetilde\mu(\Delta)$ we define
$V_x:=E_x/\mathfrak p_x \subset \mathfrak g/\mathfrak p_x\cong \mathfrak m^+\cong T_x(\Omega)$,
where
\[ E_x:=\left\{[[\xi_1,\pi],\xi_2](x) \in \mathfrak g:\begin{split}
& \xi_j,\pi \in \mathfrak g,\; \xi_j(x)\,\mathrm{mod}\, \mathfrak p_x \in T_x(Z),\;j=1,2,\\
& \pi\;\text{vanishes to the order $\ge 2$ at $x$} \end{split}\right\}. \]
Replacing $T_x(\Omega)$ by $\mathfrak m^+$ and $Z$ by $S$, $V_x\subset \mathfrak m^+$ is also defined for any $x\in S$.
Let $F:=\mathfrak m^+\times \mathfrak m^+$ be the trivial vector bundle over $\mathfrak m^+$.

\begin{lem}\label{lem:NewBundleV}
Let the notation be as before. Suppose $\Omega$ is an irreducible bounded symmetric domain of rank $\ge 2$ which is not necessarily of tube type.
Defining $V:=\bigcup_{x\in Z} V_x \subseteq T_\Omega|_Z$, $V\subset T_\Omega|_Z$ is a holomorphic vector subbundle such that $T_x(Z)\subseteq V_x\subseteq W_x$ for any $x\in Z$.
Moreover, for a general point $b'\in \partial\Delta$, we have a local holomorphic extension $f$ of $\widetilde\mu$ to $U_{b'}:=\mathbb B^1(b',\epsilon)$, $\epsilon>0$, such that $V':=\bigcup_{x\in S} V_x \subset F|_S$ is a holomorphic vector bundle, where $S:=f(U_{b'})$.
\end{lem}
\begin{proof}
The existence of a local holomorphic extension $f$ of $\widetilde\mu$ around a general point $b'\in \partial\Delta$ follows from Mok \cite{Mo12}.
From the definition of $V_x$ for any $x\in Z$ (resp.\,$x\in S$), it follows readily that $V'\subset F|_S$ (resp.\,$V:=\bigcup_{x\in Z} V_x\subset T_\Omega|_Z$) is a holomorphic vector subbundle because $V_x$ varies holomorphically as $x$ varies on $S$ (resp.\,$Z$).

Note that $[\mathfrak k^{\mathbb C},\mathfrak m^+]\subset \mathfrak m^+$, $[\mathfrak k^{\mathbb C},\mathfrak m^-]\subset \mathfrak m^-$, $[\mathfrak m^+,\mathfrak m^-]\subset \mathfrak k^{\mathbb C}$,
$[\mathfrak m^+,\mathfrak m^+]=0$, $[\mathfrak m^-,\mathfrak m^-]=0$, and $\mathfrak m^-$ is the vector space of holomorphic vector fields on $X_c$ vanishing to the order $\ge 2$ at $x$.
Thus, for any $x\in Z$, $V_x=E_x/\mathfrak p_x$ is identical to $[[T_x(Z),\mathfrak m^-],T_x(Z)]$.
From now on we can identify $V_x=[[T_x(Z),\mathfrak m^-],T_x(Z)]$ for any $x\in Z$.
Thus, for any $x\in Z$ we have 
$V_x=[[T_x(Z),\mathfrak m^-],T_x(Z)]\subseteq [[W_x,\mathfrak m^-],W_x] \subseteq W_x$
by Lemma \ref{LemHVSB_CSD1} and \cite[Lemma 4.3]{Ts93}.

We may assume that $\eta$ is of rank $k$ and write $\eta=\sum_{j=1}^k \eta_j e_{\psi_j}$ in normal form, where $\eta_1\ge \cdots \ge \eta_k>0$.
From Lie theory we have $[[e_{\varphi},e_{-\varphi}],e_{\varphi}]
=\varphi(H_{\varphi}) e_{\varphi}$ with $\varphi(H_{\varphi})\neq 0$, $[e_{\varphi},e_{-\varphi}]=H_\varphi$ for any $\varphi\in \Delta_M^+$, and $[H_{\psi_j},e_{\pm \psi_i}]=0$ for distinct $i,j$, $1\le i,j\le r=\mathrm{rank}(\Omega)$.
Thus, we have 
\[ \begin{split}
[[\eta,e_{-\psi_j}],\eta] 
=& \sum_{s,t=1}^k \eta_s \eta_t [[e_{\psi_s},e_{-\psi_j}],e_{\psi_t}]
= \sum_{t=1}^k \eta_j \eta_t [[e_{\psi_j},e_{-\psi_j}],e_{\psi_t}]\\
=& \sum_{t=1}^k \eta_j \eta_t [H_{\psi_j},e_{\psi_t}]
= \eta_j^2 [H_{\psi_j},e_{\psi_j}]\\
=& \eta_j^2 \psi_j(H_{\psi_j}) e_{\psi_j}.
\end{split}\]
Since $\sigma_j:=\eta_j^2 \psi_j(H_{\psi_j})\neq 0$ for $1\le j\le k$, we choose $\alpha_j={1\over \sigma_j}e_{-\psi_j}\in \mathfrak m^-$ and we have $ e_{\psi_j}=[[\eta,\alpha_j],\eta]\in V_x$ for $1\le j\le k$.
Therefore, we have $\eta=\sum_{j=1}^k \eta_j e_{\psi_j}\in V_x$.
In particular, $T_x(Z)=\mathbb C \eta \subset V_x$ and we conclude that $T_x(Z)\subseteq V_x \subseteq W_x$ for any $x\in Z$, as desired.
\end{proof}

From Lemma \ref{lem:NewBundleV} we have $T_x(Z)\subseteq V_x \subseteq W_x$ for any $x\in Z=\widetilde\mu(\Delta)$. But our goal here is to construct a holomorphic vector bundle which extends the definition of $W$ to some open neighborhood of a general point on the unit circle $\partial\Delta$.
Therefore, when $\Omega$ is of tube type, we will show that $V_x=W_x$ for any $x\in Z$ and thus $V:=\bigcup_{x\in Z}V_x=W$.
Then, we will have the local extension $V'=\bigcup_{x\in S} V_x$ of $V=W$, where $S:=f(U_{b'})$ as in the above.
Recall that $T_x(Z)$ is spanned by a rank-$k$ vector $\eta_x$ for any $x\in Z$ and $1\le k\le \mathrm{rank}(\Omega)$. In the notation as above we have

\begin{lem}\label{ExtW}
Suppose $\Omega$ is an irreducible bounded symmetric domain of rank $r \ge 2$ which is not necessarily of tube type, and $\eta_x$ is of rank $k<r$ for any $x\in Z$. Then, we have $V_x=W_x$ for any $x\in Z$.
\end{lem}
\begin{proof}
Since we have $V_x\subseteq W_x$ for any $x\in Z$ by Lemma \ref{lem:NewBundleV}, it remains to show that $W_x\subseteq V_x$ for any $x\in Z$.
In what follows we simply write $\eta=\eta_x$ in normal form. Recall that $V_x=[[T_x(Z),\mathfrak m^-],T_x(Z)]$ and $T_x(Z)=\mathbb C \eta$. We will also make use of the normal form of $W_x$ as in the proof of Lemma \ref{LemHVSB_CSD1} and Remark \ref{RemHVSB_CSD1}.

\medskip\noindent(1)
Consider the case where $\Omega\cong D^{\mathrm{I}}_{p,q}$, $q\ge p=r\ge 2$.
Then, we have
\[ \mathfrak m^+ = \left\{ \begin{bmatrix}
0 & A\\
0 & 0
\end{bmatrix}\in M(p+q,p+q;\mathbb C): A\in M(p,q;\mathbb C)\right\}\]
and
\[ \mathfrak m^- = \left\{\begin{bmatrix}
0 & 0 \\
B & 0
\end{bmatrix} \in M(p+q,p+q;\mathbb C): B\in M(q,p;\mathbb C)\right\}.\]
Write $\eta = \begin{bmatrix}
0 & \eta' \\
0 & 0
\end{bmatrix} \in \mathfrak m^+$ in normal form so that $$\eta'=\mathrm{diag}_{p,q}(\eta_1,\ldots,\eta_k,0,\ldots,0)$$ and $\eta_1\ge \cdots \ge \eta_k>0$.
For any $\beta=\begin{bmatrix}
0 & 0 \\
B & 0
\end{bmatrix}\in \mathfrak m^-$ we have
$[\eta,\beta]=\eta\beta-\beta\eta
= \begin{bmatrix}
\eta'B & 0\\
0 & -B\eta'
\end{bmatrix}$
and thus
$$[[\eta,\beta],\eta]
= \begin{bmatrix}
\eta'B & 0\\
0 & -B\eta'
\end{bmatrix}\eta-\eta\begin{bmatrix}
\eta'B & 0\\
0 & -B\eta'
\end{bmatrix}
= \begin{bmatrix}
0 & 2 \eta' B \eta'\\
0 & 0
\end{bmatrix}.$$
Writing $B=\begin{bmatrix}
B_1&B_2\\
B_3&B_4
\end{bmatrix}$ so that $B_1\in M(k,k;\mathbb C)$ and $$\eta''=\mathrm{diag}_{k,k}(\eta_1,\ldots,\eta_k)\in M(k,k;\mathbb C),$$ we have
$2 \eta' B \eta'
=  \begin{bmatrix}
2\eta'' B_1 \eta'' & 0 \\
0 & 0
\end{bmatrix}$.
Note that $\eta''$ is invertible because $\det\eta''=\prod_{j=1}^k \eta_j \neq 0$.
Thus, for any $\gamma=\begin{bmatrix}
0 & A\\
0 & 0
\end{bmatrix} \in W_x\subset T_x(\Omega)\cong\mathfrak m^+$ so that $A=\begin{bmatrix}
A' & 0\\
0 & 0
\end{bmatrix}$ for some $A'\in M(k,k;\mathbb C)$ (by the proof of Lemma \ref{LemHVSB_CSD1}), we may choose $\beta:=\begin{bmatrix}
0 & 0 \\
B & 0
\end{bmatrix}\in \mathfrak m^-$ such that $B=\begin{bmatrix}
B_1&B_2\\
B_3&B_4
\end{bmatrix}$ with $B_1 = {1\over 2} (\eta'')^{-1}A'(\eta'')^{-1}\in M(k,k;\mathbb C)$.
Then, from the above computations we have
$[[\eta,\beta],\eta] = \gamma$. Hence, $W_x\subseteq V_x$ so that $V_x=W_x\cong M(k,k;\mathbb C)$.

\medskip\noindent (2) 
When $\Omega\cong D^{\mathrm{II}}_{2r}$ or $D^{\mathrm{II}}_{2r+1}$, $r\ge 2$, we may replace $\eta''$ by $\hat\eta$ $=$ $\mathrm{diag}_{2k,2k}(\eta_1J_1,\dots,\eta_kJ_1)$ $\in$ $M_a(2k;\mathbb C)$, $\eta_1\ge \cdots \ge \eta_k>0$, in the above, where $J_1=\begin{pmatrix}
0 & 1\\
-1 &0
\end{pmatrix}$.
Then, $\hat\eta$ is invertible and $B_1$ $=$ ${1\over 2} \hat\eta^{-1}A'\hat\eta^{-1}$ $\in$ $M(2k,2k;\mathbb C)$ is antisymmetric whenever $A'$ is antisymmetric. Hence, we also obtain $W_x\subseteq V_x$ so that $V_x=W_x\cong M_a(2k;\mathbb C)$.

\medskip\noindent (3) Consider the case where $\Omega\cong D^{\mathrm{III}}_p$, $p\ge 2$.
By restricting to the space of $p$-by-$p$ symmetric matrices, we also have $W_x\subseteq V_x$ in this case by the same arguments in the above.
This is because $\eta''$ is a diagonal matrix and thus $B_1 = {1\over 2} (\eta'')^{-1}A'(\eta'')^{-1}\in M(k,k;\mathbb C)$ is symmetric whenever $A'$ is symmetric.
Hence, we have $V_x=W_x\cong M_s(k;\mathbb C)$.

\medskip For any irreducible bounded symmetric domain $\Omega$ of rank $\ge 2$, if $\eta$ is of rank $k=1$, then from the proof of Lemma \ref{LemHVSB_CSD1} we already have $T_x(Z)=W_x\cong T_x(\Delta_\eta)$ and thus $V_x=W_x\cong T_x(\Delta_\eta)$ by the fact that $T_x(Z)\subseteq V_x\subseteq W_x$, where $\Delta_\eta=\mathbb C\eta\cap \Omega$ is a minimal disk of $\Omega$.

\medskip\noindent (4) For $\Omega\cong D^{\mathrm{IV}}_N$, $N\ge 3$, or $\Omega\cong D^{V}$, we have $r=\mathrm{rank}(\Omega)=2$. Then, $\eta$ is of rank $k=1$ and the result follows from the last paragraph.

\medskip\noindent (5)
Finally, we consider the case where $\Omega\cong D^{\mathrm{VI}}$, which is of rank $3$. We are done if $\eta$ is of rank $1$ as in the above. Thus, it remains to show that $W_x\subseteq V_x$ when $\eta$ is of rank $k=2$.
We will make use of the data obtained from \cite{Zh84} as in the proof of Lemma \ref{LemHVSB_CSD1}.
When $\eta$ is of rank $2$, we may assume $\eta=\eta_1e_{x_1-x_2}+\eta_2e_{x_1+x_2+x_3}$ with $\eta_1\ge \eta_2>0$.
From direct computation, for each $\varphi\in \Delta_M^+\smallsetminus \Psi$, if $[[e_{x_1-x_2},e_{-\varphi}],e_{x_1+x_2+x_3}]\neq 0$, then $[[e_{x_1-x_2},e_{-\varphi}],e_{x_1+x_2+x_3}]$ is a nonzero scalar multiple of one of the $e_{x_1-x_j}$, $4\le j\le 7$, and $e_{x_1+x_3+x_i}$, $4\le i\le 7$. Moreover, recall that $[[e_{\psi},e_{-\psi}],e_\psi]$ is a nonzero scalar multiple of $e_\psi$ for $\psi\in \Psi$.
Write $\Psi=\{\psi_1,\psi_2,\psi_3\}$ with $\psi_1=x_1-x_2,\;\psi_2=x_1+x_2+x_3$ and $\psi_3=d-x_3$.
From the Jacobi identity and $[\mathfrak m^+,\mathfrak m^+]=0$, we also have $[[e_{\psi_1},e_{-\varphi}],e_{\psi_2}]=[[e_{\psi_2},e_{-\varphi}],e_{\psi_1}]$.
For $1\le j\le 3$, we have $[[e_{\psi_j},e_{-\varphi}],e_{\psi_j}]=0$ for any $\varphi\in \Delta_M^+\smallsetminus\{\psi_j\}$ because $\psi_j-\psi_i$ is not a root for $i\neq j$ and $2\psi_j-\varphi$ is not a root whenever $\varphi\in \Delta_M^+\smallsetminus \Psi$.
Fixing any $\varphi\in \{x_1-x_j,\;4\le j\le 7,\;x_1+x_3+x_i,\; 4\le i\le 7\}\subset \Delta_M^+$, we have shown that $e_{\varphi}=c_{\varphi}[[e_{\psi_1},e_{-\phi}],e_{\psi_2}]$ for some $\phi\in \Delta_M^+\smallsetminus \Psi$ and some scalar constant $c_\varphi\neq 0$.
Thus, we have
\[\begin{split}
 [[\eta,e_{-\phi}],\eta]
=& \sum_{s,t=1}^2 \eta_s\eta_t [[e_{\psi_s},e_{-\phi}],e_{\psi_t}]
= 2\eta_1\eta_2 [[e_{\psi_1},e_{-\phi}],e_{\psi_2}]\\
=& 2\eta_1\eta_2 {1\over c_{\varphi}} e_{\varphi}
\end{split}\]
so that $[[\eta,\beta],\eta]=e_{\varphi}$, where $\beta:={c_\varphi\over 2\eta_1\eta_2}e_{-\phi}\in \mathfrak m^-$.
Since $V_x=[[T_x(Z),\mathfrak m^-],T_x(Z)]$ and $T_x(Z)=\mathbb C\eta$, $V_x$ contains the $\mathbb C$-linear span of $e_{x_1-x_j}$, $4\le j\le 7$, $e_{x_1+x_3+x_i}$, $4\le i\le 7$, $e_{x_1-x_2}$ and $e_{x_1+x_2+x_3}$.
On the other hand, we have $W_x=\mathcal N_{e_{d-x_3}}$ as in the proof of Lemma \ref{LemHVSB_CSD1}, where $d:=\sum_{j=1}^7x_j$.
By direct computation $\mathcal N_{e_{d-x_3}}=\mathcal N_{e_{x_1-x_2}}\cap \mathcal N_{e_{x_1+x_2+x_3}}$ is the $\mathbb C$-linear span of $e_{x_1-x_j}$, $4\le j\le 7$, $e_{x_1+x_3+x_i}$, $4\le i\le 7$, $e_{x_1-x_2}$ and $e_{x_1+x_2+x_3}$. 
Hence, we have $W_x\subseteq V_x$ and thus $W_x=V_x$, as desired.
\end{proof}

\begin{rem}
\rm{
\begin{enumerate}
\item[(a)] It is evident from the arguments in the proof of Lemma \ref{ExtW} that analogue of the statement of Lemma \ref{ExtW} holds true when $\Omega$ is of tube type and of type $\mathrm{I}$, $\mathrm{II}$ or $\mathrm{III}$, and $\eta$ is of rank $r=\mathrm{rank}(\Omega)$.
\item[(b)] When $\Omega$ is of type $\mathrm{IV}$ or of type $\mathrm{VI}$, the analogue of the statement of Lemma \ref{ExtW} also holds true when $\eta$ is of rank $r=\mathrm{rank}(\Omega)$.
More precisely, when $\Omega$ is of type $\mathrm{IV}$, we have $\mathrm{rank}(\Omega)=2$ and the result follows from Tsai \cite[Proof of Lemma 5.2]{Ts93}. When $\Omega$ is of type $\mathrm{VI}$, we have $\mathrm{rank}(\Omega)=3$ and the result follows from explicit computation by taking Lie brackets of root vectors.
\end{enumerate}
}
\end{rem}

\subsubsection{Estimates on the Kobayashi metric and the Kobayashi distance and vanishing of the second fundamental form}
By applying the rescaling argument to the local holomorphic extension of the holomorphic isometry $\widetilde\mu$ around a general point $b'\in \partial\Delta$ as in Proposition \ref{ProConstHI1}, we can obtain another holomorphic isometry from $(\Delta,m_0g_\Delta)$ to $(\Omega,g_\Omega)$ satisfying the two properties in Proposition \ref{ProConstHI1}. We will still denote such a holomorphic isometry by $\widetilde\mu$ and its image by $Z$. Then, we may construct the vector subbundle $W\subset T_\Omega|_Z$ and the holomorphic vector bundle $V$ over $Z$ for the holomorphic curve $Z$ as we have done before.
By the same arguments, the statements of Lemmas \ref{LemHVSB_CSD1}, \ref{LemHVSB_W1}, \ref{TauTensorHolo1}, \ref{lem:NewBundleV} and \ref{ExtW} hold true.
\begin{lem}\label{LemTensorV}
Under the above assumptions, for any $x\in Z$ and $\eta, \beta \in \Gamma_{\mathrm{loc},x}(Z,T_Z)$, we have $\tau_x(\eta(x)\varotimes \beta(x)) = 0$, i.e., $(\nabla_\eta \beta)(x)\in W_x$, equivalently $\tau|_{T_Z\varotimes T_Z}\equiv 0$.
\end{lem}
\begin{proof}
By Lemma \ref{TauTensorHolo1} we may regard $\tau|_{T_Z\varotimes T_Z}$ as a holomorphic section $\hat\tau$ $\in$ $\Gamma$($Z$, $S^2T_Z^*\varotimes (T_\Omega|_Z/W)$).
Let $\nu_k=\epsilon_k \;\rm{mod}\; W$ be holomorphic basis of the quotient bundle $T_\Omega|_Z/W$, namely,
$$\nu_k(\zeta)=\epsilon_k(\zeta)\;\rm{mod}\; W_{\widetilde\mu(\zeta)},$$
where $\epsilon_k(\zeta) = {\partial\over \partial z_k}\big|_{z=\widetilde \mu(\zeta)}$.
By identifying $Z=\widetilde\mu(\Delta)\cong \Delta$ we may write
$\hat\tau(\zeta) = \sum_k \tau_{11}^k(\zeta) d\zeta\varotimes d\zeta \varotimes \nu_k(\zeta)$.
Then, we have
\[ \lVert \hat\tau(\zeta) \rVert
\le \sum_k |\tau_{11}^k(\zeta)| \lVert d\zeta \rVert^2 \lVert \nu_k(\zeta)\rVert. \]
Since $\widetilde\mu$ is a holomorphic isometry, we have $\left\lVert \widetilde \mu'(\zeta) \right\rVert^2_{g_\Omega}
=\left\lVert {\partial\over \partial \zeta} \right\rVert^2_{m_0g_\Delta}= {m_0\over (1-|\zeta|^2)^2}$.
Thus, we have $\lVert d\zeta \rVert \le C''\cdot \delta(\zeta)$ for some real constant $C''>0$, where $\delta(\zeta):=1-|\zeta|$.
Moreover, we have $\lVert\nu_k(\zeta)\rVert\le \lVert \epsilon_k(\zeta)\rVert_{g_\Omega}$ (cf.\,\cite{Mo10}).
We claim that
\[ \lVert \epsilon_k(\zeta)\rVert_{g_\Omega} \le C'{1\over \delta(\zeta)} \]
for some positive real constant $C'$.
The idea is to use the Kobayashi distance, the Kobayashi metric on $\Omega$, and the convexity of $\Omega$.
Denote by $d_\Delta(\cdot,\cdot)$ (resp.\,$d_\Omega(\cdot,\cdot)$) the Kobayashi distance on $\Delta$ (resp.\,$\Omega$) with $d_\Delta(0,\zeta)=\log{1+|\zeta|\over 1-|\zeta|}$ and $d_\Delta(\cdot,\cdot)$ is defined by using the Bergman metric $ds_\Delta^2$ on $\Delta$ (cf.\,\cite{Ko98}).
From \cite{Ko98}, for a complex manifold $M$ we define the Kobayashi pseudo-metric by
\[ F_{M}(v)
=\inf \left\{\lVert \hat v \rVert_{ds_\Delta^2}: \hat v \in T_0(\Delta),\;f\in \mathrm{Hol}(\Delta,M),\;f(0)=x,\;f_*\hat v = v
\right\} \]
for $v\in T_x(M)$, $x\in M$.
For $x\in \Omega$, let $\delta_\Omega(x)=
\delta(x,\partial\Omega)$ be the Euclidean distance from $x$ to the boundary $\partial\Omega$.
Note that ${1\over \sqrt{2}}F_{\mathbb B^N}(\xi)
= \lVert \xi \rVert_{g_{\mathbb B^N}}$.
Fix $x\in \Omega$.
By the definition of $\delta_\Omega(x)=\delta(x,\partial\Omega)$, we have $\mathbb B^N(x,\delta_\Omega(x))\subseteq \Omega$ and thus we have a holomorphic map $f:\mathbb B^N\to \Omega$ given by $f(w) = \delta_\Omega(x) w + x$.
Then, $f$ maps $\mathbb B^N$ biholomorphically onto $\mathbb B^N(x,\delta_\Omega(x))$ and
$df_{\bf 0} \left({1\over \delta_\Omega(x)}{\partial\over \partial w_j}\big|_{\bf 0}\right)
= {\partial\over \partial z_j}\big|_{x}$.
For $v=\epsilon_j(\zeta)={\partial\over \partial z_j}\big|_{\widetilde\mu(\zeta)}\in T_{\widetilde\mu(\zeta)}(\Omega)$, by the Ahlfors-Schwarz lemma and \cite[p.\,90]{Ko98} there is a positive real constant $C_0$ (independent of the choice of vectors tangent to $\Omega$) such that
\[ \begin{split}
\left\lVert v\right \rVert_{g_\Omega}
&\le C_0 F_{\Omega}(v)
\le  C_0 F_{\mathbb B^N}\left({1\over \delta_\Omega(x)}{\partial\over \partial w_j}\bigg|_{\bf 0}\right)\\
&=\sqrt{2} C_0\left\lVert{1\over \delta_\Omega(x)}{\partial\over \partial w_j}\bigg|_{\bf 0} \right\rVert_{g_{\mathbb B^N}}
={\sqrt{2}C_0\over \delta_\Omega(x)},
\end{split}\]
where $x=\widetilde \mu(\zeta)$.
In particular, there is a positive real constant $C$ such that
$\left\lVert \epsilon_j(\zeta)\right\rVert_{g_\Omega}$
$\le$ $C{1\over \delta_\Omega(\widetilde\mu(\zeta))}$
for $1\le j\le N$ and $\zeta\in \Delta$.
Since $\Omega\Subset \mathbb C^N$ is convex, it follows from \cite[Proposition 2.4]{Me93} that there is $C_1\in \mathbb R$ such that
$C_1- {1\over 2} \log \delta_\Omega(z) \le {1\over 2} d_\Omega({\bf 0},z)$
for any $z\in \Omega$ (cf.\,Remark below).
(Noting that Mercer \cite{Me93} defined the Kobayashi distance to be $k_\Omega(\cdot,\cdot)={1\over 2}d_\Omega(\cdot,\cdot)$.)
Then, we have
$e^{-2C_1} \delta_\Omega(z) \ge e^{-d_\Omega({\bf 0},z)}$
so that
\[ \delta(\zeta) \le 2\cdot e^{-d_\Delta(0,\zeta)}
\le 2 \cdot e^{-d_\Omega({\bf 0},\widetilde\mu(\zeta))} \le 2 e^{-2C_1} \cdot \delta_\Omega(\widetilde\mu(\zeta)). \]
It follows that $\lVert \epsilon_j(\zeta)\rVert_{g_\Omega}
\le C {1\over \delta_\Omega(\widetilde\mu(\zeta))}
\le C' {1\over \delta(\zeta)}$
for $\zeta\in \Delta$ and $1\le j\le N$, where $C'$ is some positive real constant. The claim is proved.
Thus, we have
\[ \lVert \hat\tau(\zeta) \rVert
\le \hat C \,\delta(\zeta) \cdot \sum_k |\tau_{11}^k(\zeta)| \]
for some positive real constant $\hat C$.
Here the summation in the above inequality is a finite sum.
By Lemma \ref{lem:NewBundleV} and Lemma \ref{ExtW} we can extend the definition of $\hat\tau$ to an open neighborhood of a general point on $\partial\Delta$.
Thus, $\lVert \hat\tau(\zeta) \rVert^2$ can be extended as a real-analytic function on some open neighborhood $U_{b'}$ of a general point $b'\in \partial\Delta$ in $\mathbb C$ (by Lemma \ref{LemQuotRA1}) 
and each $|\tau_{11}^k(\zeta)|$ is bounded from above by a uniform positive real constant on $U_{b'}$.
In particular, we have $\lVert \hat\tau(\zeta) \rVert\to 0$ as $\zeta\to b''$ for any $b''\in U_{b'}\cap\partial\Delta$.
Hence, we have $\lVert \hat\tau(\zeta) \rVert\to 0$ as $\zeta \to b'$ for a general point $b'\in \partial\Delta$.

From Mok \cite{Mo12} we have a local holomorphic extension $F$ of the holomorphic isometry $\widetilde\mu$ around any general point $\hat b\in \partial\Delta$.
By applying the rescaling argument to $F$ as in Proposition \ref{ProConstHI1} and choosing a good boundary point $\hat b\in \partial\Delta$, we can obtain another holomorphic isometry from $(\Delta,m_0g_\Delta)$ to $(\Omega,g_\Omega)$, still denoted by $\widetilde\mu$ for simplicity, such that the following hold true.
\begin{enumerate}
\item[(a)] $\widetilde\mu$ satisfies the two properties in Proposition \ref{ProConstHI1}.
\item[(b)] Constructing the vector subbundle $W\subset T_\Omega|_Z$ over the holomorphic curve $Z:=\widetilde\mu(\Delta)$ as we have done before, the statements of Lemmas \ref{LemHVSB_CSD1}, \ref{LemHVSB_W1} and \ref{TauTensorHolo1} hold true by the same arguments as in the corresponding proofs.
\item[(c)] For the holomorphic section $\hat\tau\in \Gamma(Z,S^2T_Z^*\varotimes (T_\Omega|_Z/W))$ representing $\tau|_{T_Z\varotimes T_Z}$ over the (new) holomorphic curve $Z$, $\lVert \hat\tau(\zeta) \rVert^2$ extends real-analytically around a general point in $\partial\Delta$ and that $\lVert \hat\tau(\zeta) \rVert\to 0$ as $\zeta \to b'$ for a general point $b'\in \partial\Delta$ by the above arguments.
\end{enumerate}
By the analogous arguments in the proof of Proposition \ref{ProConstHI1} for showing that $\lVert \widetilde\sigma\rVert^2\equiv\text{constant}$, we may also obtain that $\lVert \hat\tau(\zeta) \rVert^2$ is identically constant on $\Delta$. Together with part (c) in the above, we have $\lVert \hat \tau(\zeta) \rVert^2\equiv 0$ on $\Delta$, i.e., $\tau|_{T_{Z}\varotimes T_{Z}}(\zeta)\equiv 0$ on $\Delta$.
\end{proof}
\begin{rem}\rm{
For any bounded symmetric domain $\Omega$ the inequality from \cite[Proposition 2.4]{Me93} can be derived using the Polydisk Theorem.
We refer the readers to the Appendix (i.e., Section \ref{App}) of the current article.}
\end{rem}

\begin{lem}\label{LemVSF1}
In the above construction, we have $\tau\equiv 0$.
\end{lem}
\begin{proof}
By Lemma \ref{LemTensorV}, we have $\tau|_{T_Z\varotimes T_Z}\equiv 0$, i.e., $(\nabla_\eta \hat\eta)(x)\in W_x$ for any $\eta\in T_x(Z)$, $\hat\eta\in \Gamma_{\mathrm{loc},x}(Z,T_Z)$ and $x\in Z$.
Note that $R_{\eta\overline\zeta\alpha\overline\beta}=0$ for $\eta\in T_x(Z)$, $\zeta\in \mathcal N_\eta$, and any $\alpha,\beta\in T_x(\Omega)$, where $x\in Z$.
From the definition of $W$, we have
$R(\nabla_\eta\hat\eta,\overline\zeta,\alpha,\overline\beta)=0$, because $\gamma\in \Gamma(Z,W)$ if and only if for any $x\in Z$ we have $R_{\gamma(x)\overline\zeta\alpha\overline\beta}=0$ for any $\alpha,\beta \in T_x(\Omega)$ and any $\zeta \in \mathcal N_{\eta}$, where $\eta\in T_x(Z)$.
Thus, for any $x\in Z$ we have
$R(\eta,\overline{(\nabla_{\overline\eta}\zeta)(x)},\alpha,\overline\beta)=0$ for any $\alpha,\beta \in T_x(\Omega)$.
In particular, $(\nabla_{\overline\eta}\zeta)(\widetilde\mu(w))\in \mathcal N_{\eta(w)}$ for any $w\in \Delta$.
For any $\gamma\in \Gamma_{\mathrm{loc},x}(Z,W)$, $\zeta\in \mathcal N_\eta$ and any $\alpha,\beta \in \Gamma_{\mathrm{loc},x}(Z,T_\Omega|_Z)$, we have $R_{\gamma\overline\zeta \alpha\overline\beta}=0$ so that
\[ R(\nabla_\eta\gamma,\overline\zeta,\alpha,\overline\beta)+R(\gamma,\overline{\nabla_{\overline\eta}\zeta},\alpha,\overline\beta) =0. \]
Since $(\nabla_{\overline\eta}\zeta)(\widetilde\mu(w))\in \mathcal N_{\eta(w)}$, we have
$R((\nabla_\eta\gamma)(\widetilde\mu(w)),\overline{\zeta},\alpha,\overline{\beta}) = 0$
for arbitrary $\zeta\in \mathcal N_{\eta(w)}$, $\alpha,\beta\in T_{\widetilde\mu(w)}(\Omega)$.
Therefore, $(\nabla_\eta\gamma)(\widetilde\mu(w)) \in W_{\widetilde\mu(w)}$ for arbitrary $w\in \Delta$, i.e., $\tau\equiv 0$.
\end{proof}
\begin{lem}\label{LemInCSD1}
In the above construction, we have $Z=\widetilde\mu(\Delta)\subset \Omega'$ for some characteristic subdomain $\Omega'\subseteq \Omega$ of rank $k$.
\end{lem}
\begin{proof}
From the above construction, $T_x(Z)$ is spanned by a rank-$k$ vector $\eta(w)$ at any point $x=\widetilde\mu(w)\in Z$ ($w\in \Delta$) and there is a holomorphic vector subbundle $W\subset T_\Omega|_Z$ with $T_Z\subset W\subset T_\Omega|_Z$.
We first show that there is a characteristic subdomain $\Omega' \subset \Omega$ of rank $k$ such that $Z$ is tangent to $\Omega'$ to the order at least $2$ at some point $\widetilde\mu(w_0)$, $w_0\in \Delta$, and $T_{\widetilde\mu(w_0)} (\Omega') = W_{\widetilde\mu(w_0)}$.
By considering the normal form of $W_{\widetilde\mu(w_0)}$, it is clear that there is a characteristic subdomain $\Omega'\subset \Omega$ of rank $k$ such that $\widetilde\mu(w_0)\in \Omega'$ and $T_{\widetilde\mu(w_0)} (\Omega') = W_{\widetilde\mu(w_0)}$.
Moreover, such $\Omega'$ is unique for each fixed $w_0$.
Actually, if there is a characteristic subdomain $\Omega'' \subset \Omega$ such that $\widetilde\mu(w_0)\in \Omega''$ and $T_{\widetilde\mu(w_0)} (\Omega'') = W_{\widetilde\mu(w_0)}$, then by choosing some $\Phi \in \Aut(\Omega)$ with $\Phi(\widetilde\mu(w_0)) = {\bf 0}$, both $\Phi(\Omega')$ and $\Phi(\Omega'')$ are linear sections of $\Omega$ by complex vector subspaces in $\mathbb C^N\cong \mathfrak m^+$.
But then their holomorphic tangent spaces at ${\bf 0}$ coincide to each other so that $\Phi(\Omega')=\Phi(\Omega'')$, i.e., $\Omega'=\Omega''$.
Since $\tau\equiv 0$ by Lemma \ref{LemVSF1}, we have $(\nabla_\eta \gamma)(\widetilde\mu(w)) \in W_{\widetilde\mu(w)}$ for any $w\in \Delta$, where $\eta\in T_{\widetilde\mu(w)}(Z)$ and $\gamma\in \Gamma_{\mathrm{loc},\widetilde\mu(w)}(Z,W)$.

Denote by $\pi:\mathbb G(T_\Omega, n_{r-k}(\Omega))\to \Omega$ the Grassmann bundle, where $\mathbb G$($T_x(\Omega)$, $n_{r-k}(\Omega)$) is the Grassmannian of the complex $n_{r-k}(\Omega)$-dimen- sional vector subspaces of $T_x(\Omega)$ for each $x\in \Omega$.
From \cite[p.\,99]{MT92}, we let $\mathcal{NS}_{r-k}(\Omega)$ be the collection of all $n_{r-k}(\Omega)$-planes which are holomorphic tangent spaces to the $k$-th characteristic subdomains of $\Omega$.
Then, $\mathcal{NS}_{r-k}(\Omega)$ lies in the Grassmann bundle $\mathbb G(T_\Omega, n_{r-k}(\Omega))$ and is a holomorphic fiber bundle over $\Omega$ with $\mathcal{NS}_{r-k}(\Omega)\cong \mathcal{NS}_{r-k,{\bf 0}}(\Omega)\times \Omega$.
For each $x\in \Omega$ and each $k$-th characteristic subdomain $\Omega'_x\subset \Omega$ containing $x$, we can lift $\Omega'_x$ to $\mathcal{NS}_{r-k}(\Omega)$ as
\[ \widehat{\Omega'_x} = \{ [T_y(\Omega')]\in \mathcal{NS}_{r-k,y}(\Omega): y\in \Omega'_x\}. \]
Such a lifting of $k$-th characteristic subdomains of $\Omega$ forms a tautological foliation $\mathscr F$ on $\mathcal{NS}_{r-k}(\Omega)$ with $n_{r-k}(\Omega)$-dimensional leaves $\widehat{\Omega'_x}$.
Then, we let $\hat Z$ be the tautological lifting of $Z$ to $\mathcal{NS}_{r-k}(\Omega)$ defined by
\[ \hat Z = \{ [W_x]\in \mathcal{NS}_{r-k,x}(\Omega) : x\in Z \}. \]
Since $(\nabla_\eta \gamma)(\widetilde\mu(w_0))\in W_{\widetilde\mu(w_0)}$, $\hat Z$ is tangent to $\widehat{\Omega'}$ at $[W_{\widetilde\mu(w_0)}]$.
Actually, since $(\nabla_\eta \gamma)(x)\in W_x$ for any $x\in Z$, $\hat Z$ is tangent to the leaf $\widehat{\Omega_x'}$ of $\mathscr F$ at $[W_x]$ for any $x\in Z$, where $\Omega'_x \subset \Omega$ is the characteristic subdomain of rank $k$ at $x$ satisfying $T_x(\Omega'_x)=W_x$.
Therefore, $\hat Z$ is an integral curve of the integrable distribution defined by the foliation $\mathscr F$.
From the general theory of foliations, such an integral curve of the distribution induced from $\mathscr F$ must lie inside the single leaf $\widehat{\Omega'}$ of $\mathscr F$, which is also the maximal integral submanifold of the induced integrable distribution.
Therefore, $\hat Z$ itself should lie inside the leaf $\widehat{\Omega'}$ of the foliation $\mathscr F$ because $\hat Z$ is path connected.
Note that $Z$ is the image of $\hat Z$ under the canonical projection $\mathbb G(T_\Omega,n_{r-k}(\Omega)) \to \Omega$.
But then the above argument shows that $Z$ should lie in $\Omega'$ because $\hat Z\subset \widehat{\Omega'}$.
\end{proof}

Let $\Omega$ be an irreducible bounded symmetric domain of rank $r\ge 2$ which is not necessarily of tube type.
Recall that any invariantly geodesic submanifold of $\Omega$ is an irreducible bounded symmetric domain of rank $\le r$.
From the results in Section \ref{Sec:CHIE} and in this section, we have
\begin{pro}\label{pro:redIGM}
Let $\Omega$ be an irreducible bounded symmetric domain of rank $r\ge 2$ which is not necessarily of tube type, and $Z=\widetilde\mu(\Delta)$ be constructed as above.
Assume that the tangent vector $\eta_x$ spanning $T_x(Z)$ is of rank $k$, $1\le k<r$.
Then, there is an invariantly geodesic submanifold $\Omega'\subset \Omega$ such that $\Omega'$ is a rank-$k$ irreducible bounded symmetric domain of tube type and $Z\subseteq \Omega'$. In particular, $\eta_x\in T_x(\Omega')$ is a rank-$k$ tangent vector in $T_x(\Omega')$ for $x\in Z$.
\end{pro}

\subsubsection{The Poincar\'e-Lelong equation and proof of Theorem \ref{ThmTubeDomain1}}
From the above construction and lemmas, we can complete the proof of Theorem \ref{ThmTubeDomain1}, as follows.
\begin{proof}[Proof of Theorem \ref{ThmTubeDomain1}]
From the holomorphic embedding $\mu:U\to \mathbb C^N$, by choosing an arbitrary general point $b\in U\cap\partial\Delta$ we have constructed in Proposition \ref{ProConstHI1} a holomorphic isometry $\widetilde \mu:(\Delta,m_0g_\Delta)\to (\Omega,g_\Omega)$
such that $\widetilde \mu(0)={\bf 0}$,
$\lVert \widetilde \sigma(\widetilde\mu(w))\rVert^2 \equiv \lVert\sigma(\mu(b))\rVert^2$ on $\Delta$ and
the normal form of ${\widetilde \mu'(w)\over \lVert \widetilde \mu'(w)\rVert_{g_\Omega}}$ is independent of $w\in \Delta$ and of rank $k$, where $k$ is some integer satisfying $1\le k\le r=\rank(\Omega)$.
By Lemma \ref{LemInCSD1}, $Z=\widetilde\mu(\Delta)$ lies inside a characteristic subdomain $\Omega'\subseteq \Omega$ of rank $k$. When $k=r=\rank(\Omega)$, we have $\Omega'=\Omega$.
Note that $\Omega$ is of tube type, so $\Omega'$ is also of tube type.
Denote by $\sigma'(x)$ the second fundamental form of $(Z,g_{\Omega'}|_Z)$ in $(\Omega',g_{\Omega'})$ at $x\in Z$, where the K\"ahler metric $g_{\Omega'}=g_{\Omega}|_{\Omega'}$ on $\Omega'$ is precisely the restriction of $g_{\Omega}$ to $\Omega'$.

We write $\Omega'=G_0'/K'$ and fix an arbitrary point $w\in \Delta$.
If $\widetilde\mu'(w)$ is a rank-$k'$ vector in $T_{\widetilde\mu(w)}(\Omega')$, then by applying the $K'$-action, the normal form of $\widetilde \mu'(w)$ is tangent to some totally geodesic polydisk $\Pi_{k'}\cong \Delta^{k'}$ in the maximal polydisk $\Pi_k\cong \Delta^k$ of $\Omega'$, which also lies in $\Delta^r\cong\Pi\subset \Omega$.
This implies that the normal form of $\widetilde \mu'(w)$ as a tangent vector in $T_{\widetilde \mu(w)}(\Omega)$ is of rank $k'$.
Therefore, we have $k=k'$ and $\widetilde \mu'(w)$ is a generic vector in $T_{\widetilde\mu(w)}(\Omega')$ for $w\in \Delta$.

Since $\Omega'$ is of tube type, it follows from \cite[Proposition 1]{Mo02} that the $(k-1)$-th characteristic bundle $\mathcal S_{k-1}(\Omega')$ of $\Omega'$ is of codimension $1$ in the projectivized tangent bundle $\mathbb PT_{\Omega'}$ of $\Omega'$.
We refer the readers to Mok \cite[p.\,293]{Mo02} for the notion of the $l$-th characteristic bundles.
From \cite{Mo02}, we have the Poincar\'e-Lelong equation
\begin{equation}\label{Eq:PL}
{\sqrt{-1}\over 2\pi}\partial\overline\partial\log\lVert s\rVert_o^2 = mc_1(L,\widehat{g_{\Omega'}}) - l c_1(\pi^*E,\pi^*g_o) + [\mathcal S_{k-1}(\Omega')],
\end{equation}
where $s\in \Gamma(\mathbb PT_{\Omega'}, L^{-m} \varotimes \pi^* E^l)$ such that the zero divisor of $s$ is precisely $\mathcal S_{k-1}(\Omega')$, $E=\mathcal O(1)|_{\Omega'}$, $L\to \mathbb PT_{\Omega'}$ is the tautological line bundle and $[\mathcal S_{k-1}(\Omega')]$ denotes the current of integration over $\mathcal S_{k-1}(\Omega')$.
Here, $\mathcal O(1)$ denotes the positive generator of the Picard group of the compact dual Hermitian symmetric space $X_c'$ of $\Omega'$. (We may also write $\mathcal O_{X_c'}(1)$ in place of $\mathcal O(1)$ in order to avoid ambiguity, and such a notational convention will be adopted later on in Section \ref{Sec:4.2.3} where we deal with reducible bounded symmetric domains.)
Actually, we also have $m=k$ and $l=2$ by \cite[Proposition 3]{Mo02}.
Denote by $\omega$ the K\"ahler form of $(\Omega',{g_{\Omega'}})$.
Since $\widetilde \mu:(\Delta,m_0g_\Delta)\to (\Omega,{g_{\Omega}})$ is a holomorphic isometry and $Z=\widetilde\mu(\Delta)\subset \Omega'$, we may regard $\widetilde \mu:(\Delta,m_0g_\Delta)\to (\Omega',{g_{\Omega'}})$ as a holomorphic isometry.
Let
\[ \hat{Z} = \{ [\alpha] \in \mathbb P(T_x(\Omega')): x\in Z,\; T_x(Z) = \mathbb C \alpha \} \]
be the tautological lifting of $Z$ to $\mathbb P T_{\Omega'}$.
Then, we have $\hat Z\cap \mathcal S_{k-1}(\Omega')=\varnothing$.
Since the normal form of ${\widetilde \mu'(w)\over \lVert \widetilde \mu'(w)\rVert_{g_\Omega}}$ is constant from the construction, $\lVert s\rVert_o>0$ is constant on $\hat Z$ and thus $\sqrt{-1} \partial\overline\partial \log \lVert s\rVert_o \equiv 0$ on $\hat Z$.
Since $ \hat{Z}$ is disjoint from $\mathcal S_{k-1}(\Omega')$, restricting Eq. (\ref{Eq:PL}) to $\hat Z$ we have an identity of smooth (1,1)-forms on $\hat Z$
\begin{equation}\label{Eq:11form1}
kc_1(L,\widehat{g_{\Omega'}})|_{\hat Z} - 2 c_1(\pi^*E,\pi^*g_o)|_{\hat Z} \equiv 0
\end{equation}
as a consequence of Eq. (\ref{Eq:PL}) and the fact that $\sqrt{-1} \partial\overline\partial \log \lVert s\rVert_o \equiv 0$ on $\hat Z$.
In particular, we have
\begin{equation}\label{Eq:11form2}
 kc_1(T_{Z},{g_{\Omega'}}|_{Z}) - 2 c_1(E,g_o)|_{Z} \equiv 0
\end{equation}
by Eq. (\ref{Eq:11form1}).
Note that $c_1(T_{Z},{g_{\Omega'}}|_{Z}) = {1\over 2\pi}\kappa_{Z} \omega|_{Z}$ by the formula for the Gaussian curvature $\kappa_{Z}$ of $(Z,{g_{\Omega'}}|_Z)$ and \cite[p.\,36]{Mo89}.
In addition, we have $c_1(E,g_o)|_{Z}=-{c\over 4\pi} \omega|_Z$ for some $c>0$.
Thus, we obtain
${k\over 2\pi}\kappa_{Z} \omega|_{Z} + {c\over 2\pi}\omega|_Z\equiv 0$ by Eq. (\ref{Eq:11form2}), i.e., $\kappa_{Z}\equiv -{c\over k}$.
Denote by $\Delta_k$ the holomorphic disk of maximal Gaussian curvature $-{2\over k}$, i.e., of diagonal type in the maximal polydisk $\Delta^k\cong \Pi_k\subset \Omega'$.
Then, we have $-k\kappa_{\Delta_k} \equiv c$ and $\kappa_{\Delta_k}\equiv -{2\over k}$ so that $c=2$ (cf.\,\cite[p.\,297]{Mo02}).
Therefore, we have $\kappa_{Z} \equiv -{2\over k}$.
By the Gauss equation we have
$\lVert \sigma'(\widetilde\mu(w))\rVert^2 \le -{2\over k}-\kappa_Z= -{2\over k} + {2\over k} = 0$
so that $\lVert \sigma'(\widetilde\mu(w))\rVert^2 \equiv 0$ on $\Delta$, i.e., $(Z,{g_{\Omega}}|_Z)\subset (\Omega',{g_{\Omega}}|_{\Omega'})$ is totally geodesic.
But then $(\Omega',{g_{\Omega}}|_{\Omega'})\subseteq (\Omega,{g_{\Omega}})$ is totally geodesic so that $(Z,{g_{\Omega}}|_Z)\subset (\Omega,{g_{\Omega}})$ is totally geodesic and thus
$\lVert \widetilde \sigma(\widetilde \mu(w)) \rVert^2\equiv 0$ on $\Delta$.
In particular, we have $\lVert \sigma(\mu(b))\rVert^2 = \lVert \widetilde \sigma(\widetilde \mu(w)) \rVert^2 = 0$.
Since $b\in U\cap\partial\Delta$ is an arbitrary general point, we have $\lVert \sigma(\mu(w))\rVert^2 \to 0$ as $w\to b'$ for a general point $b'\in U\cap\partial\Delta$.
\end{proof}

\subsection{Complete proof of Theorem \ref{MainThm}}
In Section \ref{Sec:CHIE}, we constructed a holomorphic isometry $\widetilde \mu:(\Delta,m_0g_\Delta)\to (\Omega,g_\Omega)$ into an irreducible bounded symmetric domain with certain properties.
The following shows that our study on such a holomorphic isometry may be reduced to the case where $\Omega$ is of tube type.
\begin{pro}\label{Pro_Reduction_Tube_Domain}
Let $\Omega\Subset \mathbb C^N$ be an irreducible bounded symmetric domain of rank $r\ge 2$ and let $\widetilde\mu:(\Delta,m_0 g_\Delta) \to (\Omega,g_\Omega)$ be the constructed holomorphic isometry so that the holomorphic tangent spaces $T_x(Z)=\mathbb C\eta_x$ of $Z:=\widetilde\mu(\Delta)$ are $\Aut(\Omega)$-equivalent and $\eta_y\in T_y(\Omega)$ is a generic vector for any $y\in Z$.
Then, there exists an invariantly geodesic submanifold $\Omega'\subset \Omega$ containing $Z$ such that $\Omega'$ is an irreducible bounded symmetric domain of rank $r$ and of tube type.
In particular, $(Z,g_\Omega|_Z)\subset (\Omega,g_\Omega)$ is totally geodesic.
\end{pro}
\begin{proof}
If $\Omega$ is of tube type, then the result follows from the proof of Theorem \ref{ThmTubeDomain1}.
From now on we consider the case where $\Omega$ is of non-tube type.
From the classification of irreducible bounded symmetric domains, $\Omega$ is biholomorphic to either
$D^{\mathrm{I}}_{p,q}$ ($p<q$), $D^{\mathrm{II}}_{2n+1}$ ($n\ge 2$) or $D^{\mathrm{V}}$.
Define $P:T_\Omega\varotimes T_\Omega\to T_\Omega\varotimes T_\Omega$ by $g(P(\alpha\varotimes \beta),\overline\gamma\varotimes \overline\delta) = R_{\alpha\overline\gamma\beta\overline\delta}(\Omega,g_\Omega)$.
Here $g_x(\cdot,\cdot)$ is a natural Hermitian pairing of the basis for $S^2 T_x(\Omega)$, i.e.,
$g_x(e_i\cdot e_j, \overline{e_s}\cdot \overline{e_l})=1$ (resp.\,$0$) if $\{i,j\}=\{s,l\}$ (resp.\,$\{i,j\}\neq \{s,l\}$).
Then, $P$ is parallel because $\nabla R\equiv 0$.
We define $\rho:(T_\Omega\varotimes T_\Omega)\varotimes T_\Omega^*\to T_\Omega$ so that for each $x\in \Omega$,
$\rho_x:(T_x(\Omega)\varotimes T_x(\Omega))\varotimes T_x^*(\Omega)\to T_x(\Omega)$
is a multilinear map given by $\rho_x(\mu\varotimes \nu)(\omega^*) = \omega^*(\nu)\mu$ for decomposable elements
$(\mu\varotimes \nu)\varotimes \omega^*\in (T_x(\Omega)\varotimes T_x(\Omega))\varotimes T_x^*(\Omega)$.
We have $P(\alpha\varotimes \alpha) = \sum_{\varphi,\varphi'\in \Delta_M^+} R_{\alpha\overline{e_\varphi}\alpha\overline{e_{\varphi'}}}(\Omega,g_\Omega) e_\varphi \varotimes e_\varphi'$ and
$\rho(P(\alpha\varotimes \alpha)\varotimes e_{\mu}^*)
= \sum_{\varphi\in \Delta_M^+} R_{\alpha\overline{e_\varphi}\alpha\overline{e_{\mu}}}(\Omega,g_\Omega)
e_\varphi$.
Define the vector subbundle $V:=\rho(P(\eta\varotimes \eta)\varotimes T_\Omega^*)\subset T_\Omega|_Z$, where $\eta$ is a non-zero holomorphic vector field on $Z=\widetilde\mu(\Delta)\subset \Omega$.

By using the normal form $\eta(w)\in T_{\bf 0}(\Omega)$ of ${\widetilde\mu'(w)\over \lVert \widetilde\mu'(w) \rVert_{g_\Omega}}$, if $\Omega$ is of the classical type, then it follows from direct computation of the Riemannian curvature of $(\Omega,g_\Omega)$ that the normal form of $V_x$ ($x\in Z$) as a complex vector subspace of $T_{\bf 0}(\Omega)$ is exactly $M(p,p;\mathbb C)=T_{\bf 0}(D^{\mathrm{I}}_{p,p})$ (resp.\,$M_a(2n;\mathbb C)=T_{\bf 0}(D^{\mathrm{II}}_{2n})$) if $\Omega\cong D^{\mathrm{I}}_{p,q}$ ($p<q$) (resp.\,$D^{\mathrm{II}}_{2n+1}$ ($n\ge 2$)).
When $\Omega\cong D^{\mathrm{V}}$, it follows from the computation of Tsai \cite[pp.\,149--151]{Ts93} and $R(v,\overline w)v'=-[[v,\overline w],v']$ that the normal form of $V_x$ ($x\in Z$) as a complex vector subspace of $T_{\bf 0}(\Omega)$ is exactly $T_{\bf 0}(\Omega')$ for some invariantly geodesic submanifold $\Omega'\subset \Omega$ satisfying $\Omega'\cong D^{\mathrm{IV}}_8$.
Actually, we write the normal form $\eta(w)=\eta_1(w)e_{x_1-x_2}+\eta_2(w)e_{x_1+x_2+x_3}$ and we compute
$R(\eta(w),\overline{e_{\varphi}})\eta(w)=[[e_{-\varphi},\eta(w)],\eta(w)]$ for each noncompact positive root $\varphi$.
It follows from Tsai \cite[pp.\,149--151]{Ts93} that the normal form of $V_x$ is $\rho(P(\eta(w)\varotimes\eta(w))\varotimes T^*_{\bf 0}(\Omega))$, which is spanned by $e_{x_1-x_i}$, $4\le i\le 6$; $e_{x_1+x_3+x_i}$, $4\le i\le 6$; $e_{x_1-x_2}$ and $e_{x_1+x_2+x_3}$.
Here $\eta(w)=\eta_{\widetilde\mu(w)}$ for $w\in \Delta$.
In particular, the normal form of $V_x$ is exactly $T_{\bf 0}(Q^8)=T_{\bf 0}(D^{\mathrm{IV}}_8)$, where ${\bf 0}$ is identified with the base point $o\in Q^8$.
It is then obvious that $\mathrm{Span}_{\mathbb C}\{e_{\psi_j}(x):j=1,\ldots,k\}\subset V_x$ and $\eta_x\in V_x$ for each $x\in Z$ for each $x\in Z$.
By similar arguments as in the proof of Lemma \ref{LemHVSB_W1}, $V\subset T_\Omega|_Z$ is a holomorphic vector subbundle with $T_Z\subset V$.

Define the second fundamental form $\tau:T_Z\varotimes V\to T_\Omega|_Z/V$ by $\tau(\eta\varotimes \gamma) = \nabla_\eta \gamma\mod V$.
Then, it follows from the arguments in the proof of Lemma \ref{TauTensorHolo1} that $\tau$ is holomorphic since $V_x=T_x (\Omega'_x)$ for some invariantly geodesic submanifold $\Omega'_x\subset \Omega$.
Note that the vector bundle $V$ here is actually the same as the vector bundle $V$ in Lemma \ref{lem:NewBundleV} and Lemma \ref{ExtW}.
Representing $\tau|_{T_Z\varotimes T_Z}$ as a holomorphic section $\hat\tau\in \Gamma(Z,S^2T^*_Z\varotimes (T_\Omega|_Z/V))$, we can extend the definition of $\hat\tau(\zeta)$ to an open neighborhood of a general point on $\partial\Delta$ by Lemma \ref{lem:NewBundleV} and Lemma \ref{ExtW}.
Then, by the arguments in the proof of Lemma \ref{LemTensorV} we have $\tau|_{T_Z\varotimes T_Z}\equiv 0$ after applying the rescaling argument to a local holomorphic extension of $\widetilde\mu$ around a general point $b'\in \partial\Delta$ if necessary.
From the definition of $V\subset T_\Omega|_Z$ and the fact that $(\nabla_{\eta}\hat\eta)(x)\in V_x$ for any $x\in Z$, $\eta\in T_x(Z)$ and $\hat\eta\in \Gamma_{\mathrm{loc},x}(Z,T_Z)$, we have $\tau\equiv 0$.
Actually, $\rho$ is a contraction and thus for $\hat\eta\in T_x(Z)$ and $\eta\in \Gamma_{\mathrm{loc},x}(Z,T_Z)$, we have
\[ \begin{split}
&\nabla_{\hat\eta}(\rho(P(\eta\varotimes\eta)\varotimes \omega^*))(x)\\
=& \rho(\nabla_{\hat\eta}(P(\eta\varotimes\eta))\varotimes \omega^*))(x) + \rho(P(\eta\varotimes\eta)\varotimes (\nabla_{\hat\eta}\omega^*))(x)\\
=&  \rho(
P((\nabla_{\hat\eta}\eta)(x)\varotimes\eta(x))\varotimes \omega^*(x)
) 
+\rho(P(\eta(x)\varotimes(\nabla_{\hat\eta}\eta)(x))
\varotimes \omega^*(x)) \\
&+ \rho(P(\eta(x)\varotimes\eta(x))\varotimes (\nabla_{\hat\eta}\omega^*)(x)),
\end{split}\]
which lies in $V_x$ because $(\nabla_{\hat\eta}\eta)(x)\in V_x$ and $[[\mathfrak m^-,V_x],V_x]\subset V_x$ (cf.\,Tsai \cite[Lemma 4.3]{Ts93}).
In other words, $V$ is parallel on $Z$.
By applying the foliation technique as in the proof of Lemma \ref{LemInCSD1}, there is an invariantly geodesic submanifold $\Omega'\subset \Omega$ such that $Z\subset \Omega'$ and $T_x(\Omega')=V_x$ for any $x\in Z$.
In addition, such a submanifold $\Omega'$ is irreducible and of tube type as a Hermitian symmetric space of the noncompact type. More precisely, we have
\begin{enumerate}
\item[(i)]
If $\Omega\cong D^{\mathrm{I}}_{p,q}$ ($p<q$) (resp.\,$\Omega\cong D^{\mathrm{II}}_{2n+1}$ ($n\ge 2$)), then $\Omega'\cong D^{\mathrm{I}}_{p,p}$ (resp.\,$\Omega'\cong D^{\mathrm{II}}_{2n}$).
\item[(ii)]
If $\Omega\cong D^{\mathrm{V}}$, then $\Omega'\cong D^{\mathrm{IV}}_8$.
\end{enumerate}
From the arguments in the proof of Theorem \ref{ThmTubeDomain1}, $(Z,g_\Omega|_Z)\subset (\Omega',g_\Omega|_{\Omega'})$ is totally geodesic and thus $(Z,g_\Omega|_Z)\subset (\Omega, g_\Omega)$ is totally geodesic.
\end{proof}

By Proposition \ref{pro:redIGM}, Proposition \ref{Pro_Reduction_Tube_Domain} and the proof of Theorem \ref{ThmTubeDomain1}, we have actually proven Theorem \ref{MainThm} under the assumption that the bounded symmetric domain $\Omega$ is irreducible.

Now, it remains to consider the case where the bounded symmetric domain $\Omega$ is reducible.
The idea is to generalize the methods to the case where $\Omega$ is reducible throughout Section \ref{Sec:CHIE}, Section \ref{Proof:Sec1}, and that in Proposition \ref{Pro_Reduction_Tube_Domain}. Then, this will complete the proof of Theorem \ref{MainThm}.

We write $\Omega=\Omega_1\times \cdots \times \Omega_m \Subset \mathbb C^{N_1}\times \cdots \times \mathbb C^{N_m} = \mathbb C^N$ for some integer $m\ge 1$, where $\Omega_j\Subset \mathbb C^{N_j}$ is an irreducible bounded symmetric domain in its Harish-Chandra realization for $j=1,\ldots,m$.
Equipping $\Omega$ (resp.\,$\Delta$) with the Bergman metric $ds_\Omega^2$ (resp.\,$ds_\Delta^2$), by slight modifications we obtain analogues of Lemma \ref{LemAsGC1}, Lemma \ref{LemSeqHE1}, Lemma \ref{LemEigen1}, Proposition \ref{ProConstHI1} and the results in Section \ref{Proof:Sec1} when $\Omega$ is reducible.
Recall that $\mu:U=\mathbb B^1(b_0,\epsilon)\to \mathbb C^{N_1}\times\cdots \times \mathbb C^{N_m}=\mathbb C^N$, $\epsilon>0$, is a holomorphic embedding such that $\mu(U\cap\Delta)\subset \Omega$ and $\mu(U\cap\partial\Delta)\subset \partial\Omega$, where $b_0\in \partial\Delta$.
We also write $\mu=(\mu_1,\ldots,\mu_m)$ with $\mu_j:U\to \mathbb C^{N_j}$ being a holomorphic map for $j=1,\ldots,m$.
\subsubsection{Basic settings}
We write the Bergman kernel
$K_\Omega(z,\xi)={1\over Q_\Omega(z,\xi)}$ for some polynomial $Q_\Omega(z,\xi)$ in $(z,\overline\xi)$.
Then, we have the K\"ahler form
$\omega_{ds_\Omega^2}=-\sqrt{-1}\partial\overline\partial\log Q_\Omega(z,z)$ of $(\Omega,ds_\Omega^2)$.
When $\Omega=\Delta$, we have
$Q_\Delta(z,\xi) = \pi\cdot (1-z\overline{\xi})^2$ for $z,\xi\in \mathbb C$.
We can construct a germ of holomorphic isometry $\widetilde\mu$ as in Lemma \ref{LemSeqHE1} and Proposition \ref{ProConstHI1}.
Actually, for a general point $b\in U\cap\partial\Delta$ there is an open neighborhood $U_b$ of $b$ in $U\subset\mathbb C$ such that
\[ Q_\Omega(\mu(w),\mu(w)) = \chi(w) (1-|w|^2)^{\lambda'}
= {\chi(w)\over \pi^{\lambda'\over 2}} Q_\Delta(w,w)^{\lambda'\over 2}
 \]
on $U_b$ for some positive smooth function $\chi$ on a neighborhood of $\overline{U_b}$ and some positive integer $\lambda'$.
We may construct the sequence $\{\widehat\mu_j=\Phi_j\circ \mu\circ \varphi_j\}_{j=1}^{+\infty}$ as in Section \ref{Sec:CHIE} such that
\[ \widehat\mu_j^* \omega_{ds_\Omega^2}
={\lambda'\over 2} \omega_{ds_\Delta^2}
-\sqrt{-1}\partial\overline\partial\log \chi(\varphi_j(\zeta)). \]
Then, we obtain a germ of holomorphic isometry
$\widetilde\mu: \left(\Delta, {\lambda'\over 2} ds_\Delta^2;0\right)
\to \left(\Omega, ds_\Omega^2;{\bf 0}\right)$
by taking the limit of some subsequence of $\{\widehat\mu_j\}_{j=1}^{+\infty}$.
Note that such a germ $\widetilde\mu$ could be extended to a holomorphic isometry from $\left(\Delta, {\lambda'\over 2} ds_\Delta^2\right)$ to $\left(\Omega, ds_\Omega^2\right)$ by the extension theorem of Mok \cite{Mo12}.
We also denote the extension of $\widetilde\mu$ by $\widetilde\mu$ and write $Z=\widetilde\mu(\Delta)$.
By decomposing $T_x(\Omega)=T_{x_1}(\Omega_1)\varoplus \cdots\varoplus T_{x_m}(\Omega_m)$ for $x=(x_1,\ldots,x_m)\in\Omega_1\times \cdots \times \Omega_m$, we may decompose the normal form $\eta(w)=\eta_1(w)+\ldots+\eta_m(w)\in T_{{\bf 0}}(\Omega_1)\varoplus \cdots\varoplus T_{{\bf 0}}(\Omega_m)$ of ${\widetilde\mu'(w)\over \lVert \widetilde\mu'(w) \rVert_{ds_\Omega^2}}$.
Then, we have analogous results as in Proposition \ref{ProConstHI1} for the case where $\Omega$ is reducible.
More precisely, the normal form $\eta(w)$ is independent of $w\in \Delta$ and $\lVert \widetilde\sigma(\widetilde\mu(w))\rVert^2 \equiv \lVert \sigma(\mu(b))\rVert^2$ on $\Delta$, where $\widetilde\sigma(x)$ denotes the second fundamental form of $(Z,ds_\Omega^2|_Z)$ in $(\Omega,ds_\Omega^2)$ at $x\in Z$.

From now on $Z=\widetilde\mu(\Delta)$ has $\mathrm{Aut}(\Omega)$-equivalent holomorphic tangent spaces $T_x(Z)=\mathbb C\eta_x$ and $\eta_y\in T_y(\Omega)$ is of rank $k$ for any $y\in Z$.

\subsubsection{Insertion of a tube domain containing the embedding Poincar\'e disk}\label{Sec:ZinIGM}
The first step is to show that since the holomorphic tangent spaces of $Z:=\widetilde\mu(\Delta)$ are $\mathrm{Aut}(\Omega)$-equivalent and of rank $k$, $Z$ lies inside an invariantly geodesic submanifold $\Omega'\subset\Omega$ of rank $k$ and of tube type as a bounded symmetric domain.
Write $\widetilde\mu=(\widetilde\mu_1,\ldots,\widetilde\mu_m)$, where $\widetilde\mu_j:\Delta\to \Omega_j\Subset \mathbb C^{N_j}$ is a holomorphic map for $j=1,\ldots,m$.

By permuting the irreducible factors $\Omega_j$'s of $\Omega$, we may assume that
$\eta(w)=\eta_1(w)+\ldots+\eta_m(w)\in T_{\bf 0}(\Omega)=T_{\bf 0}(\Omega_1)\varoplus \cdots \varoplus T_{\bf 0}(\Omega_m)$ is of rank $k=\sum_{j=1}^{m} k_j$ and each 
$\eta_i(w)\in T_{\bf 0}(\Omega_i)$ is of rank $k_i$ such that $k_l>0$ for $l=1,\ldots,m'$, $k_j=0$, $\eta_j(w)=0$ and $\widetilde\mu_j(w)\equiv x'_j$ is a constant map for $m'+1\le j\le m$ provided that $m'<m$.

\vskip 0.25cm
\noindent\textbf{Tube type:}
We first consider the case where $\Omega$ is of tube type, equivalently all $\Omega_j$'s are of tube type.
For $x\in \Omega$, let $Q_x$ be a Hermitian bilinear form on $T_x(\Omega)\varotimes\overline{T_x(\Omega)}$ given by
$Q_x(\alpha\varotimes\overline\beta,\alpha'\varotimes\overline{\beta'}) = R_{\alpha\overline{\alpha'}\beta'\overline\beta}(\Omega,ds_\Omega^2)$.
For $x_j\in \Omega_j$, we also let $Q^{(j)}_{x_j}$ be a Hermitian bilinear form on $T_{x_j}(\Omega_j)\varotimes\overline{T_{x_j}(\Omega_j)}$ defined by
$Q^{(j)}_{x_j}(\alpha\varotimes\overline\beta,\alpha'\varotimes\overline{\beta'}) := R_{\alpha\overline{\alpha'}\beta'\overline\beta}(\Omega_j,ds_{\Omega_j}^2)$ and let $\mathcal N^{(j)}_{\alpha_j}$ be the null space of the Hermitian bilinear form $H^{(j)}_{\alpha_j}(v,v'):= R_{\alpha_j\overline{\alpha_j}v\overline{v'}}(\Omega_j,ds_{\Omega_j}^2)$ for $\alpha_j\in T_{x_j}(\Omega_j)$.
For $w\in \Delta$, we define $W_{\widetilde\mu(w)}
:= \left\{ v\in T_{\widetilde\mu(w)}(\Omega) : Q_{\widetilde\mu(w)}(v\varotimes \overline \zeta,\cdot )\equiv 0\;\;\forall\;\zeta\in \mathcal N_{\widetilde\mu'(w)}\right\}$.
Then, we have
$W_{\widetilde\mu(w)}= \bigoplus_{j=1}^m W^{(j)}_{\widetilde\mu_j(w)}$,
where
\[ W^{(j)}_{\widetilde\mu_j(w)}
:=\left\{v_j \in T_{\widetilde\mu_j(w)}(\Omega_j): Q^{(j)}_{\widetilde\mu_j(w)}(v_j\varotimes\overline\zeta,\cdot)\equiv 0\; \;\forall\;\zeta\in \mathcal N^{(j)}_{\widetilde\mu_j'(w)}\right\},\]
$j=1,\ldots,m$.
For $x=(x_1,\ldots,x_m)\in Z\subset \Omega=\Omega_1\times\cdots \times \Omega_m$, we have
\[ \begin{split}
W_x&= \bigoplus_{j=1}^m W^{(j)}_{x_j}\\
&= \begin{cases}
T_{x_1}(\Omega'_{1,x_1})\varoplus\cdots \varoplus
T_{x_{m'}}(\Omega'_{m',x_{m'}})
\varoplus \{{\bf 0}\}\varoplus \cdots \varoplus \{{\bf 0}\}
& \text{if } m'<m\\
T_{x_1}(\Omega'_{1,x_1})\varoplus\cdots \varoplus
T_{x_{m}}(\Omega'_{m,x_{m}}) & \text{if } m'=m
\end{cases}
\end{split}\]
for some characteristic subdomain $\Omega'_{j,x_j}\subseteq \Omega_j$ of rank $k_j$, $j=1,\ldots,m'$.
Note that it is possible that $\Omega'_{i,x_i}=\Omega_i$ for some $i$.
Similarly, we may define the holomorphic vector bundle $V$ (resp.\,$V'$) as in Lemma \ref{lem:NewBundleV} and Lemma \ref{ExtW}. Then, by the arguments in the proofs of Lemma \ref{lem:NewBundleV} and Lemma \ref{ExtW} we have $V_x=W_x$ for any $x\in Z$.
Thus, our results in Section \ref{Proof:Sec1} can be generalized to the case where $\Omega$ (resp.\,$\Omega'$) is reducible.
It follows from the arguments in Section \ref{Proof:Sec1} that there is a characteristic subdomain $\Omega'$ of $\Omega$ containing $Z=\widetilde\mu(\Delta)$ such that $\Omega'=\Omega'_1\times\cdots \times \Omega'_{m'}\times \{x_{m'+1}\}\times\cdots \times \{x_m\}$
(resp.\,$\Omega':=\Omega'_1\times\cdots \times \Omega'_{m}$)
if $m'< m$ (resp.\,$m'=m$),
where $\Omega'_j\subset \Omega_j$ is a characteristic subdomain of rank $k_j$, $1\le j\le m'$.
Note that each $\Omega'_j$ is of tube type and each $\eta_j(w)\in T_{\bf 0}(\Omega_j')$ is of rank $k_j=\mathrm{rank}(\Omega_j')$ for $j=1,\ldots,m'$.

\vskip 0.25cm
\noindent\textbf{Non-tube type:}
Suppose $\Omega=\Omega_1\times\cdots \times \Omega_m$ is of non-tube type.
If $k_l<\mathrm{rank}(\Omega_l)$ for some $l$, $1\le l\le m'$, then we have $Z\subset \Omega_1\times\cdots \times \Omega_{l-1}\times \Omega'_{l}\times \Omega_{l+1}\times\cdots \times \Omega_{m}$ for some invariantly geodesic submanifold $\Omega'_{l}$ of $\Omega_l$ such that $\Omega'_l$ is an irreducible bounded symmetric domain of tube type and of rank $k_l$ by making use of Proposition \ref{pro:redIGM}. Inductively, there is an invariantly geodesic submanifold $\Omega'$ of $\Omega$ such that $\Omega'$ is a bounded symmetric domain of rank $k$ and $Z\subseteq \Omega'$.
In this case, $T_x(Z)$ is spanned by a generic vector in $T_x(\Omega')$ for any $x\in Z$.
From now on we may suppose that $T_x(Z)$ is spanned by a generic vector in $T_x(\Omega)$ for any $x\in Z$ and $m'=m$ without loss of generality.

In analogy to the case in which we consider the holomorphic vector subbundle $W\subset T_\Omega|_Z$, we generalize the method in the proof of Proposition \ref{Pro_Reduction_Tube_Domain} to the case where $\Omega$ is reducible and equipped with the Bergman metric $ds_\Omega^2$.
The key point is that our construction of the holomorphic vector subbundle $V\subset T_\Omega|_Z$ comes from the Riemannian curvature tensor of $(\Omega,ds_\Omega^2)$, which is decomposed into the sum of Riemannian curvature tensors of $(\Omega_j,ds_{\Omega_j}^2)$ for $j=1,\ldots,m$.
Note that we may also define $V$ as in Lemma \ref{lem:NewBundleV} and we also have the vector bundle $V'$ extending $V$ locally in the sense of Lemma \ref{lem:NewBundleV}.
Then, it follows that there is an invariantly geodesic submanifold $\Omega'_j\subseteq \Omega_j$ of rank equal to that of $\Omega_j$ and of tube type for $j=1,\ldots,m$ such that $Z\subset \Omega':=\Omega'_1\times\cdots \times \Omega'_{m}$.
In particular, $\Omega'\subset \Omega$ is an invariantly geodesic submanifold which is of tube type and $\mathrm{rank}(\Omega')= \mathrm{rank}(\Omega)$.

\medskip
In any case, given a bounded symmetric domain $\Omega$ of rank $r$, the Poincar\'e disk $Z$ lies inside an invariantly geodesic submanifold $\Omega'\subset \Omega$ of rank $k$ and of tube type such that the holomorphic tangent spaces $T_x(Z)$ are $\Aut(\Omega')$-equivalent and $T_y(Z)$ is spanned by a generic vector in $T_y(\Omega')$ for any $y\in Z$.
This completes the first step of the proof of Theorem \ref{MainThm}.

\subsubsection{Application of the Poincar\'e-Lelong equation in the reducible case}\label{Sec:4.2.3}
We note that the method of using the Poincar\'e-Lelong equation as in the proof of Theorem \ref{ThmTubeDomain1} may be extended to the case where the bounded symmetric domain $\Omega'$ is reducible.
\begin{pro}\label{Pro_General_Poincare_Lelong}
Let $\Omega'=\Omega'_1\times\cdots \times \Omega'_{{m'}}$ be a bounded symmetric domain of tube type and of rank $k$, where $\Omega'_j$, $1\le j\le m'$, are the irreducible factors of $\Omega'$ and $m'$ is a positive integer.
Equip $\Omega'$ with a K\"ahler metric $g'_{\Omega'}:=\sum_{j=1}^{m'} \mathrm{Pr}_j^* g'_{\Omega'_j}$ on $\Omega'$, where $g'_{\Omega'_j}:=\lambda_j g_{\Omega'_j}$ for some positive integer $\lambda_j$ and $\mathrm{Pr}_j:\Omega' \to \Omega'_j$ is the canonical projection onto the $j$-th irreducible factor $\Omega'_j$ of $\Omega'$, $j=1,\ldots,m'$.
We also let $Z\subset \Omega'$ be the holomorphic curve, i.e., $Z$ is the image of the holomorphic isometry $\widetilde\mu: (\Delta,\lambda ds_\Delta^2) \to (\Omega',g'_{\Omega'})$ for some positive real constant $\lambda>0$, such that $T_x(Z)$ is spanned by a rank-$k$ unit vector $\eta_x\in T_x(\Omega')$ for any $x\in Z$ and the normal form of $\eta_y$ is independent of $y\in Z$.
Then, $(Z,g'_{\Omega'}|_Z)\subset (\Omega',g'_{\Omega'})$ is totally geodesic.
\end{pro}
\begin{proof}
If $\Omega'$ is irreducible, then we are done by the proof of Theorem \ref{ThmTubeDomain1}.
Consider the case where $\Omega'_1\times\cdots \times \Omega'_{m'}$ is reducible and of tube type with irreducible factors $\Omega'_j$, $1\le j\le m'$, and $m'\ge 2$ is an integer.
Under the assumptions, each $\Omega'_j$ is an irreducible bounded symmetric domain of rank $k_j\ge 1$ and of tube type, $1\le j \le m'$, so that $k=\sum_{j=1}^{m'}k_j$.
We only need to apply the method in the proof of Theorem \ref{ThmTubeDomain1} and that in \cite{Mo02}, and we generalize the settings to the case where $\Omega'$ is reducible.
Denote by $\mathcal S^{(j)}_{l,x_j}(\Omega'_j)$ the $l$-th characteristic variety for $\Omega'_j$ at $x_j\in \Omega'_j$, $j=1,\ldots,m'$.
For $x=(x_1,\ldots,x_{m'})\in \Omega'$, we denote by
$\mathcal S_{k-1,x}^j(\Omega')=
\left\{ [v_1\varoplus \cdots \varoplus v_{m'}]\in \mathbb P\left(T_{x_1}(\Omega'_1)\varoplus \cdots \varoplus T_{x_{m'}}(\Omega'_{m'})\right): v_j \in \widehat{\mathcal S}^{(j)}_{k_j-1,x_j}(\Omega'_j)\right\}$, where $\widehat{\mathcal S}^{(j)}_{k_j-1,x_j}(\Omega'_j)$ is the affine cone over $\mathcal S^{(j)}_{k_j-1,x_j}(\Omega'_j)$ in $T_{x_j}(\Omega'_j)$, $1\le j\le m'$.
Then, $\mathcal S_{k-1,x}(\Omega'):= \bigcup_{j-1}^{m'} \mathcal S_{k-1,x}^j(\Omega')$ is a union of $m'$ hypersurfaces of $\mathbb P(T_x(\Omega'))$.
Moreover, we obtain a divisor $\mathcal S_{k-1}^j(\Omega')=\bigcup_{x\in \Omega'}  \mathcal S_{k-1,x}^j(\Omega')\subseteq \mathbb PT_{\Omega'}$, which yields a divisor line bundle $[\mathcal S_{k-1}^j(\Omega')]$ over $\mathbb PT_{\Omega'}$ for $1\le j\le m'$.

Let $L\to \mathbb PT_{X_c'}$ be the tautological line bundle and $\pi: \mathbb PT_{X_c'}\to X_c'$ be the projectivized tangent bundle over the compact dual Hermitian symmetric space $X_c'$ of $\Omega'$.
Writing $X_c'=X_{c,1}'\times\cdots \times X_{c,m'}'$ so that each $X_{c,j}'$ is the compact dual Hermitian symmetric space of $\Omega_j'$, we have $\mathrm{Pic}(X_c')\cong \mathrm{Pic}(X_{c,1}')\times\cdots \times \mathrm{Pic}(X_{c,m'}')$.
In analogy to the case of $\Omega'$, we define the divisor $\mathcal S_{k-1}^j(X_c')\subset \mathbb PT_{X_c'}$ and a divisor line bundle $[\mathcal S_{k-1}^j(X_c')]$ over $\mathbb PT_{X_c'}$ for $1\le j\le m'$.
Denote by $\mathrm{Pr}_{c,j}:X_c' \to X_{c,j}'$ the canonical projection onto the $j$-th irreducible factor $X_{c,j}'$ of $X_c'$ and $\pi_{c,j}:=\mathrm{Pr}_{c,j}\circ \pi$, $j=1,\ldots,m'$.
Therefore, $\mathrm{Pic}(\mathbb P T_{X_c'})$ is generated by $\pi_{c,j}^*\mathcal O_{X_{c,j}'}(1)$, $j=1,\ldots,m'$, and $L$.
Since $\mathcal S^j_{k-1,x}(X_c')$ is of degree $k_j$ as a subvariety of $\mathbb P(T_x(X_c'))$ for any $x\in X_c'$, $L^{k_j}\varotimes [\mathcal S^j_{k-1}(X_c')]$ is a holomorphic line bundle which is trivial on every fiber of $\pi:\mathbb PT_{X_c'}\to X_c'$ by Mok \cite[p.\,293]{Mo02}.
Then, it follows from the proof of \cite[Proposition 3]{Mo02} that
$[\mathcal S^j_{k-1}(X_c')]\cong L^{-k_j}\varotimes \pi_{c,j}^* \mathcal O_{X_{c,j}'}(2)$ when $\Omega'_j$ is of rank $\ge 2$.
If $\Omega'_j\cong \Delta$ is the unit disk for some $j$, then we also have $[\mathcal S_{k-1}^j(X_c')]\cong L^{-1}\varotimes \pi_{c,j}^* \mathcal O_{X_{c,j}'}(2)$ with $X_{c,j}'\cong \mathbb P^1$.

We also denote by $\pi : \mathbb PT_{\Omega'} \to \Omega'$ the canonical projection for simplicity, and recall that $\mathrm{Pr}_j:\Omega'\to \Omega_j'$ is the canonical projection onto the $j$-th irreducible factor of $\Omega'$.
Write $\pi_j:=\mathrm{Pr}_j\circ \pi$ and let $E_j$ be the restriction of $\mathcal O_{X_{c,j}'}(1)$ to $\Omega_j'$ for $j=1,\ldots,m'$.
We also denote by $L$ the restriction of $L$ to $\Omega'$ and $\widehat{g'_{\Omega'}}$ the canonical Hermitian metric on $L|_{\Omega'}$ induced from the K\"ahler metric $g'_{\Omega'}$ on $\Omega'$.
By duality, we have
$[\mathcal S_{k-1}^j(\Omega')]\cong L^{-k_j}\varotimes\pi_j^*E_j^{2}$ for $j=1,\ldots,m'$.
It follows from \cite{Mo02} that for $j=1,\ldots,m'$ we have the Poincar\'e-Lelong equation
\begin{equation}\label{Eq:RedPL1}
{\sqrt{-1}\over 2\pi}\partial\overline\partial\log\lVert s_j\rVert_o^2
=k_j c_1\big(L,\widehat{g'_{\Omega'}}\big) -
2c_1\left(\pi_j^* E_j,
\pi_j^*h^j_o\right) + [\mathcal S^j_{k-1}(\Omega')],
\end{equation}
where $s_j$ is a non-trivial holomorphic section of $L^{-k_j}\varotimes \pi_j^*E_j^{2}$ whose zero set is precisely $\mathcal S_{k-1}^j(\Omega')$ and $[\mathcal S^j_{k-1}(\Omega')]$ denotes the current of integration over $\mathcal S^j_{k-1}(\Omega')$.
Here the Hermitian metric $h^j_o$ on $E_j=\mathcal O_{X_{c,j}'}(1)|_{\Omega'_j}$ is induced from the K\"ahler metric $g'_{\Omega_j'}$ on $\Omega_j'$.
Let $\hat Z$ be the tautological lifting of $Z$ to $\mathbb PT_{\Omega'}$.
Then, $\hat Z$ is disjoint from $\mathcal S_{k-1}^j(\Omega')$ for any $j$.
Since the normal form of the unit tangent vector $\eta_x$ in $T_x(Z)$ is independent of $x\in Z$, $\lVert s_j\rVert_o>0$ is constant on $\hat Z$ from the construction of $\hat Z$ and thus $\partial\overline\partial\log\lVert s_j\rVert_o^2\equiv 0$ on $\hat Z$.
Therefore, by Eq. (\ref{Eq:RedPL1}) we have
\[ k_j c_1\big(L,\widehat{g'_{\Omega'}}\big)|_{\hat Z} -
 2c_1\left(\pi_j^* E_j,
\pi_j^*h^j_o\right)|_{\hat Z} = 0 \]
and thus
\begin{equation}\label{Eq:redPL2}
-k_jc_1(T_Z,g'_{\Omega'}|_{Z}) +
2c_1\left( \mathrm{Pr}_j^*E_j,
\mathrm{Pr}_j^* h^j_o\right)|_Z = 0.
\end{equation}
It follows from \cite{Mo02} and the proof of Theorem \ref{ThmTubeDomain1} that
$$2 c_1\left( \mathrm{Pr}_j^*E_j,\mathrm{Pr}_j^* h^j_o\right) = -{2\over 2\pi} \mathrm{Pr}_j^*\omega_{g_{\Omega'_j}} = -{1\over \pi \lambda_j}\mathrm{Pr}_j^*\omega_{g'_{\Omega'_j}}$$
for any $j$. Moreover, we have $c_1(T_Z,g'_{\Omega'}|_{Z})={1\over 2\pi} \kappa_Z \omega_{g'_{\Omega'}}|_Z$. Therefore, we have $-\lambda_jk_j{1\over 2\pi} \kappa_Z \omega_{g'_{\Omega'}}|_Z = {1\over \pi}\mathrm{Pr}_j^*\omega_{g'_{\Omega'_j}}|_Z$ for any $j$ by Eq. (\ref{Eq:redPL2}) and thus
\[ -\sum_{j=1}^{m'} \lambda_j k_j\kappa_Z \omega_{g'_{\Omega'}}|_Z
= 2\sum_{j=1}^{m'}\mathrm{Pr}_j^*\omega_{g'_{\Omega'_j}}|_Z
= 2\omega_{g'_{\Omega'}}|_Z. \]
Writing $l_0:=-\sum_{j=1}^{m'} \lambda_j k_j$, the above equality becomes $l_0\kappa_Z \omega_{g'_{\Omega'}}|_Z=2\omega_{g'_{\Omega'}}|_Z$, i.e., $l_0\kappa_Z\equiv 2$.
Denote by $\Delta_k$ a totally geodesic holomorphic disk in $(\Omega',g'_{\Omega'})$ of constant Gaussian curvature $\kappa_{\Delta_k}$ which is equal to the maximal holomorphic sectional curvature of $(\Omega',g'_{\Omega'})$.
Then, we have $\kappa_{\Delta_k}=-{2\over \sum_{j=1}^{m'}\lambda_j k_j}$, where $k_j=\mathrm{rank}(\Omega_j')$, $j=1,\ldots,m'$.
Let $\sigma'(x)$ be the second fundamental form of $(Z,g'_{\Omega'}|_Z)\subset (\Omega',g'_{\Omega'})$ at $x\in Z$ and $\eta_x\in T_x(Z)\subset T_x(\Omega')$ be a unit tangent vector.
Then, we have $\lVert \sigma'(x)\rVert^2 = R_{\eta_x\overline{\eta_x}\eta_x\overline{\eta_x}}(\Omega',g'_{\Omega'})-\kappa_Z\le \kappa_{\Delta_k}-\kappa_Z = {2\over l_0}-{2\over l_0}=0$ for any $x\in Z$ by the Gauss equation, i.e., $\lVert \sigma'\rVert^2\equiv 0$, and thus $(Z,g'_{\Omega'}|_Z)\subset (\Omega',g'_{\Omega'})$ is totally geodesic.
\end{proof}
\subsubsection{Conclusion of the proof}
From our construction and the above two steps, we complete the proof of Theorem \ref{MainThm} as follows.
\begin{proof}[Proof of Theorem \ref{MainThm}]
The case where $\Omega$ is of rank $1$ is obviously true by our construction in Section \ref{Sec:CHIE}, so we assume that $\Omega$ is of rank $\ge 2$.
Following the construction of the holomorphic curve $Z$ throughout Sections \ref{Sec:CHIE} and \ref{Sec:Proof}, we first consider the case where $\Omega$ is of tube type.
Then, we have shown that $Z\subset \Omega'$ for some rank-$k$ characteristic subdomain $\Omega'\subset \Omega$ of tube type such that the holomorphic tangent spaces $T_x(Z)=\mathbb C \eta_x$ are $\mathrm{Aut}(\Omega')$-equivalent and $\eta_y\in T_y(\Omega')$ is a generic vector for any $y\in Z$.
It follows from Proposition \ref{Pro_General_Poincare_Lelong} that $(Z,ds_\Omega^2|_Z)\subset (\Omega',ds_{\Omega}^2|_{\Omega'})$ is totally geodesic.
Then, $(Z,ds_\Omega^2|_Z)\subset (\Omega,ds_\Omega^2)$ is totally geodesic by the total geodesy of $(\Omega',ds_\Omega^2|_{\Omega'})$ in $(\Omega,ds_\Omega^2)$.
From the proof of Theorem \ref{ThmTubeDomain1}, we have $\lVert \sigma(\mu(w))\rVert^2 \to 0$ as $w\to b$ for a general point $b\in U\cap\partial\Delta$.
Hence, the proof is complete under the assumption that $\Omega$ is of tube type.

Now, it remains to consider the case where $\Omega$ is of non-tube type.
In Section \ref{Sec:ZinIGM}, we have shown that $Z\subset \Omega'$ for some invariantly geodesic submanifold $\Omega'\subset \Omega$ such that $\Omega'$ is of tube type, the holomorphic tangent spaces $T_x(Z)=\mathbb C\eta_x$ are $\mathrm{Aut}(\Omega')$-equivalent and $\eta_y\in T_y(\Omega')$ is a generic vector for any $y\in Z$.
Writing $\Omega'=\Omega'_1\times\cdots \times \Omega'_m \subset \Omega = \Omega_1\times \cdots \times \Omega_m$, we have
$ds_{\Omega}^2|_{\Omega'}
= \sum_{j=1}^m (p(\Omega_j)+2) \mathrm{Pr_j}^*g_{\Omega'_j}$, where $\mathrm{Pr}_j:\Omega'\to \Omega'_j$ is the canonical projection onto the $j$-th irreducible factor of $\Omega'$ for $1\le j\le m$.
Here, for $1\le j\le m$, $p(\Omega_j)$ denotes the complex dimension of the VMRTs of the compact dual Hermitian symmetric space of $\Omega_j$ at the base point (see Section \ref{Sec:2.1}).
Then, Proposition \ref{Pro_General_Poincare_Lelong} asserts that $(Z,ds_\Omega^2|_Z)\subset (\Omega',ds_\Omega^2|_{\Omega'})$ is totally geodesic.
This yields the total geodesy of $(Z,ds_\Omega^2|_Z)\subset (\Omega,ds_\Omega^2)$.
In analogy to the case where $\Omega$ is of tube type, from our construction we have $\lVert \sigma(\mu(w))\rVert^2 \to 0$ as $w\to b$ for a general point $b\in U\cap\partial\Delta$.
\end{proof}

\section{Applications}\label{Sec:5}

\subsection{Total geodesy of equivariant holomorphic embeddings}
As a first application of Theorem \ref{ThmHIAT} we have a result on the total geodesy of equivariant holomorphic isometries between bounded symmetric domains, as follows.  

\begin{thm}[Theorem 3.5.2 \cite{Mo11}]\label{ThmEquiv}
Let $D$ and $\Omega$ be bounded symmetric domains, $\Phi:\Aut_0(D)\to \Aut_0(\Omega)$ be a group homomorphism, and $F:D\to \Omega$ be a $\Phi$-equivariant holomorphic map.
Then, $F(D) \subset \Omega$ is a totally geodesic complex submanifold with respect to the Bergman metric $ds_\Omega^2$.
\end{thm}

In Mok \cite[p.\,255]{Mo11} a brief sketch of the deduction of Theorem \ref{ThmEquiv} from Theorem \ref{ThmHIAT} was given. To make the article self-contained we give here the full proof, and in the next application we will make use of Theorem \ref{ThmEquiv} to study the uniformization map $\pi: \Omega \to X_\Gamma := \Omega/\Gamma$ from a bounded symmetric domain $\Omega$ to a not necessarily arithmetic quotient $X_\Gamma := \Omega/\Gamma$ by a torsion-free discrete subgroup $\Gamma \subset {\mbox{Aut}}(\Omega)$.

\begin{proof}[Proof of Theorem \ref{ThmEquiv}]
Let $D = D_1\times\cdots \times D_m$ be the decomposition of $D$ into irreducible factors, where $m\ge 1$. Denote by $\sigma$ the $(1,0)$-part of the second fundamental form of $D$ in $\Omega$. $F^*ds_\Omega^2$ is ${\rm Aut}_0(D)$-equivariant, hence $F^*ds_\Omega^2 = \lambda_1\pi_1^*ds_{D_1}^2 + \cdots + \lambda_m\pi_m^*ds_{D_m}^2$ for some $\lambda_i \ge 0$, $1 \le i \le m$, where $\pi_i: D = D_1\times\cdots \times D_m \to D_i$ denotes the canonical projection onto the $i$-th Cartesian factor. Thus, removing factors $D_i$ for which $\lambda_i = 0$ we may assume without loss of generality that $F$ is a holomorphic immersion. Let now $x_0 \in D$ and $U$ be a sufficiently small open neighborhood of $x_0$ in $D$ such that $F|_U : U \to \Omega$ is a holomorphic embedding.  We identify $U$ with $S := F(U) \subset \Omega$, and denote by $R^\Omega$ resp.\,$R^S$ the curvature tensor of $(\Omega,ds_\Omega^2)$ resp.\,$(S,ds_\Omega^2|_S)$. 

Denote by $\sigma$ the (1,0)-part of the second fundamental form of $(S,ds_\Omega^2|_S)$ $\hookrightarrow$ $(\Omega,ds_\Omega^2)$. For $\alpha, \beta \in T_{x_0}(U) \cong T_{F(x_0)}(S)$, by the Gauss equation we have $R^S_{\alpha\overline{\alpha}\beta\overline{\beta}} = R^\Omega_{\alpha\overline{\alpha}\beta\overline{\beta}} -\|\sigma(\alpha,\beta)\|^2$.  If now we take $i \neq j$, $1 \le i, j \le m$, and $\alpha = \eta_i$ resp.\,$\beta = \eta_j$, where, by an obvious abuse of notation, $\eta_i \in T_{x_0}(D_i)$ resp.\,$\eta_j \in T_{x_0}(D_j)$, then 
$$
0 = R^S_{\eta_i\overline{\eta_i}\eta_j\overline{\eta_j}} = R^\Omega_{\eta_i\overline{\eta_i}\eta_j\overline{\eta_j}} - \|\sigma(\eta_i,\eta_j)\|^2 \le 
-\|\sigma(\eta_i,\eta_j)\|^2,
$$
which implies in particular that $\sigma(\eta_i,\eta_j)$ = 0.  To prove that $\sigma \equiv 0$ it suffices therefore to show that for any $i$, $1 \le i \le m$, we have $\sigma(\eta'_i,\eta''_i) = 0$ whenever $\eta'_i, \eta''_i \in T_{x_0}(D_i)$.  When $D_i$ is of rank $\ge 2$, for $\alpha, \zeta \in T_{x_0}(D_i)$ such that $R^S_{\alpha\overline{\alpha}\zeta\overline{\zeta}} = 0$, by the Gauss equation  
$$
0 = R^S_{\alpha\overline{\alpha}\zeta\overline{\zeta}} = R^\Omega_{\alpha\overline{\alpha}\zeta\overline{\zeta}} - \|\sigma(\alpha,\zeta)\|^2 \le 
-\|\sigma(\alpha,\zeta)\|^2,
$$ 
so that $\sigma(\alpha,\zeta) = 0$.  From the proof of Hermitian metric rigidity (Mok \cite[Proposition 3.4]{Mo87}), by polarization this already implies that $\sigma(\eta'_i,\eta''_i) = 0$ whenever $\eta'_i, \eta''_i \in T_{x_0}(D_i)$. On the other hand, when $D_i$ is of rank 1, for any nonzero vector $\eta_i \in T_{x_0}(D_i)$ there is a (totally geodesic) minimal disk $\Delta_{\eta_i} \subset D_i$ such that $T_{x_0}(\Delta_{\eta_i}) = \mathbb C\eta_i$.  From the $\Phi$-equivariance of $F$ the norm of the second fundamental form $\sigma_1$ of $F(\Delta_{\eta_i})$ in $\Omega$ is a constant. Hence, by Theorem \ref{ThmHIAT} $F|_{\Delta_{\eta_i}}$ is totally geodesic, i.e., $\sigma_1 \equiv 0$. Since by the Gauss equation we have $\|\sigma(\eta_i,\eta_i)\|^2 \le \|\sigma_1(\eta_i,\eta_i)\|^2$ we conclude that $\sigma(\eta_i,\eta_i) = 0$.  As we vary $\eta_i \in T_{x_0}(D_i)$, by polarization we conclude that $\sigma(\eta_i',\eta_i'') = 0$ for any $\eta_i', \eta_i'' \in T_{x_0}(D_i)$.  The proof of Theorem \ref{ThmEquiv} is complete.
\end{proof} 

\subsection{Total geodesy of algebraic subsets admitting compact quotients}
We now apply Theorem \ref{ThmEquiv} to study a problem arising from functional transcendence theory (cf.\,Ullmo-Yafaev \cite{UY11}).  It is given by Theorem \ref{ThmBialg} in the Introduction concerning varieties which are bi-algebraic with respect to the uniformization map $\pi: \Omega \to X_\Gamma :=\Omega/\Gamma$ (cf.\,Theorem \ref{ThmBialg} for the precise statement). Our proof of Theorem \ref{ThmBialg} yields a stronger statement.

\begin{thm}\label{ThmUnif}  
Let $\Omega \Subset \mathbb C^N$ be a bounded symmetric domain in its Harish-Chandra realization, and $Z \subset \Omega$ be an irreducible algebraic subset.  Suppose there exists a torsion-free discrete subgroup $\check\Gamma \subset {\rm Aut}(\Omega)$ such that $\check\Gamma$ stabilizes $Z$ and $Z/\check\Gamma$ is compact.  Then, $Z \subset \Omega$ is totally geodesic.
\end{thm}

In Theorem \ref{ThmBialg}, which generalizes the cocompact case of Ullmo-Yafaev \cite{UY11}, $\check Y := Z/\check\Gamma$ is assumed to be a subvariety on some projective quotient manifold $X_\Gamma := \Omega/\Gamma$ where $\Gamma \subset {\rm Aut}(\Omega)$ is a torsion-free and not necessarily arithmetic cocompact lattice.  In Theorem \ref{ThmUnif} by contrast there is no ambient projective manifold $X_\Gamma$.

\vskip 0.2cm
\noindent
{\it Deduction of Theorem \ref{ThmBialg} from Theorem \ref{ThmUnif}} \
In the notation of Theorem \ref{ThmBialg} let $\Gamma \subset {\rm Aut}(\Omega)$ be a torsion-free cocompact lattice and write $X_\Gamma := \Omega/\Gamma$, which is a projective manifold.  Let $\pi: \Omega \to X_\Gamma$ be the uniformization map, $Y \subset X_\Gamma$ be an irreducible subvariety, and $Z \subset \Omega$ be an irreducible component of $\pi^{-1}(Y)$.  Let $\check \Gamma \subset \Gamma$ be the subgroup given by $\check\Gamma := \big\{\gamma \in \Gamma: \gamma(Z) = Z\big\}$. Then $\check\Gamma$ acts as a torsion-free discrete group of automorphisms on $Z$, and, defining $\check Y := Z/\check\Gamma$, the canonical map $\alpha: \check Y \to X_\Gamma$ is a birational morphism onto $Y$, hence $\check Y$ is projective, and Theorem \ref{ThmBialg} follows from Theorem \ref{ThmUnif}. \qquad \qquad \qquad \qquad\qquad  \qquad\qquad \qquad\qquad \qquad \qquad \qquad\qquad \qquad\qquad \qquad\qquad \quad \ $\square$

\subsubsection{Strategy of proof of Theorem \ref{ThmUnif}}
When $\Gamma \subset \mbox{\rm Aut}_0(\Omega)$ is an arithmetic but not necessarily cocompact lattice, Theorem \ref{ThmBialg} was established by Ullmo-Yafaev \cite{UY11}.  Their proof makes use of a monodromy result of Andr\'e-Deligne (cf.\,\cite{An92}) which relies on Hodge Theory, for which arithmeticity of the lattices plays a crucial role. Our proof of Theorem \ref{ThmUnif}, which implies Theorem \ref{ThmBialg} as we have seen, will rely on the existence of K\"ahler-Einstein metrics on compact K\"ahler manifolds with ample canonical line bundle and the proof of the semisimplicity theorem for the identity component of a regular covering of such a manifold due to Nadel \cite{Na90}.

Our proof of Theorem \ref{ThmUnif} breaks up into several steps culminating in the use of Theorem \ref{ThmEquiv}. We will prove that $Z$ is nonsingular and that the K\"ahler manifold $(Z,ds_\Omega^2|_Z)$ is the image of some bounded symmetric domain by an equivariant holomorphic isometric embedding in order to be able to apply Theorem \ref{ThmEquiv} to conclude that $(Z,ds_\Omega^2|_Z) \subset (\Omega,ds_\Omega^2)$ is totally geodesic.

To start with let $H_0 \subset G_0$ be the identity component of the stabilizer subgroup of $Z$. It follows readily from the algebraicity of $Z \subset \Omega$ that $\dim_\mathbb R H_0 > 0$. We cannot show directly that $H_0$ acts transitively on $Z$.  In its place we show using methods of complex analysis that there exists a complex algebraic subgroup $H \subset G := \mbox{Aut}(X_c)$ which is at the same time a complexification of $H_0$ in $G$, such that the orbit $Hz_0$ for any $z_0 \in Z$ contains $Z$. (Recall that $\Omega \subset X_c = G/P$ is the Borel embedding.) In particular, $Z$ is nonsingular. These first steps of the argument remain valid when $Z/\check \Gamma$ is only assumed quasi-projective. 

Since $Z$ is nonsingular and the discrete group $\check \Gamma$ acts without fixed points on $Z$, it follows that $\check Y := Z/\check \Gamma$ is nonsingular. To proceed further we will make use of the compactness of $\check Y$, so that it is a projective manifold with ample canonical line bundle, implying the existence on $\check Y$ of a K\"ahler-Einstein metric of negative Ricci curvature (Aubin \cite{Au78}, Yau \cite{Ya78}).  By Nadel \cite{Na90}, which exploited the polystability of the holomorphic tangent bundle of $\check Y$ as a consequence of the existence of K\"ahler-Einstein metrics, we know that the identity component of ${\rm Aut}(Z)$ is semisimple and of the noncompact type. The same proof applies to show that $H_0 \subset G_0$ is a semisimple Lie subgroup of the noncompact type.
Let $S := H_0x \subset Z$ for some $x\in Z$.
By cohomological arguments, we will deduce that $\dim_{\mathbb R}(S)=\dim_{\mathbb R}(Z)$.
At the same time, this will also imply $S\cong H_0/L$ for some maximal compact subgroup $L$ of $H_0$ by dimension arguments.
As a consequence, $S=Z \subset \Omega$ is a Hermitian symmetric space of the noncompact type and Theorem \ref{ThmUnif} will follow from Theorem \ref{ThmEquiv}.

\subsubsection{Pseudo-homogeneity of algebraic subsets admitting quasi-projective quotients} 
We say that an irreducible algebraic subset $E \subset \Omega \subset X_c$ is pseudo-homogeneous to mean that it is an open subset in the complex topology of an orbit in $X_c$ under some complex algebraic subgroup of $G = {\rm Aut}_0(X_c)$.  We will prove that the algebraic subset $Z \subset \Omega$ in Theorem \ref{ThmUnif} is pseudo-homogeneous in this sense by means of methods of complex analysis, more precisely by means of Riemann extension theorem on bounded plurisubharmonic functions and the maximum principle on plurisubharmonic functions on compact complex spaces. 

It is convenient to introduce the Zariski topology on $\Omega$ and its algebraic subsets.  A subset $E \subset \Omega$ is Zariski closed if and only if it is an algebraic subset of $\Omega$.  For a Zariski closed subset $V \subset \Omega$, $V$ inherits the Zariski topology from $\Omega$ by restriction, and a subset $E \subset V$ is Zariski closed if and only if $E \subset \Omega$ is Zariski closed.   

In what follows we make use of gothic letters to denote real or complex Lie algebras of real or complex Lie groups in a self-evident manner. The Lie algebra of a real (resp.\,complex) Lie group will be identified with its tangent space (resp.\,holomorphic tangent space) at the identity element.

In order to convert the problem concerning discrete groups $\check\Gamma \subset {\rm Aut}(\Omega)$ which stabilize an algebraic subset $Z \subset \Omega$ to questions on Lie groups of holomorphic isometries, to start with we prove

\begin{pro}\label{pro5.20}
In the notation of Theorem \ref{ThmUnif}, there exists a positive-dimen-sional algebraic subgroup $H_0 \subset G_0$
such that $h(Z) = Z$ for any $h \in H_0$ and such that $H_0 \cap \check\Gamma$ is of finite index in $\check\Gamma$.
\end{pro}

\begin{proof}
Recall that by definition $Z$ is an irreducible component of $\widehat Z \cap \Omega$ for some irreducible subvariety $\widehat Z \subset X_c = G/P$, where  $G = {\rm Aut}_0(X_c)$ and $P \subset G$ is some parabolic subgroup. Define now $\mathscr H := \big\{h \in G: h(\widehat Z) = \widehat Z \big\}$. $\mathscr H \subset G$ is a subgroup defined by a set of algebraic equations on $G$, and as such it is a complex algebraic subgroup.  For any $\gamma \in \check\Gamma$ we have $\gamma(Z) = Z$, hence also $\gamma(\widehat Z) = \widehat Z$ by the identity theorem for holomorphic functions.  Therefore, $\check\Gamma \subset \mathscr H$. 

We claim that $\check\Gamma$ is an infinite group.  This is obvious when $\check Y$ is compact by the maximum principle. In general, let $\check Y \subset W$ be a projective compactification, and let $\sigma: W^\dagger \to W$ be a desingularization. Suppose $\check \Gamma$ is finite.  Then, any continuous bounded plurisubharmonic function $\varphi$ on $\Omega$, when restricted to $Z$, gives rise to a continuous bounded plurisubharmonic function $\psi$ on $\check Y^\dagger$ obtained by summing over the finite fibers of the uniformization map $\varpi: Z \to \check Y$ and pulling back to the nonsingular model $\check Y^\dagger = \sigma^{-1}(\check Y)$ by $\sigma: W^\dagger \to W$.  Clearly one can choose $\varphi$ so that $\psi$ is nonconstant.  On the other hand, by the Riemann extension theorem for bounded plurisubharmonic functions, $\psi$ extends to a plurisubharmonic function on the compact complex manifold $W^\dagger$ and must hence be constant by the maximum principle, a plain contradiction.

Define $H_0' :=\mathscr H \cap G_0 \subset G_0$.  Since $\check\Gamma \subset   
\mathscr H \cap G_0$ and the algebraic group $H_0'$ has at most a finite number of connected components, 
from ${\rm Card}(\check\Gamma) = \infty$ we conclude that $\dim(H_0') > 0$.  
Moreover, writing $H_0$ for the identity component of the algebraic group $H_0'$, $H_0\cap \check\Gamma$ is of finite index in $\check\Gamma$, and we have proven Proposition \ref{pro5.20}, as desired.
\end{proof}

\begin{pro}\label{pro5.21}
Let $H_0 \subset G_0$ be a connected real algebraic subgroup.  Then there exists a connected complex algebraic subgroup $H \subset G$ such that $T_e(H)$ agrees with $\frak h_0 \otimes_\mathbb R\mathbb C$.
\end{pro}

\begin{proof}
At the level of Lie algebras we have $\frak h_0 \subset \frak g_0$.  Define $\frak h := \frak h_0 \otimes_\mathbb R\mathbb C$ and write $r := \dim_\mathbb R H_0$.
Let $\mathcal H$ be the simply connected complex Lie group with Lie algebra $\cong \frak h$. Then, there exists a holomorphic homomorphism $\alpha: \mathcal H \to G$ with discrete kernel such that $d\alpha(T_e(\mathcal H)) = \frak h \subset \frak g$.  There exists thus a complex submanifold $U$ containing $e$ of some open subset $W \subset X$ such that $U$ is the image under $\alpha$ of some open neighborhood of $e \in \mathcal H$.  We will assume without loss of generality that $U$ is closed under taking inverses, i.e., $U = U^{-1}$. It remains to prove that $(U;e)$ is the germ of a complex algebraic subgroup $H \subset G$. Embed $G_0 \subset G$, $G_0 \subset \mathbb R^N$, $G \subset \mathbb R^N\otimes_\mathbb R\mathbb C = \mathbb C^N$, as the real form of a complex algebraic group.  $H_0 \subset G_0$ is defined as the common zero set of a finite-dimensional real vector subspace $I(H_0) \subset \mathbb R[x_1,\cdots,x_N]$.  Let $\Lambda: \mathbb R[x_1,\cdots,x_N] \to \mathbb R[z_1,\cdots,z_N]$ be the $\mathbb R$-algebra homomorphism defined from $\Lambda(x_i) = z_i$ for $1 \le i \le N$.  Define $E := \Lambda(I(H_0))\otimes_\mathbb R\mathbb C$, and denote by $V \subset G \subset \mathbb C^N$ the irreducible component containing $e$ of the common zero set of $E$.  From the definition we have $H_0 \subset V$.  From the implicit function theorem $V$ is smooth along $H_0$ and we have $\dim_\mathbb C V = r = \dim_\mathbb R H_0$, hence the germs of complex submanifolds $(V;e)$ and $(U;e)$ of $(G;e)$ are identical.  Therefore, the complex affine algebraic subvariety $V \subset G$ contains $U$ and it remains to check that (a) $V$ is closed under multiplication induced from $G$ and (b) any $x \in V$ is invertible in $V$.  Granting this, the proof is completed by setting $H = V$.

Define $F := \big\{y\in V: yV \subset V, y^{-1}V\subset V \big\}$.  Then, $F \subset V$ is defined as the common zero set of a set of complex polynomials and it is hence a complex algebraic subvariety. On the other hand, for any $x \in U$, $xU$ contains an open neighborhood of $x$ in $U$, hence $xV \subset V$, and similarly $x^{-1}V \subset V$, so that $U \subset V$ and hence $F = V$ by the identity theorem for holomorphic functions. In particular, $V$ is closed under multiplication, proving (a).  Moreover, for any $y \in V$, we have $yV \subset V$ and $y^{-1}V \subset V$, so that also $V \subset yV$, hence $yV = V$. Therefore, there exists $w \in V$ such that $yw = e$, hence also $wy = e$, so that any $y \in V$ is invertible, proving (b), as desired.   
\end{proof}

We call $H \subset G$ the complexification of $H_0$ inside $G$.

\begin{pro}\label{density}
Let $\Omega \Subset \mathbb C^N \subset X_c$ be a bounded symmetric domain in its Harish-Chandra realization and Borel embedding into $X_c = G/P$, the compact dual of $\Omega$, where $G$ is the identity component of $\text{\rm Aut}(X_c)$. Let $G_0$ be the identity component of  $\text{\rm Aut}(\Omega)$, $G_0 \subset G$ being a noncompact real form.  Let $Z \subset \Omega$ be an irreducible algebraic subset.  Suppose there exists a torsion-free discrete subgroup $\check \Gamma \subset {\rm Aut}(\Omega)$ such that $\check\Gamma$ stabilizes $Z$ and $\check Y = Z/\check\Gamma$ is quasi-projective.  Let $H_0 \subset G_0$ be the identity component of the $($positive-dimensional$)$ stabilizer subgroup of $Z$, and $H \subset G$ be the complexification of $H_0$ inside $G$. Then, $Z$ is an irreducible component of $Hx \cap \Omega$.
\end{pro}

\begin{proof}
Recall that by definition the irreducible algebraic subset $Z \subset \Omega$ is an irreducible component of $\widehat Z \cap \Omega$ for some irreducible projective algebraic subvariety $\widehat Z \subset X_c$. Consider the orbit $Hx \subset \widehat Z$ of $x \in Z$ under the complex algebraic group $H \subset G$.  Since $S = H_0x \subset Z$ and $Z \subset \Omega$ is a complex-analytic subvariety, we have the inclusion $(Hx;x) \subset (\widehat Z;x)$ of germs of subvarieties, hence $Hx \subset \widehat Z$. We prove first of all that $Hx \cap Z$ is dense in $Z$ with respect to the Zariski topology on $Z$.  Suppose otherwise.  Then there exists a Zariski closed subset $E \subsetneq Z$ such that $E \supset Hx \cap Z$.  There exists a projective algebraic subvariety $\widehat E$ such that $E$ is the union of a finite number of irreducible components of $\widehat E \cap \Omega$. Writing $N = \dim_\mathbb C(\Omega)$, let now $P(z_1,\cdots,z_N)$ be a polynomial in $N$ complex variables such that $P|_{\widehat E \cap \mathbb C^N}\equiv 0$ and such that $P|_{\widehat Z \cap \mathbb C^N} \not\equiv 0$.  

Next, using $P \in \mathbb C[z_1,\cdots,z_N]$ we will derive a contradiction by means of the maximum principle. Define a real function $\Phi: \Omega \to \mathbb R$ by $\Phi(z) = {\rm sup}\{|P(\gamma z)|: \gamma \in \check \Gamma\}$.  Write $f_\gamma(z) := P(\gamma z)$ for $z \in \Omega$. Regarding $\{f_\gamma\}_{\gamma \in \check \Gamma}$ as a family of holomorphic functions on $\Omega$, we have the uniform bound $|f_\gamma(z)| \le \text{\rm sup}\{|P(z)|: z \in \overline{\Omega}\} < \infty$.  From Cauchy estimates, the family of holomorphic functions $\{f_\gamma\}_{\gamma \in \check \Gamma}$ is uniformly Lipschitz on any compact subset of $\Omega$ and it follows that $\Phi$ is uniformly Lipschitz on any compact subset of $\Omega$.  In particular, $\Phi: \Omega \to \mathbb R$ is a continuous bounded plurisubharmonic function on $\Omega$.   Restricting to $Z$ we have $\Phi(z) = 0$ whenever $z \in Hx \cap Z \subset E$ and $\Phi(z_0) \neq 0$ for some $z_0 \in {\rm Reg}(Z) - E$.  By the definition of $\Phi$ we have $\Phi(\gamma z) = \Phi(z)$ for any $\gamma \in \check \Gamma$, hence we obtain by descent a nonconstant bounded plurisubharmonic function $\varphi: W-A \to \mathbb R$. Denote by $\check Y \subset W$ a projective compactification, and define $A := {\rm Sing}(W) \cup (W-\check Y)$. 
Let $\sigma: W^\sharp \to W$ be a desingularization of $W$ and
define $\varphi^\sharp: W^\sharp - \sigma^{-1}(A) \to \mathbb R$ by $\varphi^\sharp = \varphi\circ\sigma$.  Then, $\varphi$ is a nonconstant bounded plurisubharmonic function defined on the nonempty Zariski open subset $W^\sharp - \sigma^{-1}(A) \subset W^\sharp$.  By the Riemann extension theorem for bounded plurisubharmonic functions, $\varphi^\sharp$ extends to a plurisubharmonic function, to be denoted by the same symbol, on the projective manifold $W^\sharp$.  By the maximum principle for plurisubharmonic functions $\varphi^\sharp$ must necessarily be a constant, a plain contradiction.    

Since $H \subset G$ acts algebraically on $X_c$, the Zariski closure of $Hx$ in $\widehat Z \subset X_c$ is the same as its topological closure, and we conclude from the above that $\overline{Hx} \cap Z = Z$. Suppose now $Hx \cap Z \subsetneq Z$ and let $y \in Z-Hx$. The same argument applies to $y$ (in place of $x$) and we have $\overline{Hy} \cap Z = Z$, contradicting with the fact that $Hx$ and $Hy$ are distinct and hence disjoint orbits.  We conclude that $Hx \cap Z = Z$ for any $x \in Z$, i.e., $Z \subset Hx$ for any $x \in Z$. Hence, the germs of subvarieties $(Z;x)$ and $(Hx;x)$ at $x \in Z$ are identical and $Z$ is an irreducible component of $Hx\cap\Omega$.   
\end{proof}

As a direct consequence of Proposition \ref{density}, $Z \subset \Omega$ is a complex submanifold because $Z$ is an irreducible component of $Hx\cap \Omega$ for $x\in Z$. To anticipate the use of this assertion, we state this result as a corollary in the following.

\begin{cor}\label{Cor:Smoothness_of_Z}
Let $\Omega \Subset \mathbb C^N$ be a bounded symmetric domain in its Harish-Chandra realization, and $Z \subset \Omega$ be an irreducible algebraic subset.  Suppose there exists a torsion-free discrete subgroup $\check\Gamma \subset {\rm Aut}(\Omega)$ such that $\check\Gamma$ stabilizes $Z$ and $Z/\check\Gamma$ is quasi-projective.
Then, $Z \subset \Omega$ is a complex submanifold.
\end{cor}
 
\subsubsection{Preliminaries from Riemannian geometry on bounded symmetric domains}
The following lemma in Riemannian geometry is well-known but we include a proof for easy reference. Note that  for any Riemannian symmetric space $(M,h)$ of the semisimple and noncompact type, the underlying manifold $M$ is real-analytic, and $h$ is a real-analytic metric. 

\begin{lem}\label{lem5.23}
Let $(M,h)$ be a Riemannian symmetric manifold $(M,h)$ of the semisimple and noncompact type, and $\gamma$ be an isometry of $(M,h)$. Then, an irreducible component $\Sigma(\gamma)$ of the fixed point set $\text{\rm Fix}(\gamma)$ of any isometry $\gamma$ of $(M,h)$ is necessarily a totally geodesic submanifold.
\end{lem} 

\begin{proof}
Fix$(\gamma) \subset M$ is a real-analytic subvariety. Let $\Sigma(\gamma) \subset {\rm Fix}(\gamma)$ be any irreducible component, and $x \in \Sigma(\gamma)$ be a smooth point.  Since $\gamma(x) = x$ and $\gamma|_{\Sigma(\gamma)} = \text{\rm id}_{\Sigma(\gamma)}$, we have $d\gamma(\eta) = \eta$ for any $\eta \in T_x(\Sigma(\gamma)$). Let $\ell \subset M$ be a geodesic passing through $x$ such that $T_x(\ell) \subset T_x(\Sigma(\gamma))$. From $d\gamma(\eta) = \eta$ for $\eta \in T_x(\ell)$ we conclude that $\gamma(y) = y$ for any $y \in \ell$ by the uniqueness of parametrized geodesics with fixed initial point and fixed initial velocity.
Hence, $\ell \subset \Sigma(\gamma)$. It follows that $\sigma(\eta,\eta) = 0$
for the second fundamental from $\sigma$ of $\Sigma(\gamma) \subset M$ at $x$, and by polarization we have $\sigma \equiv 0$ on ${\rm Reg}(\Sigma(x))$. Finally, being the image of a vector subspace $V \subset T_x(M)$ under the exponential map ${\bf exp}_x: T_x(M) \to M$ at a nonsingular point $x \in \Sigma(\gamma)$, the totally geodesic subset $\Sigma(\gamma) \subset M$ of the Cartan-Hadamard manifold $(M,h)$ is necessarily nonsingular, and it follows that $\Sigma(\gamma) \subset M$ is a totally geodesic submanifold, as desired.
\end{proof}

We have the following lemma on the stabilizer subgroup of $Z \subset \Omega$. 

\begin{lem}
Let $Z \subset \Omega$ be an algebraic subset, and let $\Omega' \subset \Omega$ be the smallest totally geodesic complex submanifold containing $Z$.
Suppose $\gamma \in \text{\rm Aut}(\Omega')$ such that $\gamma|_Z = \text{\rm id}_Z$.  Then, $\gamma = \text{\rm id}_{\Omega'}$.
\end{lem}

\begin{proof}
By hypothesis $\gamma \in \text{\rm Aut}(\Omega')$ is such that $\gamma|_Z  = \text{\rm id}_Z$. Note that $(\Omega',ds_\Omega^2|_{\Omega'})$ is a Hermitian symmetric space of the noncompact type.  Let now $\Sigma(\gamma)$ be an irreducible component of Fix$(\gamma)$ such that $Z \subset \Sigma(\gamma)$. 
Since $\gamma$ is a holomorphic automorphism on $\Omega'$, $\Sigma(\gamma) \subset \Omega'$ is a complex-analytic subvariety.
By Lemma \ref{lem5.23}, $\Sigma(\gamma) \subset \Omega'$ is a totally geodesic complex submanifold.  From the minimality of $\Omega' \subset \Omega$ among all totally geodesic complex submanifolds containing $Z$, we have $\Sigma(\gamma) = \Omega'$. In other words, $\gamma = \text{\rm id}_{\Omega'}$, as desired.
\end{proof}    

From now on, replacing $Z \subset \Omega$ by $Z \subset \Omega'$ if necessary we assume without loss of generality that $\Omega$ is the smallest bounded symmetric domain containing $Z$ so that the natural homomorphism 
$\Phi: H_0 \to \text{\rm Aut}(Z,ds_\Omega^2|_Z)$ defined by $\Phi(\gamma) = \gamma|_Z$ is injective. 

\subsubsection{Nadel's semisimplicity theorem on automorphism groups of universal covers of projective manifolds with ample canonical line bundle}
To prove that $Z \subset \Omega$ is totally geodesic it would suffice to prove that $(Z,ds^2_\Omega|_Z)$ is abstractly biholomorphically isometric to a Hermitian symmetric manifold of the semisimple and noncompact type in such a way that Aut$_0(Z,ds^2_\Omega|_Z)$ embeds equivariantly into $G_0 = \text{\rm Aut}_0(\Omega,ds^2_\Omega)$, from which the total geodesy of $Z \subset \Omega$ will follow from Theorem \ref{ThmEquiv}. We have a positive-dimensional algebraic subgroup $H_0 \subset G_0$ acting on $Z$, but to proceed further there are two difficulties. First of all, we are short of proving that $H_0$ acts transitively on $Z$.  Secondly, even when we know that $H_0$ acts transitively on $\Omega$ it is not clear that the inclusion $H_0 \subset G_0$ extends to an equivariant homomorphism Aut$_0(Z,ds^2_\Omega|_Z) \hookrightarrow G_0$. 

While the preparation towards proving Theorem \ref{ThmUnif} works so far equally well when $\check Y = Z/\check\Gamma$ is quasi-projective, from now on we return to the situation where $\check Y = Z/\check\Gamma$ is compact as in the hypothesis of the theorem. For compact K\"ahler manifolds we have the following result of Nadel \cite{Na90}. 

\begin{thm}
Let $X$ be a compact K\"ahler manifold with ample canonical line bundle, and denote by $\pi:\widetilde X \to X$ the uniformization map.  Then, $\text{\rm Aut}_0(\widetilde X)$ is a semisimple Lie group of the noncompact type.
\end{thm} 

Here a semisimple Lie group $Q$ is said to be of the noncompact type if and only if in the direct product decomposition of the universal covering group $\widetilde{Q}$ of $Q$ there are no compact factors.

We have proven that $Z \subset Hx$ for some complex algebraic subgroup $H \subset G = \text{\rm Aut}_0(X_c)$, so that in particular $Z \subset \Omega$ is nonsingular, and $\check Y := Z/\check \Gamma$ is a projective manifold.
The K\"ahler metric $ds^2_\Omega|_Z$ is of nonpositive bisectional curvature and strictly negative holomorphic sectional curvature by the monotonicity on bisectional curvatures resulting from Gauss' equation, hence $\check Y$ inherits a K\"ahler metric of strictly negative Ricci curvature, proving that $\check Y$ has ample canonical line bundle.  Hence, Nadel \cite{Na90} applies to $\check Y$.  However, we will need a modified version as given below, which follows immediately from the proof in \cite{Na90}, since we are dealing with holomorphic isometries of $(Z,ds^2_\Omega|_Z)$ which are restrictions of holomorphic automorphisms of $\Omega$ which stabilize $Z$.
Recall that we have assumed that there is no proper totally geodesic complex submanifold $\Omega' \subset \Omega$ which contains $Z$. We have

\begin{pro}\label{pro5.26}
Suppose there exists a torsion-free discrete subgroup $\check\Gamma\subset \mathrm{Aut}(\Omega)$ such that $\check\Gamma$ stabilizes $Z$ and $Z/\check \Gamma$ is compact. Let $H_0\subset G_0:=\mathrm{Aut}_0(\Omega)$ be the identity component of the subgroup of $G_0$ which stabilizes $Z$. Then, $H_0\subset G_0$ is a semisimple Lie group without compact factors.
\end{pro}

\subsubsection{Proof of Theorem \ref{ThmUnif}}
By Corollary \ref{Cor:Smoothness_of_Z}, $Z \subset \Omega \cong G_0/K$ is a complex submanifold.
Recall $\hat\Gamma := H_0\cap \check\Gamma \subset \check\Gamma$ is a subgroup of finite index by Proposition \ref{pro5.20}.
Thus, by the assumption that $Z/\check\Gamma$ is compact, $\hat Y:=Z/\hat\Gamma$ is a compact complex manifold.
Let $x\in Z$ and $S:=H_0 x\subset Z$.
For any $z\in \Omega$, write $K_z:=\{g\in G_0:g(z)=z\}$ and $(H_0)_z=\{h\in H_0: h(z)=z\}$.
Since $(H_0)_z\subset K_z$ and $K_z$ is a compact group, the isotropy subgroup $(H_0)_z$ is compact, where $z\in \Omega$.
Thus, there is a maximal compact subgroup $L$ of $H_0$ containing $(H_0)_x$, i.e., $(H_0)_x\subseteq L \subseteq H_0$.
By Cartan's fixed point theorem, $L$ has a fixed point $y\in \Omega$ so that $L\subset K_y$.
In particular, $L\subseteq (H_0)_y \subset H_0$ and thus $L=(H_0)_y$ because $L\subset H_0$ is maximal compact and $(H_0)_y$ is compact.
Since $H_0$ is a connected real algebraic group, $H_0/L \cong \mathbb R^n$ is homeomorphic to $\mathbb R^n$ for some integer $n\ge 0$ (see Borel \cite[p.\,124]{Bo98}).

Now, $S_{\hat\Gamma}:=\hat\Gamma \backslash H_0/L$ is a $K(\hat\Gamma,1)$ since $H_0/L\cong \mathbb R^n$ is contractible and $H_0/L \to S_{\hat\Gamma}$ is the universal covering map.
For the notion of $K(\pi,n)$'s, we refer the readers to Whitehead \cite[Chapter V, p.\,244]{Wh78}.
Recall the following fundamental theorem in Algebraic Topology.
\begin{thm}[$\text{cf.\,Whitehead \cite[(4.3) Theorem, p.\,225]{Wh78}}$]\label{thm_homotopy}
Let $N$ be a connected CW complex and $M$ be a $K(\pi,1)$.
Then, the correspondence $f\mapsto f_*$ induces the one-to-one correspondence between $[N,x_0;M,y_0]$ and $\mathrm{Hom}(\pi_1(N),\pi_1(M))$, where $[N,x_0;M,y_0]$ denotes the set of homotopy classes of continuous maps from $N$ to $M$ which map $x_0\in N$ to $y_0\in M$, and $\mathrm{Hom}(\pi_1(N),\pi_1(M))$ denotes the set of all group homomorphisms from $\pi_1(N)$ to $\pi_1(M)$.
\end{thm}

The inclusion map $\iota_{\hat Y}: \hat Y  \hookrightarrow \Omega/\hat\Gamma=:X_{\hat\Gamma}$ induces a group homomorphism 
\[ \Phi:=(\iota_{\hat Y})_*:\pi_1(\hat Y)\to \pi_1(X_{\hat\Gamma}). \]
Let $g:S_{\hat\Gamma}=\hat\Gamma\backslash H_0/L \hookrightarrow X_{\hat\Gamma}:=\Omega/\hat\Gamma$ be the inclusion map which is induced from the natural inclusion $H_0/L \hookrightarrow \Omega$.
Recall that $H_0/L\cong \mathbb R^n$ is contractible as shown above, hence $H_0/L$ is simply connected. On the other hand, $\Omega$ is a Cartan-Hadamard manifold, hence contractible, {\it a fortiori\/} simply connected, and thus $X_{\hat\Gamma}=\Omega/\hat\Gamma$ is a $K(\hat\Gamma,1)$. Therefore, the fundamental group $\pi_1(S_{\hat\Gamma})$ (resp.\,$\pi_1(X_{\hat\Gamma})$) of $S_{\hat\Gamma}$ (resp.\,$X_{\hat\Gamma}$) can be naturally identified with the group ${\rm Cov}_{S_{\hat\Gamma}}\cong \hat\Gamma$ (resp.\,${\rm Cov}_{X_{\hat\Gamma}}\cong \hat\Gamma$) of all covering transformations of the universal covering map $H_0/L\to S_{\hat\Gamma}$ (resp.\,$\Omega\to X_{\hat\Gamma}$),  and it follows readily that the induced group homomorphism $g_*:\pi_1(S_{\hat\Gamma})\to \pi_1(X_{\hat\Gamma})$ is an isomorphism.
Define the group homomorphism $(g_*)^{-1}\circ \Phi:\pi_1(\hat Y) \to \pi_1(S_{\hat\Gamma})$.
Then, by Theorem \ref{thm_homotopy} we have a continuous map $f:\hat Y\to S_{\hat\Gamma}$ such that $f_* = (g_*)^{-1}\circ \Phi$.
Now, the composition
\[ g\circ f: \hat Y \to X_{\hat\Gamma} \]
is a continuous map inducing the homomorphism $(g\circ f)_*=g_*\circ f_*=\Phi$.
Thus, Theorem \ref{thm_homotopy} asserts that $g\circ f:\hat Y \to X_{\hat\Gamma}$ and $\iota_{\hat Y}:\hat Y \to X_{\hat\Gamma}$ are homotopic to each other.
Note that the inclusion map $g:S_{\hat\Gamma}\hookrightarrow X_{\hat\Gamma}$ is a smooth map.
Since $\hat Y$ and $S_{\hat\Gamma}$ are smooth manifolds, by Whitney's approximation theorem (cf.\,Lee \cite[Theorem 6.26, p.\,141]{Lee13})
we may choose $f$ so that $f:\hat Y\to S_{\hat\Gamma}$ is a smooth map.
As a consequence, $g\circ f:\hat Y \to X_{\hat\Gamma}$ and $\iota_{\hat Y}:\hat Y \to X_{\hat\Gamma}$ are homotopic smooth maps.
By the homotopy invariance of cohomology (cf.\,Lee \cite[Proposition 17.10, p.\,445]{Lee13}) we have the same pullback maps $(g\circ f)^*=\iota_{\hat Y}^*:H^p_{\mathrm{dR}}(X_{\hat\Gamma})\to H^p_{\mathrm{dR}}(\hat Y)$ for all $p$. 
With these results, we are ready to finish the proof of Theorem \ref{ThmUnif}, as follows.

\begin{proof}[Proof of Theorem \ref{ThmUnif}]
We may suppose $\dim_{\mathbb C}(Z)\ge 1$; otherwise, the statement is trivial.
By the previous results, $\hat Y:=Z/\hat \Gamma$ is a compact complex manifold.
Choose a point $x\in Z$ and define $S:=H_0x\subset Z$.
The Bergman metric $ds_\Omega^2$ induces a K\"ahler metric $g_{X_{\hat\Gamma}}$ on $X_{\hat\Gamma}:=\Omega/\hat\Gamma$.
Write $\hat\omega$ for the K\"ahler form of $(X_{\hat\Gamma},g_{X_{\hat\Gamma}})$. We have the compact K\"ahler submanifold $(\hat Y,g_{X_{\hat\Gamma}}|_{\hat Y})$ of $(X_{\hat\Gamma},g_{X_{\hat\Gamma}})$.
Writing $s:=\dim_{\mathbb C}(\hat Y)$, we have $\dim_{\mathbb C}(Z)=s$ and
\[ \eta_{\hat Y}:=\iota_{\hat Y}^*{\hat \omega^s\over s!}
= (g\circ f)^*{\hat \omega^s\over s!} + d\eta_0 
= {1\over s!} f^* \left(g^*\hat\omega^s\right) + d\eta_0\]
for some $(2s-1)$-form $\eta_0$ on $\hat Y$.
Moreover, $\eta_{\hat Y}$ is the volume form of the compact K\"ahler manifold $(\hat Y,g_{X_{\hat\Gamma}}|_{\hat Y})$.
Note that $\hat\omega^s$ is a differential $2s$-form on $X_{\hat\Gamma}$.
Suppose $\dim_{\mathbb R}(S_{\hat \Gamma})<2s$.
Then, we have $g^*\hat\omega^s=0$ and thus
\[ \eta_{\hat Y}
= {1\over s!}f^*(g^*\hat\omega^s) + d\eta_0 
= d\eta_0
\]
globally on $\hat Y$.
Since $\hat Y$ is a compact manifold, we have
\[ \mathrm{Vol}(\hat Y) = \int_{\hat Y} \eta_{\hat Y}
= \int_{\hat Y} d\eta_0
= 0 \]
by Stokes' Theorem, a plain contradiction.
Thus, 
\[ \dim_{\mathbb R}(H_0 x)\ge \dim_{\mathbb R}(H_0/L)=\dim_{\mathbb R}(S_{\hat \Gamma})\ge 2s\]
because $(H_0)_x\subset L$ and $H_0x\cong H_0/(H_0)_x$.
On the other hand, $\dim_{\mathbb R}(H_0 x)$ $\le$ $\dim_{\mathbb R}(Z)=2s$.
Hence, we have
\[ \dim_{\mathbb R}(H_0 x) = \dim_{\mathbb R}(H_0/L)=\dim_{\mathbb R}(Z)=2s. \]
Since $S:=H_0x\subseteq Z$ is a smooth embedded submanifold that is closed in $Z$, and $\dim_{\mathbb R}(S) =\dim_{\mathbb R}(Z)=2s$, $S=Z\subset \Omega$ is a complex submanifold.
Moreover, it follows from $H_0x\cong H_0/(H_0)_x$ and $\dim_{\mathbb R}(H_0 x)$ $=$ $\dim_{\mathbb R}(H_0/L)$ that
\[ \dim_{\mathbb R}(L)=\dim_{\mathbb R}((H_0)_x). \]
Note that $H_0$ is connected so that the maximal compact subgroup $L \subset H_0$ is connected.
Then, we have $(H_0)_x=L$ and thus $S=H_0x\cong H_0/(H_0)_x=H_0/L$.
Therefore, $Z=S\cong H_0/L$ and $H_0$ acts transitively on $Z$.

Since $\hat Y=Z/\hat \Gamma$ is compact, by Proposition \ref{pro5.26} $H_0$ is semisimple of the noncompact type, and thus $H_0/L$ is a Riemannian symmetric space of the semisimple and noncompact type.
Hence, the complex manifold $Z=H_0x\cong H_0/L$ is indeed a Hermitian symmetric space of the semisimple and noncompact type.
By Theorem \ref{ThmEquiv}, $S = Z \subset \Omega$ is totally geodesic, as desired.
\end{proof}

\section{Appendix}\label{App}
In the proof of Lemma \ref{LemTensorV} we made use of an inequality obtained by Mercer \cite[Proposition 2.4]{Me93}. 
In order to make the proof of Lemma \ref{LemTensorV} self-contained, we give a proof of the inequality for bounded symmetric domains in which we only make use of the Polydisk Theorem, as follows.
\begin{pro}
Let $\Omega\Subset \mathbb C^N$ be a bounded symmetric domain of rank $r$ and identify $\Omega\cong G_0/K$, where $K$ is the isotropy subgroup of $G_0:=\mathrm{Aut}_0(\Omega)$ at ${\bf 0}$.
Denote by $d_D(\cdot,\cdot)$ the Kobayashi pseudo-distance of any bounded symmetric domain $D\Subset\mathbb C^n$ so that $d_\Delta(0,\zeta)=\log{1+|\zeta|\over 1-|\zeta|}$ for any $\zeta\in \Delta$.
Then, for any point $z\in \Omega$ we have 
\[ d_\Omega({\bf 0},z) \ge -\log\delta(z,\partial\Omega), \]
where $\delta(x,\partial D)$ denotes the Euclidean distance from $x\in D$ to the boundary $\partial D$ of any bounded domain $D\Subset \mathbb C^n$.
\end{pro}
\begin{proof}
Note that there is $b\in \partial\Omega$ such that $\delta(z,\partial\Omega)=\lVert z-b\rVert_{\mathbb C^N}$. Here $\lVert v \rVert_{\mathbb C^n}$ denotes the complex Euclidean norm of any vector $v\in \mathbb C^n$.
In terms of the Harish-Chandra coordinates $(w_1,\ldots,w_N)$ on $\Omega$ we have the maximal polydisk $\Pi=\Delta^r \times\{{\bf 0}\}\subset \Omega$ and a holomorphic map $\pi:\Omega\to \mathbb C^r$ defined by $\pi(w_1,\ldots,w_N):=(w_1,\ldots,w_r)$ such that $\pi$ maps $\Omega$ onto the $r$-disks $\Delta^r$ (cf.\,Lemma 2.2.2 in Mok-Ng \cite{MN12}).
Up to the $K$-action on $\Omega$, we may assume that $b\in \partial\Pi$, i.e., $b=(b_1,\ldots,b_r,{\bf 0})\in \partial\Delta^r \times \{{\bf 0}\}=\partial\Pi$, because any $\gamma\in K$ is a unitary transformation on $\mathbb C^N$ and $d_\Omega(\cdot,\cdot)$ is invariant under the $G_0$-action on $\Omega$.
Write $b'=(b_1,\ldots,b_r)\in  \partial\Delta^r$, $z=(z_1,\ldots,z_N)$ and $z'=\pi(z)=(z_1,\ldots,z_r)\in \Delta^r$.
Note that there exists $k$, $1\le k\le r$, such that $\delta(z',\partial\Delta^r)=1-|z_k|$.
Then, we have
\[ \begin{split}
d_\Omega({\bf 0},z) 
&\ge d_{\Delta^r}({\bf 0},z')
= \max\left\{ \log {1+|z_j|\over 1-|z_j|}: 1\le j\le r\right\}\\
&\ge \log {1+|z_k|\over 1-|z_k|} \ge -\log (1-|z_k|)
= -\log \delta(z',\partial\Delta^r). 
\end{split}\]
On the other hand, since $b'\in \partial \Delta^r$ we have
$\delta(z',\partial\Delta^r)
\le \lVert z' - b'\rVert_{\mathbb C^r}
\le \lVert z - b\rVert_{\mathbb C^N}
=\delta(z,\partial\Omega)$.
Hence, we have $d_\Omega({\bf 0},z) \ge -\log \delta(z',\partial\Delta^r)\ge -\log\delta(z,\partial\Omega)$, as desired.
\end{proof}

\bigskip\noindent
Shan Tai Chan, Department of Mathematics, The University of Hong Kong, Pokfulam Road, Hong Kong
(E-mail: mastchan@hku.hk)\\
Ngaiming Mok, Department of Mathematics, The University of Hong Kong, Pokfulam Road, Hong Kong
(E-mail: nmok@hku.hk)
\end{document}